\documentclass[10pt]{amsart}
\usepackage{amsmath,amssymb,latexsym,hyperref,url,graphicx,mathrsfs,enumerate,tikz,microtype,enumitem}
\usepackage[english]{babel}
\usepackage[all]{xy}
\usetikzlibrary{calc}
\usetikzlibrary{arrows}
\usetikzlibrary{positioning}

\graphicspath{{figures/}}

\newtheorem{theorem}{Theorem}

\newtheorem{lemma}{Lemma}[section]
\newtheorem{proposition}[lemma]{Proposition}
\newtheorem{corollary}[lemma]{Corollary}
\theoremstyle{definition}
\newtheorem{definition}[lemma]{Definition}
\theoremstyle{remark}
\newtheorem{remark}[lemma]{Remark}
\newtheorem{example}[lemma]{Example}

\numberwithin{equation}{section}

\newcommand{\C}{\mathbb{C}}
\newcommand{\R}{\mathbb{R}}
\newcommand{\Q}{\mathbb{Q}}
\newcommand{\Z}{\mathbb{Z}}

\newcommand{\K}{\mathbb{K}}

\newcommand{\und}[1]{\underline{#1}}
\newcommand{\ov}[1]{\overline{#1}}
\newcommand{\mc}[1]{\mathcal{#1}}

\begin{document}

\title{Geometrical models for a class of reducible Pisot substitutions}
\author[B.~Loridant]{Benoit Loridant}
\address[B.~Loridant]{Lehrstuhl Mathematik und Statistik, Montanuniversit\"at Leoben, Franz Josef Stra{\ss}e 18, 8700 Leoben, Austria}
\email{benoit.loridant@unileoben.ac.at}
\author[M.~Minervino]{Milton Minervino}
\address[M.~Minervino]{Institut de Math\'ematiques de Marseille (UMR 7373), Campus de Luminy, Case 907,
13288 Marseille Cedex 9, France}
\email{milton.minervino@univ-amu.fr}
\subjclass[2010]{28A80, 52C23, 37B10}
\keywords{Substitutive systems, Rauzy fractals, tilings, stepped surfaces}


\thanks{This research is supported by the ANR/FWF project ``FAN -- Fractals and Numeration'' (ANR-12-IS01-0002, FWF grant I1136) of the French National Agency for Research (ANR) and the Austrian Science Fund (FWF) and by the JSPS/FWF project I3346 ``Topology of planar and higher dimensional self-replicating tiles'' of the Japan Society for the Promotion of Science (JSPS) and the FWF}
\date{\today}

\begin{abstract}
We set up a geometrical theory for the study of the dynamics of reducible Pisot substitutions. It is based on certain Rauzy fractals generated by duals of higher dimensional extensions of substitutions. We obtain under certain hypotheses geometric representations of stepped surfaces and related polygonal tilings, as well as self-replicating and periodic tilings made of Rauzy fractals. We apply our theory to one-parameter family of substitutions. For this family, we analyze and interpret in a new combinatorial way the codings of a domain exchange defined on the associated fractal domains. We deduce that the symbolic dynamical systems associated with this family of substitutions behave dynamically as first returns of toral translations.
\end{abstract}

\maketitle

\section{Introduction}

A {\it substitution} is a map that takes letters of some finite alphabet to finite words in that alphabet. It is {\it Pisot} if the dominant eigenvalue of its incidence matrix $M_\sigma$ is a Pisot number. The difference between {\it irreducible} and {\it reducible} is whether the characteristic polynomial of $M_\sigma$ is irreducible over $\mathbb{Q}$ or not. For more detailed definitions we refer to Section~\ref{prel}.

The {\it substitutive system} $(X_\sigma,S)$ is the action of the shift on the set of bi-infinite words all of whose finite words are factors of some iterations of the substitution on some letter. If the substitution is primitive then $(X_\sigma,S)$ is minimal and uniquely ergodic. One of the main aims is to understand the spectral behavior of such a system. The Pisot hypothesis guarantees the existence of a non-trivial Kronecker factor and is crucial in Rauzy's work \cite{Rau82}, where the geometrical theory for these systems based on {\it Rauzy fractals} was initiated. 
In this work it was shown that the substitutive system generated by the Tribonacci substitution is (measure-theoretically) a translation on a two-dimensional torus. The key point was to interpret the shift as a domain exchange on a compact self-similar domain of $\mathbb{C}$, later called Rauzy fractal. Furthermore, this fractal domain tiles $\mathbb{C}$ periodically. This geometrical construction was generalized in \cite{AI}, where it is shown that the substitutive system associated with any unimodular irreducible Pisot substitution satisfying a combinatorial hypothesis, called strong coincidence condition, is measurably conjugate to a domain exchange on the Rauzy fractal. 
The problem of understanding whether $(X_\sigma,S)$ has pure discrete spectrum, or equivalently if the Rauzy fractals tile periodically their representation space, is known as {\it Pisot conjecture} and is still open for irreducible Pisot substitutions (see \cite{ABBLS} for a survey and~\cite{Bar,Barge0000} for recent results).

\subsection*{The reducible settings}

The reducible framework is much more enigmatic and less studied in comparison with
the irreducible one.
One of the main reasons is the lack of some important tools and the existence of some significant differences.

A reducible Pisot substitution acts on an alphabet with $n$ symbols with $n$ greater than the degree $d$ of its Pisot number. The space $\mathbb{R}^n$ splits into an $M_\sigma$-invariant hyperbolic space, with a one-dimensional expanding and a $(d-1)$-dimensional contracting spaces, and a neutral space, that in this work we consider non-hyperbolic. One complication is to understand the influence of this neutral space in the dynamics and geometry of the substitution.

{\it Stepped surfaces} play an important role in the construction of tilings by Rauzy fractals. They were first defined in \cite{Rev} and used as arithmetic discrete models for hyperplanes e.g.~in \cite{IO:93,IO:94}. The concept of stepped surface plays a central role in \cite{AI} in the irreducible substitution context, where it is defined as the set of nearest colored integer points above the contracting space $\mathbb{K}_c$ of the substitution $\sigma$:
\begin{equation} \label{eq:oldstep}
\mathcal{S}=\{(\mathbf{x},a) \in \mathbb{Z}^n\times\mathcal{A} :  \pi_e(\mathbf{x}) \in [0,\pi_e(\mathbf{e}_a))  \}.  
\end{equation} 
This is a cut-and-project set, in particular a model set, with window in the expanding line. To any colored integer point one can associate a face of an hypercube of $\mathbb{R}^n$ of a certain type. The resulting union of faces provides a discrete approximation of the contracting space and is called a geometrical representation of the stepped surface. Its projection onto $\mathbb{K}_c$ is a polygonal tiling. If we replace these polygons by the Rauzy fractals we obtain generally a self-replicating multiple tiling. The main aim is to determine whether this collection forms a tiling. It was shown in \cite{IR06} that this is equivalent to having a periodic tiling by Rauzy fractals related to a domain exchange transformation and one related to a Markov partition of the toral automorphism $M_\sigma$.
Rauzy fractals are intimately connected with stepped surfaces since they are defined as attractors of a graph-directed iterated function system governed by the \emph{dual substitution} $\mathbf{E}_1^*(\sigma)$ and this dual substitution acts on elements of the stepped surface. 
Topological properties of Rauzy fractals generated by Pisot substitutions were studied in \cite{ST09}.

As pointed out in \cite{EIR}, the existence of a geometrical representation of a stepped surface in the reducible case is unclear. In this paper, the authors defined an abstract stepped surface similarly as in \eqref{eq:oldstep} as set of nearest colored points and they showed its invariance under the dual substitution $\mathbf{E}_1^*(\sigma)$. However no concrete geometrical realization was given.
An ad hoc construction for a geometrical stepped surface was given in~\cite{EI} for the reducible substitution associated with the minimal Pisot number, also known as Hokkaido substitution.
 
In \cite{EIR} a self-replicating collection made of Rauzy fractals and one related to the Markov partition for the toral automorphism of the substitution were studied. 
However, they showed that there exist reducible Rauzy fractals which cannot tile periodically. This is a significant difference with respect to the irreducible setting.  
For the Hokkaido substitution it was shown in \cite{EI} that an extended domain of the Rauzy fractal tiles indeed periodically. This can be explained with the results of \cite[Proposition~8.5]{BBK}: for a wide class of $\beta$-substitutions, the domain exchange on the Rauzy fractal is shown to be the first return of a minimal toral translation on it. The extended fundamental domain is obtained by taking into account the original Rauzy fractal together with the pieces prior to their first return. 

We mention that Rauzy fractals have been also defined in \cite{RWY14} in terms of a dual iterated function system of an algebraic graph-directed iterated function system, which provides a unified and simple framework for Rauzy fractals.

Note that the Pisot conjecture is not true for reducible Pisot substitutions, see e.g.~the Thue-Morse substitution and \cite[Example~5.3]{BBK}. On the other hand, all Pisot $\beta$-substitutions have tiling dynamical systems with pure discrete spectrum \cite{Bar,Barge0000}. However this was proven for the $\mathbb{R}$-action of translation on the convex hull of the tiling of the line induced by the substitution. Our aim is to understand the $\mathbb{Z}$-action $(X_\sigma,S)$ and, as far as we know, very little is known in the reducible setting. Notice that in the irreducible setting pure discrete spectrum of the $\mathbb{Z}$-action is equivalent to pure discrete spectrum of the $\mathbb{R}$-action by a result of \cite{CS03}. However, this equivalence does not hold for reducible substitutions.

Remark that reducible non-unimodular Pisot substitutions raised particular interest since recently irreducibility has been criticized to be a natural assumption. Indeed one can take an irreducible Pisot substitution and rewrite it to obtain another substitution that is not irreducible but has topologically conjugate dynamics. In \cite{BBJS:12} a topological condition on the first rational C\v{e}ch cohomology of the tiling space of the substitution is introduced. Pisot substitutions satisfying this condition are called \emph{homological} and it is conjectured that their tiling spaces are $m$-to-$1$ extensions of their maximal equicontinuous factor (a torus or a solenoid), where $m$ is a divisor of the norm of the Pisot number.

\subsection*{Results of this paper}
In this paper we set up a geometrical theory for the dynamics of reducible Pisot substitutions. We take inspiration from some ideas of \cite{AFHI} for the study of a free group automorphism associated with a complex Pisot root.
The main tools are the duals of higher dimensional extensions of substitutions (Section~\ref{sec:duals}), first introduced in \cite{SAI}.
Since we want to construct fractal tilings on the contracting space $\mathbb{K}_c\cong\mathbb{R}^{d-1}$ of the substitution, we want to work with $(d-1)$-dimensional faces in $\mathbb{R}^n$, thus it turns out that the dual substitution $\mathbf{E}^*_{n-d+1}(\sigma)$, and its concrete geometric realization $\mathbf{E}^{d-1}(\sigma)$, will be suitable for this task. Our Rauzy fractals will be defined as Hausdorff limits of renormalized patches of polygons generated by iterations of the dual $\mathbf{E}^{d-1}(\sigma)$. In the irreducible case these fractals coincide with the usual ones, thus our theory is more general.

We introduce some important geometrical and algebraic conditions in order to develop a tiling theory with these objects (see Section~\ref{sec:hyp} for precise definitions). We deal with {\it nice reducible substitutions}, a condition which implies that $\mathbf{E}^{d-1}(\sigma)$ induces an inflate-and-subdivide rule on $\mathbb{K}_c$, it is positive (as defined in \cite{AFHI}) and has exactly the Pisot number as inflation factor.  
The {\it geometric finiteness property} will ensure that the stepped surfaces cover the entire space $\mathbb{K}_c$. Finally we define a slight generalization of the {\it strong coincidence condition} to ensure that a domain exchange is well-defined on the subtiles of the Rauzy fractals.

Under these conditions, we solve some well-known problems for reducible Pisot substitution posed e.g.~in \cite{EIR}:
\begin{itemize}
\item We find geometrical representations for stepped surfaces and related polygonal tilings (Theorem~\ref{thm:poly} in Section~\ref{sec:step}).
\item We show that our Rauzy fractals form {\it aperiodic self-replicating tilings} by just replacing them in the polygonal tiling induced by a stepped surface (Theorem~\ref{thm:selfreptil} in Section~\ref{sec:rauzyfrac}). Examples can be found in Figures~\ref{fig:s1polandr} and~\ref{fig:s2polandr}.
\item We get natural {\it periodic tilings} by Rauzy fractals, which were missing in the previous works on reducible substitutions (Theorem~\ref{thm:pertil} in Section~\ref{sec:rauzyfrac}).
\end{itemize}
Our Rauzy fractals turn out to be exactly those extended domains considered in \cite{EI,BBK}, with the advantage that they are generated explicitly in a systematic way by the dual substitution $\mathbf{E}^{d-1}(\sigma)$. 
Throughout the paper, we apply our results to a one-parameter family of substitutions defined in Section~\ref{sec:main-ex}, which will constitute our main example. One substitution of this family (see Section~\ref{sec:hok}) was explicitly studied with the same approach in~\cite{EEFI}; a tiling result was obtained and mentioned to hold for the whole family. In Sections~\ref{sec:comb}, we will further relate different definitions for the Rauzy fractals. Indeed, 
Rauzy fractals can be defined also as the closure of the projection of vertices of the \emph{stepped line} representing geometrically the fixed point of the substitution. 
We give a new combinatorial approach which consists in applying a morphism, defined by taking into account the rational dependencies arising in the reducible case, to the stepped line. This morphism turns in some sense the stepped line into an irreducible one, considering only the letters associated with the rationally independent generators (cf.~\cite{Fret} for a similar approach in the framework of model sets for reducible substitutions).
Projecting the vertices of this modified stepped line onto $\mathbb{K}_c$ is an equivalent definition of the Rauzy fractal generated by $\mathbf{E}^{d-1}(\sigma)$. We investigate then in Section~\ref{sec:symbdyn} the codings of the domain exchange defined on our new Rauzy fractals. We conclude by showing that, for our family of substitution, $(X_\sigma,S)$ is measurably conjugate to the first return of a toral translation (Theorem~\ref{thm:conj}).
We end with Section~\ref{sec:persp} where some perspectives for future works are presented.

\begin{figure}[h]
\begin{center}$
\begin{array}{cc}
\includegraphics[scale=.35]{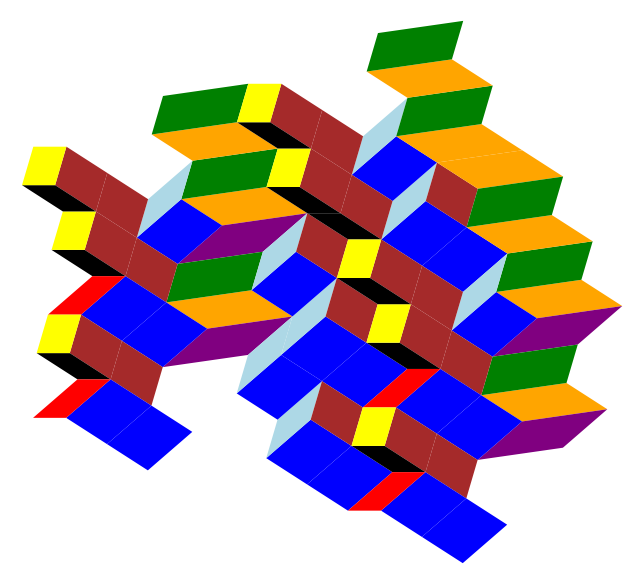} & \quad
\includegraphics[scale=.35]{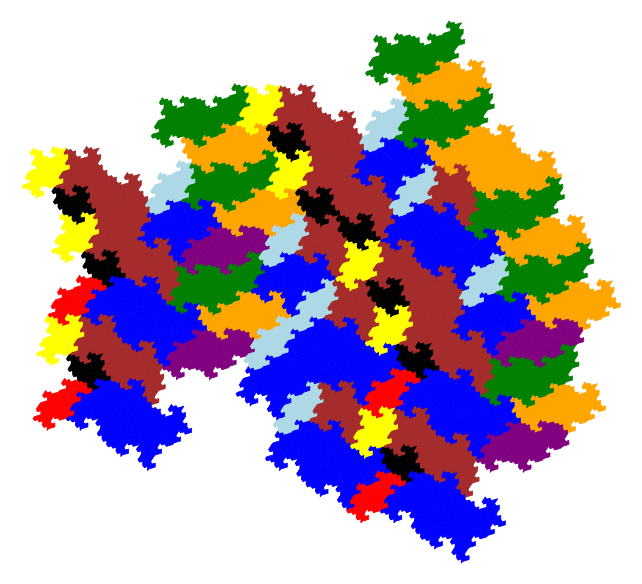}
\end{array}$
\end{center}
\caption{Polygonal and Rauzy fractals aperiodic tiling generated by $\mathbf{E}^2(\sigma_1)$ on $(\mathbf{0},1\wedge 5)+(\mathbf{0},4\wedge 5)+(\mathbf{0},1\wedge 4)$, with $\sigma_1$ defined in \eqref{fam}.}\label{fig:s1polandr}
\end{figure}

\begin{figure}[h]
\begin{center}$
\begin{array}{cc}
\includegraphics[scale=.32]{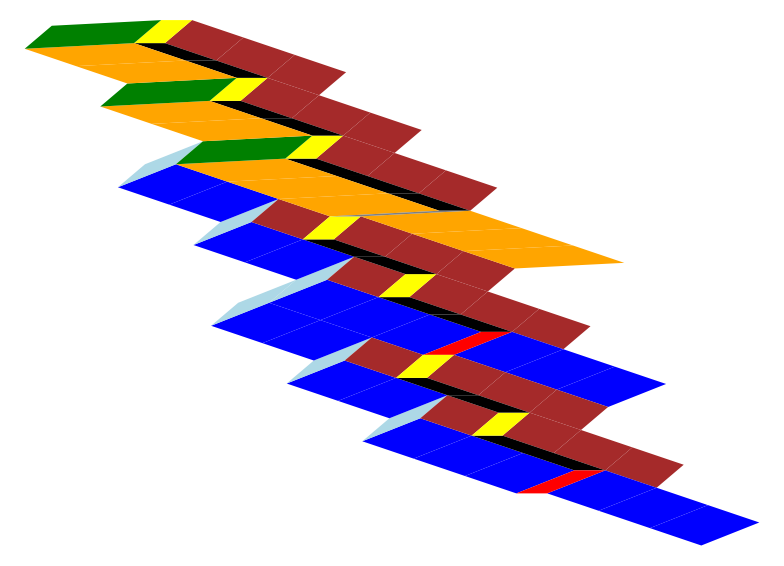} & \hspace{-.5cm}
\includegraphics[scale=.32]{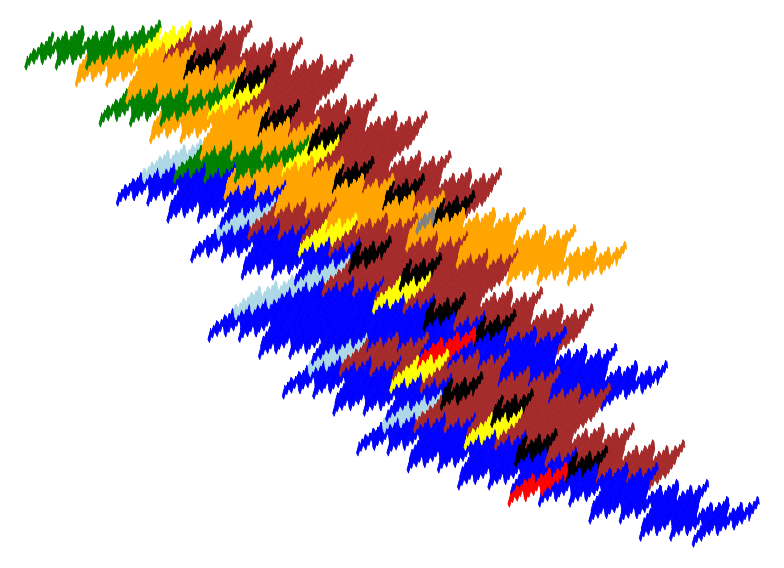}
\end{array}$
\end{center}
\caption{Polygonal and Rauzy fractals aperiodic tiling generated by $\mathbf{E}^2(\sigma_2)$ on $(\mathbf{0},3\wedge 4)+(\mathbf{0},3\wedge 5)+(\mathbf{0},4\wedge 5)$, with $\sigma_2$ defined in \eqref{fam}.}\label{fig:s2polandr}
\end{figure}

\section{Hokkaido substitution}\label{sec:hok}
Before dealing with the general case, we present here an overview of our constructions and results for a particular substitution,  namely, the well-known \emph{Hokkaido substitution}, associated with the minimal Pisot number:
$$\sigma_0:\;1\mapsto 12,\; 2\mapsto 3,\; 3\mapsto 4,\; 4\mapsto 5,\; 5\mapsto 1.
$$
 Precise definitions and statements will be given in the next sections.

The incidence matrix of $\sigma_0$ is 
$$\begin{array}{cc}
M_{\sigma_0}=\begin{pmatrix}
1&0&0&0&1\\
1&0&0&0&0\\
0&1&0&0&0\\
0&0&1&0&0\\
0&0&0&1&0
\end{pmatrix}
\end{array}
$$
and the associated polynomial is
\[ f(x)g(x) = (x^3-x-1)(x^2-x+1). \] 
The classical Rauzy fractal of the Hokkaido substitution is made of five subtiles obtained by projecting the vertices of the stepped line into the contracting space of the substitution $\mathbb{K}_c\cong\mathbb{R}^2$. Precisely, given a fixed point $u$ of $\sigma_0$,
\begin{equation}\label{eq:bro} 
\mc{R}(a) = \overline{\{\pi_c(\mathbf{l}(u_0\cdots u_{k-1})): k\in\mathbb{N}, u_k = a) \}},\quad a\in\mathcal{A}=\{1,\ldots,5\}.
\end{equation} 
Since the strong coincidence condition holds for $\sigma_0$, the domain exchange 
\begin{equation}\label{eq:classic-exchange}
E: \,\mathbf{x}\mapsto \mathbf{x}+\pi_c(\mathbf{e}_a),\quad \text{if } \mathbf{x}\in\mathcal{R}(a), 
\end{equation}
is well-defined almost everywhere. See the top of Figure~\ref{fig:dom2}.
For every unit Pisot substitution satisfying the strong coincidence condition, the substitutive system $(X_\sigma,S,\mu)$ is measurably conjugate to the domain exchange $(\mathcal{R},E,\lambda)$. See  \cite{AI,CS} for the details in the irreducible setting and \cite{EIR} for the reducible one.

As described in the introduction, the construction of stepped surfaces and the existence of periodic tilings for this substitution is not clear.
To solve these problems we consider, instead of the dual substitution $\mathbf{E}_1^*(\sigma_0)$, the higher-dimensional dual $\mathbf{E}_3^*(\sigma_0)$ and its geometric realization $\mathbf{E}^2(\sigma_0)$ (see Section~\ref{sec:duals} for precise definitions).
Indeed, it turns out that these are the right dual substitutions to consider if we want to construct stepped surfaces approximating the two-dimensional contracting plane $\mathbb{K}_c$ of $\sigma_0$. 

In fact, $\mathbf{E}^2(\sigma_0)$ acts on oriented two-dimensional faces in $\mathbb{R}^5$. These faces are represented with a wedge formalism and the substitution rule defined on them is

\begin{align*}
\mathbf{E}^2(\sigma_0): (\mathbf{0},4\wedge 5) &\mapsto (\mathbf{0},3\wedge 4)  \\ 
(\mathbf{0},3\wedge 5) &\mapsto (\mathbf{0},2\wedge 4)  \\  
(\mathbf{0},3\wedge 4) &\mapsto (\mathbf{0},2\wedge 3)  \\  
(\mathbf{0},2\wedge 5) &\mapsto (M_\sigma^{-1}\mathbf{e}_2, 4\wedge 5)  \cup (\mathbf{0},1\wedge 4)\\  
(\mathbf{0},2\wedge 4) &\mapsto (M_\sigma^{-1}\mathbf{e}_2,3\wedge 5) \cup (\mathbf{0},1\wedge 3) \\  
(\mathbf{0},2\wedge 3) &\mapsto(M_\sigma^{-1}\mathbf{e}_2,2\wedge 5)  \cup  (\mathbf{0},1\wedge 2)\\
(\mathbf{0},1\wedge 5) &\mapsto -(\mathbf{0},4\wedge 5)  \\ 
(\mathbf{0},1\wedge 4) &\mapsto -(\mathbf{0},3\wedge 5)  \\ 
(\mathbf{0},1\wedge 3) &\mapsto -(\mathbf{0},2\wedge 5)  \\ 
(\mathbf{0},1\wedge 2) &\mapsto -(\mathbf{0},1\wedge 5)  \\ 
\end{align*}

\begin{figure}[h]
\begin{center}$
\begin{array}{cc}
\includegraphics[scale=.14]{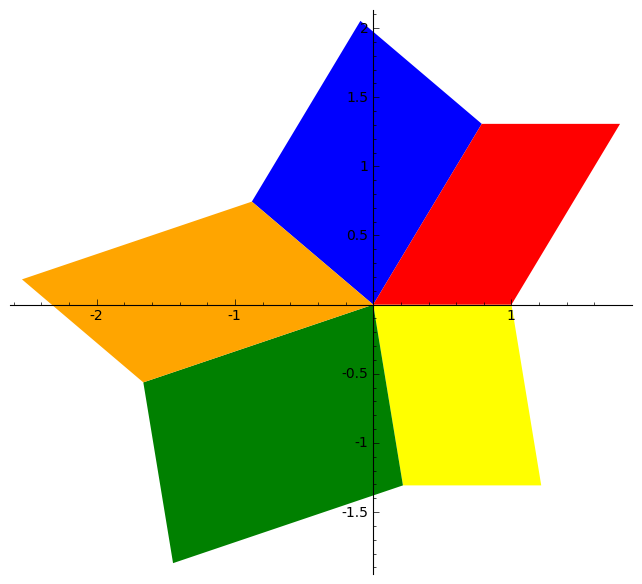} &
\includegraphics[scale=.14]{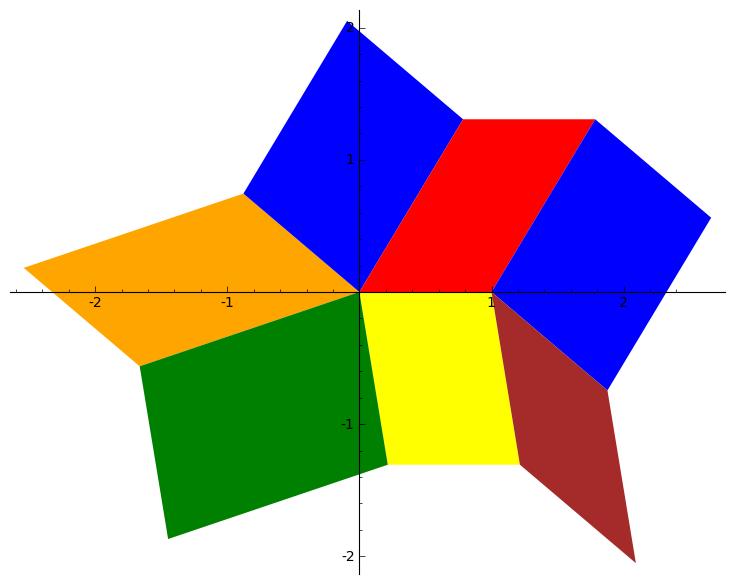} \\
\includegraphics[scale=.2]{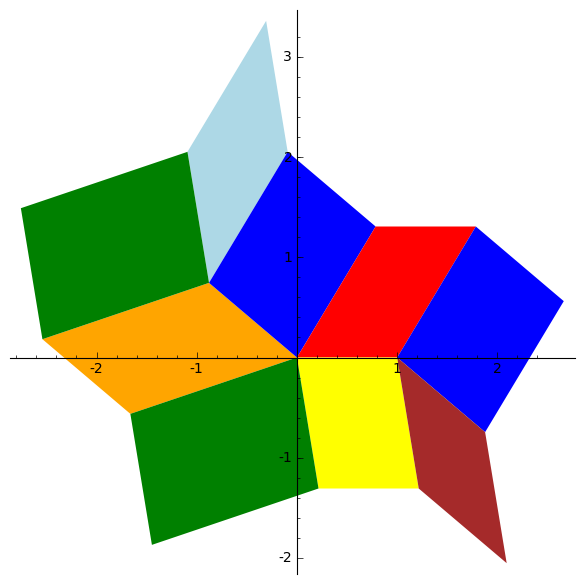} &
\includegraphics[scale=.23]{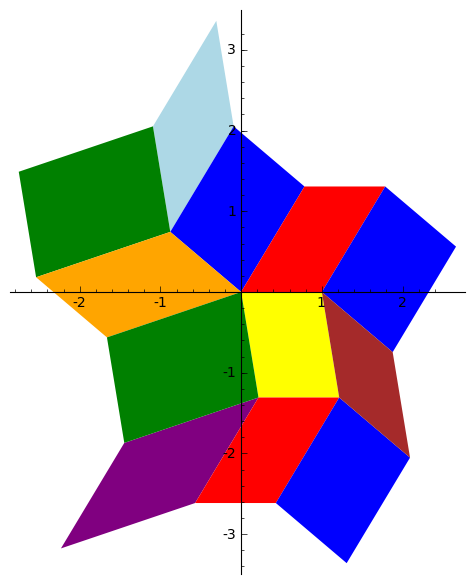} \\
\includegraphics[scale=.28]{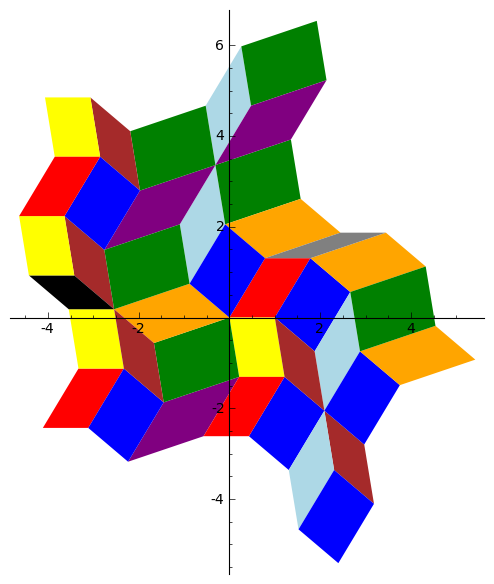} &
\includegraphics[scale=.32]{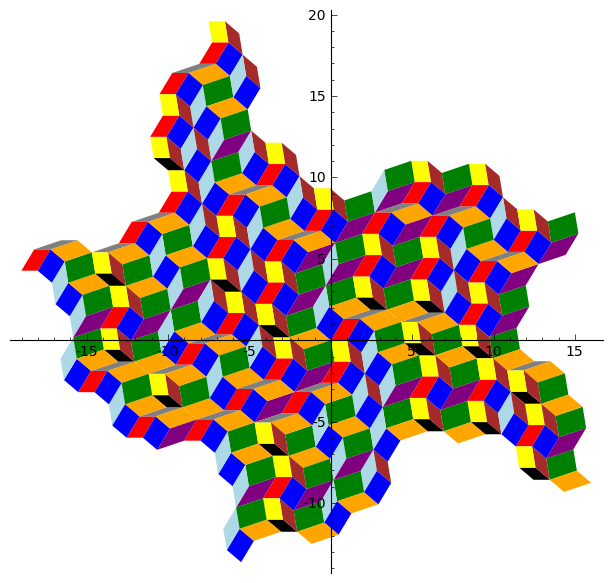}
\end{array}$
\end{center}
\caption{Iterates of $\mathbf{E}^2(\sigma_0)$ on $\mathcal{U}=\{(\mathbf{0},1\wedge 3),(\mathbf{0},1\wedge 4),(\mathbf{0},2\wedge 4),(\mathbf{0},2\wedge 5),(\mathbf{0},3\wedge 5)\}$. Faces of $\mathcal{U}$ are colored from left to right by red, yellow, green, orange and blue.\label{5patch}}
\end{figure}

We can show that $\sigma_0$ is a nice reducible substitution, that is, it satisfies all the necessary hypotheses described in Section~\ref{sec:hyp} such that $\mathbf{E}^2(\sigma_0)$ generates stepped surfaces approximating $\mathbb{K}_c$. This is illustrated in Figure~\ref{5patch}.

Iterates of $\mathbf{E}^2(\sigma_0)$ on initial sets of faces can be renormalized by the contracting action of $M_\sigma$ on $\mathbb{K}_c$. Taking the limit of this operation with respect to the Hausdorff metric allows us to define new Rauzy fractals subtiles (cf.~Definition \ref{def:rauzyfrac})
\[ \mathcal{R}(a\wedge b) = \lim_{k\to\infty} M_\sigma^k\,\pi_c(\overline{\mathbf{E}^2(\sigma_0)^k (\mathbf{0},a\wedge b)}),\quad \text{ for any } a\wedge b \in \wedge^2\mathcal{A}.
\]
Since iterating $\mathbf{E}^2(\sigma_0)$ on certain sets of faces covers $\mathbb{K}_c$, the geometric finiteness condition holds (see Definition~\ref{def:geomprop}), and we get aperiodic tilings by Rauzy fractals defined by the dual substitution $\mathbf{E}^2(\sigma_0)$ (see Theorem~\ref{thm:selfreptil}). This can be visualized in Figure~\ref{fig:s0tilings}.

\begin{figure}[h]
\begin{center}$
\begin{array}{cc}
\includegraphics[scale=.3]{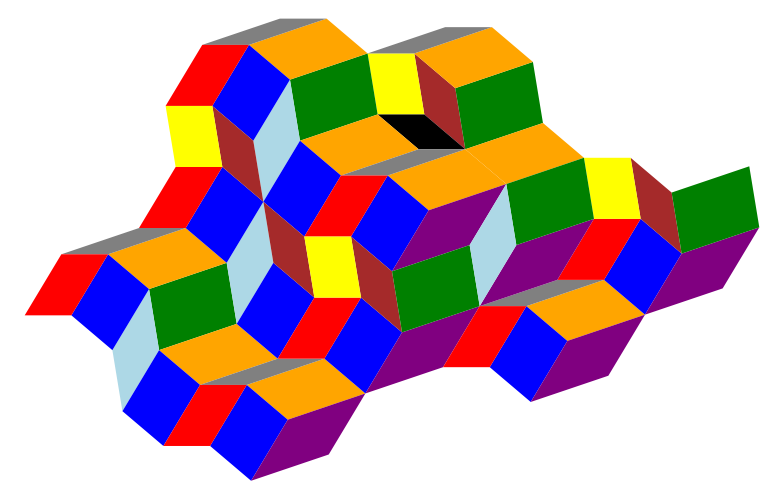} & \quad
\includegraphics[scale=.3]{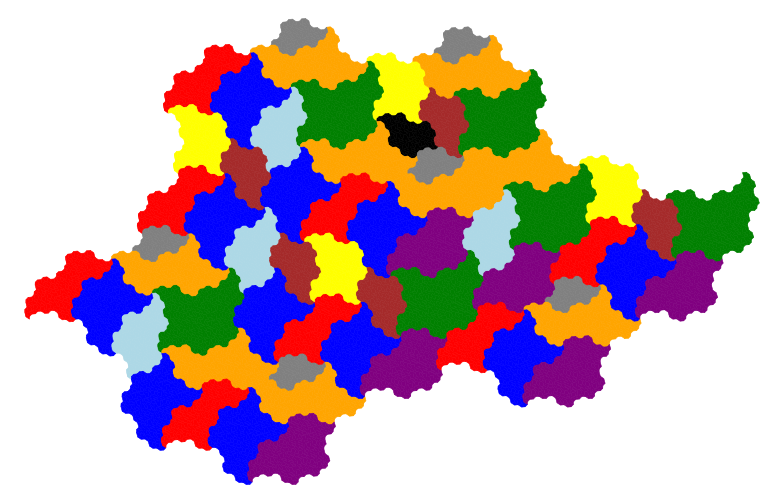}
\end{array}$
\end{center}
\caption{Polygonal and Rauzy fractals aperiodic tiling generated by $\mathbf{E}^2(\sigma_0)$ on $(\mathbf{0},2\wedge 3)+(\mathbf{0},2\wedge 4)+(\mathbf{0},3\wedge 4)$ for the Hokkaido substitution $\sigma_0$.}\label{fig:s0tilings}
\end{figure}

We can furthermore construct easily sets of two-dimensional faces which tile periodically $\mathbb{K}_c$. Applying the same renormalization process described above, we get periodic tilings by Rauzy fractals (see Figure~\ref{figper3} and cf.~Theorem~\ref{thm:pertil}).

After having defined a generalization of the strong coincidence condition (see Definition~\ref{def:scc}) and having verified that it holds for $\sigma_0$ (see Proposition~\ref{prop:scc}), a domain exchange can be defined on the Rauzy fractals $\mathcal{R}(a\wedge b)$. The tiling property implies that the considered Rauzy fractals are fundamental domains for the torus $\mathbb{T}^2$ and the domain exchange turns into a translation action on it. 

It is left to show the relations between the new Rauzy fractals $\mathcal{R}(a\wedge b)$ generated by the dual $\mathbf{E}^2(\sigma_0)$ and the $\mathcal{R}(a)$ classically defined by projection of stepped lines as in \eqref{eq:bro}. This is done in Section~\ref{sec:comb}. Indeed, we can modify the classical stepped line into another one on three letters. 
This is due to linear dependencies arising when projecting the stepped line along the neutral space and can be interpreted symbolically by applying the word morphism $\chi$ defined in \eqref{chi}. 
We can prove that the fractal domain obtained by projecting down this modified stepped line equals the Rauzy fractal generated by $\mathbf{E}^2(\sigma_0)$ on some initial set of faces (see Proposition~\ref{pro:iso}).
Using these identifications we can prove that the domain exchange $E$ on the classical Rauzy fractal $\mathcal{R}=\bigcup_{a\in\mathcal{A}}\mathcal{R}(a)$ is the first return of the toral translation given by the domain exchange on the Rauzy fractal $\mathcal{R}_\mathcal{P} = \bigcup_{a\wedge b\in \mathcal{P}} \mathcal{R}(a\wedge b)$ (see Proposition~\ref{pro:return} and Figure~\ref{fig:dom2}). 

It follows from Theorem~\ref{thm:conj} that the substitutive shift $(X_{\sigma_0},S)$ is the first return of a toral translation.

\begin{figure}
\begin{center}$
\begin{array}{cc}
\includegraphics[scale=.3]{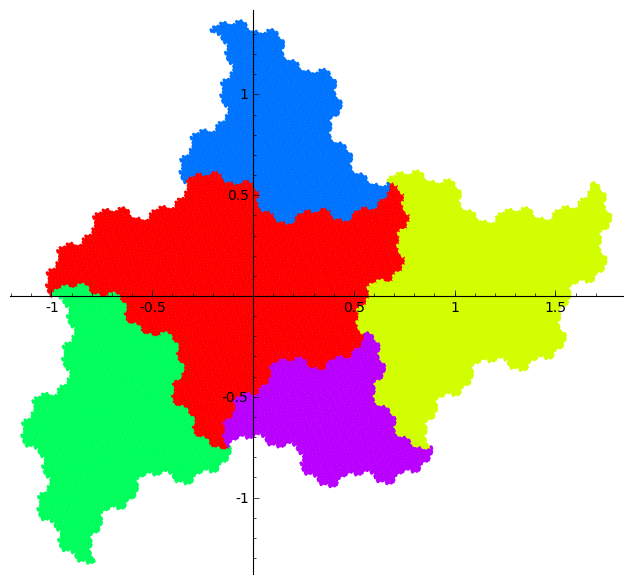} &
\includegraphics[scale=.3]{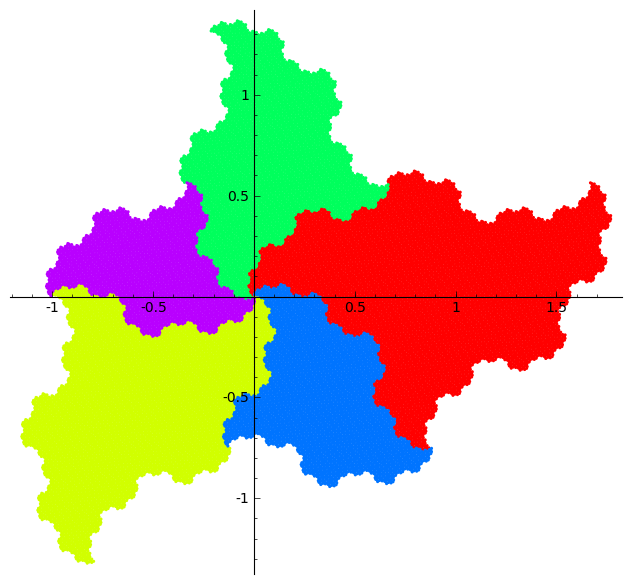} \\
\includegraphics[scale=.35]{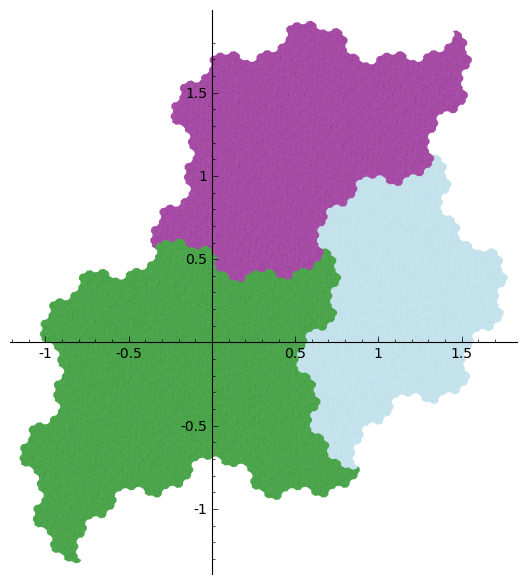} &
\includegraphics[scale=.35]{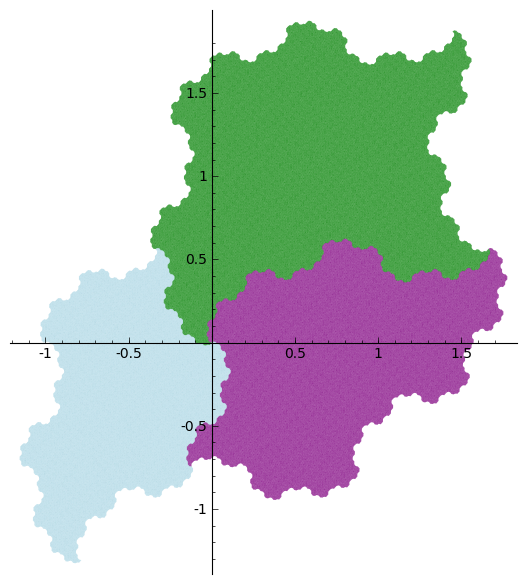}
\end{array}$
\end{center}
\caption{The domain exchanges $(\mathcal{R},E)$ and $(\widetilde{\mathcal{R}},\widetilde{E})$ for $\sigma_0$. One can see that $E$ is the first return of $\widetilde E$ on $\mathcal{R}$ (see Proposition~\ref{pro:return}): for example, the red subtile $\mathcal{R}(1)$ is inside the green $\widetilde{\mathcal{R}}(3)$; applying twice $\widetilde{E}$ on it gives $E(\mathcal{R}(1))$.\label{fig:dom2}}
\end{figure}

\section{General setting and main example}\label{sec:gen}

\subsection{Pisot substitutions}\label{prel}
Let $\mathcal{A}=\{1,2,\ldots,n\}$ be a finite alphabet. A \emph{substitution} $\sigma$ is an endomorphism of the free monoid $\mathcal{A}^*$ sending non-empty words to non-empty words. Its \emph{incidence matrix} $M_\sigma$ is defined by $(M_\sigma)_{b,a} = |\sigma(a)|_b$, for $a,b\in\mathcal{A}$, where $|w|_b$ is the number of occurrences of the letter $b$ in $w\in\mathcal{A}^*$. Denote by $\mathbf{l}:\mathcal{A}^*\to\mathbb{R}^n$, $w_0\cdots w_k\mapsto \mathbf{e}_{w_0}+\cdots+\mathbf{e}_{w_k}$ the \emph{abelianization map}. Then $\mathbf{l}\circ\sigma = M_\sigma\circ\mathbf{l}$.
 We say that a substitution $\sigma$ is \emph{primitive} if $M_\sigma$ is primitive. By the Perron-Frobenius theorem a primitive matrix has a dominant simple eigenvalue. 
A \emph{Pisot substitution} is a substitution such that the dominant eigenvalue of $M_\sigma$ is a \emph{Pisot number}, i.e.~an algebraic integer $\beta > 1$ such that its Galois conjugates $\beta'$ other than $\beta$ itself satisfy $|\beta'| < 1$. We say that a Pisot substitution is {\it unit} if the Pisot number $\beta$ is a unit, that is, its norm $N(\beta) = \pm 1$.
Given a Pisot substitution $\sigma$, suppose that the characteristic polynomial of $M_\sigma$ decomposes over $\Q$ into irreducible factors as 
\[ \det(xI-M_\sigma) = f(x) g_1(x)^{m_1}\cdots g_k(x)^{m_k},   \]
where $f(x)$ is the minimal polynomial of the Pisot root $\beta$. We call $f(x)$ the \emph{Pisot polynomial} and $g(x):= g_1(x)^{m_1}\cdots g_k(x)^{m_k}$ the \emph{neutral polynomial} associated to $\sigma$. If $g(x)=1$ we say that the substitution $\sigma$ is \emph{irreducible}, otherwise we call it \emph{reducible}.
An irreducible Pisot substitution is primitive (see \cite[Proposition~1.3]{CS}), but this is not true in general for reducible Pisot substitutions. 

The \emph{prefix} and \emph{suffix graphs} of the substitution are defined by having $\mathcal{A}$ as set of nodes and edges $a\stackrel{p}{\rightarrow} b$, respectively $a\stackrel{s}{\rightarrow} b$, whenever $\sigma(a) = pbs$, for $b\in\mathcal{A}$ and $p,s\in\mathcal{A}^*$. 

\subsection{Substitutive subshifts}
Given a fixed point $u\in\mathcal{A}^\mathbb{Z}$ of a primitive substitution $\sigma$, let $X_\sigma = \overline{\{S^k u : k\in \mathbb{Z}\}}$, where $S:\mathcal{A}^\mathbb{Z}\to\mathcal{A}^\mathbb{Z}$ is the shift $S(u_i)_{i\in\Z}=(u_{i+1})_{i\in\Z}$.
The symbolic dynamical system $(X_\sigma,S)$ is minimal (every orbit is dense) and uniquely ergodic (there exists a unique $S$-invariant measure $\mu$).

\subsection{Representation space}
In this section we define all that is necessary for the representation space where the Rauzy fractals will live.

Let $\{\beta^{(1)},\ldots,\beta^{(r)},\beta^{(r+1)},\overline{\beta^{(r+1)}},\ldots,\beta^{(r+s)},\overline{\beta^{(r+s)}}\}$ be the Galois conjugates of the Pisot number $\beta=\beta^{(1)}$. If $\deg(\beta) = d$ then $d = r + 2s$.
Choose dual bases $\{\mathbf{u}_{\beta^{(i)}}\}_{i=1}^{r+s}$, $\{\mathbf{v}_{\beta^{(i)}}\}_{i=1}^{r+s}$ of right, respectively left eigenvectors for $M_\sigma$ associated with the $\{\beta^{(i)}\}_{i=1}^{r+s}$, such that each $\mathbf{v}_{\beta^{(i)}}\in \mathbb{Z}[\beta]^n$ and renormalize $\mathbf{u}_\beta$ such that $\mathbf{u}_\beta\cdot \mathbf{v}_\beta = 1$.

There exists a unique $M_\sigma$-invariant decomposition $\R^n = \mathbb{K}_\beta \oplus \mathbb{K}_n$, where the restriction of $M_\sigma$ to $\mathbb{K}_\beta$ is hyperbolic with characteristic polynomial $f(x)$. 
The \emph{Pisot space} $\mathbb{K}_\beta$ is the direct sum of the eigenspaces associated with the $\{\beta^{(i)}\}_{i=1}^{r+s}$. It has an expanding/contracting splitting $\mathbb{K}_\beta=\mathbb{K}_e\oplus\mathbb{K}_c$, where $\mathbb{K}_e = \R\mathbf{u}_\beta$ and $\mathbb{K}_c \cong \R^{r-1}\times\C^s$ is the $(d-1)$-dimensional linear subspace generated by the vectors 
$\{\mathbf{u}_{\beta^{(i)}}\}_{i=2}^{r+s}$. 
The \emph{neutral space} $\mathbb{K}_n := f(M_\sigma)\R^n$ is generated by the eigenspaces associated with the roots of the neutral polynomial $g(x)$. Since $\mathbf{v}_{\beta^{(i)}}  \in \ker(f({}^t M_\sigma))$, we deduce that  $\mathbf{v}_{\beta^{(i)}}  \perp \mathbb{K}_n$ for all $i\in\{1,\ldots,r+s\}$. We equip each of these spaces with its appropriate (according to its nature and dimension) Lebesgue measure $\lambda$.  

We consider the projections 
\begin{alignat*}{2}
\pi &: \mathbb{R}^n \to \mathbb{K}_\beta,&\quad  \mathbf{x} &\mapsto \sum_{i=1}^{r+s}\langle \mathbf{x},\mathbf{v}_{\beta^{(i)}}\rangle \mathbf{u}_{\beta^{(i)}}
 \\
\pi_e &: \mathbb{R}^n \to \mathbb{K}_e,&\quad  \mathbf{x} &\mapsto \langle \mathbf{x},\mathbf{v}_{\beta^{(1)}}\rangle\mathbf{u}_{\beta^{(1)}}  \\
\pi_c &: \mathbb{R}^n \to \mathbb{K}_c,&\quad  \mathbf{x} &\mapsto \sum_{i=2}^{r+s}\langle \mathbf{x},\mathbf{v}_{\beta^{(i)}}\rangle \mathbf{u}_{\beta^{(i)}}
\end{alignat*} 
Notice that $\pi(\mathbb{Z}^n)$ is a full rank lattice in $\mathbb{K}_\beta$ and the $\{\pi(\mathbf{e}_i)\}_{i=1}^n$ are redundant generators, as indicated by the following lemma.
\begin{lemma}\label{lem:redund}
Let $\sigma$ be a unit reducible Pisot substitution. Then the vectors $\{\pi(\mathbf{e}_i)\}_{i=1}^n$ are $\mathbb{Z}$-linearly dependent.
\end{lemma}
\begin{proof}
We know that $\{f(M_\sigma(\mathbf{e}_i));i=1,\ldots,n\}\ne\{\mathbf{0}\},
$ since $f(M_\sigma)$ is not trivial. Let $i_0$ such that  $f(M_\sigma(\mathbf{e}_{i_0}))=:\mathbf{y}\ne \mathbf{0}$. Then $\mathbf{y}\in\mathbb{Z}^n$, because $M_\sigma$ has integer coefficients. Thus $\mathbf{y}=\sum_{i=1}^nc_i\mathbf{e}_i$, where the $c_i\in\mathbb{Z}$ are not all $0$. It follows that 
$$\pi(\mathbf{y})=\sum_{i=1}^{r+s}\langle\mathbf{y},\mathbf{v}_{\beta^{(i)}}\rangle\mathbf{u}_{\beta^{(i)}}=\mathbf{0}.
$$
Indeed, $\mathbf{y}\in\mathbb{K}_n$ and all the $\mathbf{v}_{\beta^{(i)}}$ are orthogonal to $\mathbb{K}_n$, as mentioned above. This gives the non trivial relation 
$\pi(\mathbf{y})=\sum_{i=1}^n c_i\pi(\mathbf{e}_i)=\mathbf{0}$
with integer coefficients. 
\end{proof}
Since the $\mathbf{u}_{\beta^{(i)}}$ and $\mathbf{v}_{\beta^{(i)}}$ are right and respectively left eigenvectors of $M_\sigma$ associated with $\beta^{(i)}$, for $i=1,\ldots,r+s$, we have that $M_\sigma$ commutes with $\pi$, $\pi_e$ and $\pi_c$, and it is an expansion on $\mathbb{K}_e$ and a contraction on $\mathbb{K}_c$.

Given a measurable set $W\subseteq\mathbb{K}_c$, if $\beta$ is a unit we have
\begin{equation}\label{eq:contrac} \lambda(M_\sigma W) = \beta^{-1} \lambda(W). 
\end{equation}
Indeed, $M_\sigma$ is a uniform contraction whose eigenvalues are the Galois conjugates of $\beta$, and $|\beta^{(2)}\cdots\beta^{(r)}||\beta^{(r+1)}|^2\cdots
|\beta^{(r+s)}|^2 = \beta^{-1}$ since $\beta$ is a unit.

\subsection{Main example}\label{sec:main-ex}
We consider the one-parameter family of unit reducible Pisot substitutions
\begin{equation}\label{fam}
\sigma_t:\;1\mapsto 1^{t+1}2,\; 2\mapsto 3,\; 3\mapsto 4,\; 4\mapsto 1^t 5,\; 5\mapsto 1, \quad t\in\mathbb{N}_0,
\end{equation}  
with associated polynomials 
\[ f(x)g(x) = (x^3-tx^2-(t+1)x-1)(x^2-x+1). \] 
We have by notation $n=5$ and $d=3$. The incidence matrices are
$$\begin{array}{cc}
M_{\sigma_t}=\begin{pmatrix}
t+1&0&0&t&1\\
1&0&0&0&0\\
0&1&0&0&0\\
0&0&1&0&0\\
0&0&0&1&0
\end{pmatrix}.
\end{array}
$$
We will develop this example all along the following sections, applying to it progressively the results that we will obtain. Note that the Hokkaido substitution $\sigma_0$ of Section~\ref{sec:hok} belongs to this family.

\subsection{Standing assumptions}
For the rest of the paper we will always assume that $\mathcal{A} = \{1,2,\ldots,n\}$ and $\sigma$ is a primitive unit Pisot substitution with Pisot root $\beta$ such that $\deg(\beta) = d \le n$.

\section{Higher dimensional duals}\label{sec:duals} 
We recall the definition and main properties of $k$-dimensional extensions of a substitution and their dual, first defined in \cite{SAI}.

\subsection{Faces}\label{sec:faces}

Let $\mathcal{O}_k$ be the set of elements $\und{a}=a_1\wedge\cdots\wedge a_k\in \bigwedge_{i=1}^k\mathcal{A}$ with $a_1<\cdots<a_k$.
Let $C_k$ be the free $\Z$-module with basis elements in $\Z^n \times \mathcal{O}_k$. An element $(\mathbf{x},a_1\wedge\cdots\wedge a_k)\in C_k$ is called a {\it $k$-dimensional face} and consists of its {\it base point} $\mathbf{x}$ and of its {\it type} $\und{a}:= a_1\wedge\cdots\wedge a_k$. We shall think of $C_k$ as the space of formal finite integer weighted sums of $k$-dimensional faces. We also define $(\mathbf{x},a_1\wedge\cdots\wedge a_k)$ for more general cases as follows.
\begin{itemize}
\item $(\mathbf{x},a_1\wedge\cdots\wedge a_k)=0$ if $a_i = a_j$ for some $i,j$.
\item Antisymmetry: $(\mathbf{x},a_{\tau(1)}\wedge\cdots\wedge a_{\tau(k)}) = \text{sgn}(\tau) (\mathbf{x},a_1\wedge\cdots\wedge a_k)$, where $\text{sgn}(\tau)$ is the signature of the permutation $\tau$.
\end{itemize}
Observe that this justifies the wedge product as choice of notation.

Given a face $(\mathbf{x},a_1\wedge\cdots\wedge a_k)\in C_k$, its {\it support} is 
\[
\overline{(\mathbf{x},a_1\wedge\cdots\wedge a_k)} := \bigg\{\mathbf{x} + \sum_{i=1}^k t_i \mathbf{e}_{a_i}: t_i\in[0,1] \bigg\}.
\]
For the element $0$, the support is the empty set. For a general element of $ C_k$, $P=\sum_{l=1}^Nn_l(\mathbf{x}_l,\und{a}_l)$ with $n_l\in\mathbb{Z}\setminus\{ 0\}$ and $\und{a}_l\in\mathcal{O}_k$ for all $l\in\{1,\ldots, N\}$, the support is 
$$\overline{P}:=\bigcup_{l=1}^N\overline{(\mathbf{x}_l,\und{a}_l)}.
$$

Given an element $P\in C_k$, we denote by $\{P\}$ the set of faces $F=(\mathbf{x},\und{a})$ with $\und{a}\in\mathcal{O}_k$ appearing in $P$ with a non-zero coefficient. We then write $F\in \{P\}$ or simply $F\in P$. 
  
Finally, an element $\sum_{l=1}^Nn_l(\mathbf{x}_l,\und{a}_l)$ with $n_l\in\{-1,1\}$  for all $l\in\{1,\ldots, N\}$  is called \emph{geometric}. 
We say that a face $(\mathbf{x},\und{a})\in\mathbb{Z}^n\times\mathcal{O}_k$ is \emph{positive} or \emph{positively oriented}, while $-(\mathbf{x},\und{a})$ is \emph{negative} or \emph{negatively oriented}.

\subsection{Notations} 
Suppose that, for a given $\und{a}=a_1\wedge\cdots\wedge a_k\in\bigwedge_{i=1}^k \mathcal{A}$, we have $\sigma(a_i)=p_i b_i s_i$, for all $i\in\{1,\ldots,k\}$ and $b_i\in\mathcal{A}$. Then we let $\mathbf{p}=(p_1,\ldots,p_k)$,  $\mathbf{s}=(s_1,\ldots, s_k)$ and write $\sigma(\und{a})=\mathbf{p}\und{b}\mathbf{s}$. 

We denote by $\textrm{Per}_k$ the set of permutations on $\{1,\ldots,k\}$. For $\tau\in\textrm{Per}_k$ and $\und{a}\in\bigwedge_{i=1}^k \mathcal{A}$, the element $a_{\tau(1)}\wedge\cdots\wedge a_{\tau(k)}$ will be written $\und{a}_\tau$ for short.

Finally, we denote by  $\mathbf{l}(\mathbf{p})$ and $\mathbf{l}(\mathbf{s})$ the sum of the abelianizations of all prefixes $p_i$ appearing in $\mathbf{p}$, suffixes $s_i$ in $\mathbf{s}$ respectively.

\subsection{Extensions and their duals}
\begin{definition}
The $k$-\emph{dimensional extension} of $\sigma$ is the linear map on $C_k$
\begin{equation}\label{eq:Ek}
\mathbf{E}_k(\sigma)(\mathbf{x},\und{a}) = \\ 
\sum_{\sigma(\und{a})=\mathbf{p}\und{b}\mathbf{s}} (M_\sigma \mathbf{x} + \mathbf{l}(\mathbf{p}), \und{b} ). 
\end{equation}
\end{definition}

Write $(\mathbf{x},\und{a})^*$ for the dual of the element $(\mathbf{x},\und{a})\in C_k$.
Let $C_k^*$ be the free $\Z$-module generated by the basis elements $(\mathbf{x},\und{a})^*$. 
We are interested in the dual of $\mathbf{E}_k(\sigma)$.

\begin{proposition}
We have
\begin{equation}\label{eq:Ekstar}
\mathbf{E}^*_k(\sigma)(\mathbf{x},\und{a})^* = 
\sum_{\sigma(\und{b})=\mathbf{p}\und{a}\mathbf{s}} \big(M_\sigma^{-1}(\mathbf{x} - \mathbf{l}(\mathbf{p})),\und{b} \big)^*  
\end{equation}
\end{proposition}
\begin{proof}
Let $(\mathbf{x},\und{a})^*\in C_k^*$. By definition of the dual of the linear map $\mathbf{E}_k(\sigma)$ and by~\eqref{eq:Ek}, we can write
$$\begin{array}{l}
\langle\mathbf{E}_k^*(\sigma)(\mathbf{x},a_1\wedge\cdots\wedge a_k)^*,(\mathbf{y},b_1\wedge\cdots\wedge b_k)\rangle\\\\ 
= \langle(\mathbf{x},a_1\wedge\cdots\wedge a_k)^*,\mathbf{E}_k(\sigma)(\mathbf{y},b_1\wedge\cdots\wedge b_k)\rangle
\\\\
=\sum_{\sigma(\und{b})=\mathbf{p}\und{c}\mathbf{s}}\langle(\mathbf{x},a_1\wedge\cdots\wedge a_k)^*,
 (M_\sigma \mathbf{y} + \mathbf{l}(\mathbf{p}), \und{c} )\rangle.
\end{array}
$$
A term in this sum is non-zero only for those faces $(M_\sigma \mathbf{y} + \mathbf{l}(\mathbf{p}),\und{c})$ such that $M_\sigma \mathbf{y} + \mathbf{l}(\mathbf{p}) =\mathbf{x}$ and $\und{c} = a_{\tau(1)}\wedge\cdots\wedge a_{\tau(k)}=:\und{a}_\tau$ for some permutation $\tau$ on $\{1,\ldots,k\}$. The corresponding term is then equal to $\textrm{sgn}(\tau)$. 
Thus, denoted $\mathbf{z} = M_\sigma^{-1}(\mathbf{x} - \mathbf{l}(\mathbf{p}))$, this means that 
$\text{sgn}(\tau)(\mathbf{z},b_1\wedge\cdots\wedge b_k)^*$ appears in the image $\mathbf{E}_k^*(\sigma)(\mathbf{x},\und{a})^*$, with 
$\sigma(\und{b}) = \mathbf{p}\und{a}_\tau\mathbf{s}$.
Now we can reorder it into $(\mathbf{z},b_{\tau(1)}\wedge\cdots\wedge b_{\tau(k)})^*$ with $\sigma(\und{b}_\tau) = \mathbf{p}\und{a}\mathbf{s}$
 and rename. Hence we obtain
\[
\begin{array}{l}
\mathbf{E}_k^*(\sigma)(\mathbf{x},a_1\wedge\cdots\wedge a_k)^*
=\sum_{\sigma(\und{b})=\mathbf{p}\und{a}\mathbf{s}}
 (M_\sigma^{-1}(\mathbf{x}-\mathbf{l}(\mathbf{p})), \und{b} )^*.
\end{array} 
\qedhere 
\]
\end{proof}

\subsection{Poincar\'e maps and dual substitutions}

The dual maps are abstract objects formally defined on the dual basis, which has no geometric interpretation. We will interpret geometrically duals of faces of dimension $k$ as faces of dimension $n-k$, in a Poincar\'e duality flavor (cf.~\cite{SAI,AFHI}). 

\begin{definition}
The map $\varphi_k:C_k^*\to C_{n-k}$ is defined by
\[ \varphi_k:  (\mathbf{x},a_1\wedge\cdots\wedge a_k)^* \mapsto (-1)^{a_1+\cdots+a_k}(\mathbf{x}+\mathbf{e}_{a_1}+\cdots+\mathbf{e}_{a_k},b_1\wedge \cdots\wedge b_{n-k}) \]
where $\{a_1,\ldots,a_k\}$ and $\{b_1,\ldots,b_{n-k}\}$ form a partition of $\{1,2,\ldots,n\}$ with $a_1<\cdots< a_k$, $b_1<\cdots< b_{n-k}$. If $\und{a}=a_1\wedge\cdots\wedge a_k$ we will write $\und{a}^*= b_1\wedge\cdots\wedge b_{n-k}$. We call $(\mathbf{x},\und{a}^*)$ the $(n-k)$-dimensional face \emph{transverse} to $\und{a}$.
\end{definition}
It was shown in \cite{SAI} that this map commutes with the boundary and coboundary operators (see Section~\ref{sec:boundop}). Furthermore $\varphi_k$ is invertible.
Now we can conjugate the dual maps by $\varphi_k$ to obtain explicit geometric realizations.

\begin{definition}\label{defgeodual}
The \emph{geometric dual map} $\mathbf{E}^k(\sigma)$ is defined as 
\[ \mathbf{E}^k(\sigma)\circ \varphi_{n-k} = \varphi_{n-k} \circ \mathbf{E}^*_{n-k}(\sigma). \]
\end{definition}

In general we have $n = \#\mathcal{A} \geq d =\deg(\beta)$ with strict inequality for reducible substitutions. We want to represent the action of the substitution geometrically on $\mathbb{K}_c \cong \mathbb{R}^{d-1}$, thus it makes sense to consider the action of a dual substitution on $(d-1)$-dimensional faces. For this reason we will work with the geometric realization $\mathbf{E}^{d-1}(\sigma)$ conjugate to $\mathbf{E}_{n-d+1}^*(\sigma)$ via $\varphi_{n-d+1}$ (note that if $\sigma$ is irreducible then the dual substitution is $\mathbf{E}_1^*(\sigma)$). To be more concise we will denote $\bar n = n-d+1$.
An explicit formula for $\mathbf{E}^{d-1}(\sigma)$ reads as follows.

\begin{proposition}
Let $\mathbf{x}\in\mathbb{Z}^n$ and $\und{a}=a_1\wedge\cdots\wedge a_{\bar n}\in\mathcal{O}_{\bar n}$. Then 
the following holds for the $(d-1)$-dimensional face $(\mathbf{x},\und{a}^*)$ transverse to $\und{a}$. 
\begin{align} \label{eq:Edm1} \index{dual substitution!$\mathbf{E}^{d-1}(\sigma)$}
\mathbf{E}^{d-1}(\sigma)(\mathbf{x},\und{a}^*) &= \sum_{\tau\in\mathrm{Per}_{\bar n}}\sum_{\sigma(\und{b})=\mathbf{p}\und{a}_\tau\mathbf{s}}\mathrm{sgn}(\tau) (-1)^{\und{a}+\und{b}}\big(M_\sigma^{-1}(\mathbf{x} + \mathbf{l}(\mathbf{s})),\und{b}^* \big) 
\end{align}
where $(-1)^{\und{a}}$ denotes $(-1)^{a_1+\cdots+a_{\bar n}}$. In this formula, $\und{b}=b_1\wedge\cdots\wedge b_{\bar n}\in\mathcal{O}_{\bar n}$.
\end{proposition}
\begin{proof}
A face $(\mathbf{x},\und{a}^*)$ is sent by $\varphi_{\bar n}^{-1}$ to $(-1)^{\und{a}}(\mathbf{x}-(\mathbf{e}_{a_{1}}+\cdots+\mathbf{e}_{a_{\bar n}}),\und{a})^*$. 

Applying $\mathbf{E}_{\bar n}^*(\sigma)$ we get the sum of all elements
\[ \textrm{sign}(\tau)
(-1)^{\und{a}}\big(M_\sigma^{-1}(\mathbf{x}-(\mathbf{e}_{a_1}+\cdots+\mathbf{e}_{a_{\bar n}})-(\mathbf{l}(p_1)+\cdots+\mathbf{l}(p_{\bar n}))),\und{b}\big)^*,\] 
for $\tau\in\textrm{Per}_{\bar n}$, $\und{b}=b_1\wedge\cdots\wedge b_{\bar n}\in\mathcal{O}_{\bar n}$ and $\sigma(\und{b})=\mathbf{p}\und{a}_\tau\mathbf{s}$. 

Applying now $\varphi_{\bar n}$ leads to the sum of all elements
\begin{multline*}
\textrm{sgn}(\tau)
(-1)^{\und{a}+\und{b}}\Big(M_\sigma^{-1}\big(\mathbf{x}-(\mathbf{l}(p_1a_{\tau(1)})+\cdots+\mathbf{l}(p_{\bar n}a_{\tau(\bar n)})) + \\
+M_\sigma(\mathbf{e}_{b_1}+\cdots+\mathbf{e}_{b_{\bar n}})\big),\und{b}^*\Big), 
\end{multline*}
again for $\tau\in\textrm{Per}_{\bar n}$ and $\sigma(\und{b})=\mathbf{p}\und{a}_\tau\mathbf{s}$.

Therefore, since 
$$M_\sigma \mathbf{e}_{b_i} = \mathbf{l}(\sigma(b_i)) = \mathbf{l}(p_ia_{\tau(i)}) + \mathbf{l}(s_i)$$ holds for all $i$,
we get the result.
\end{proof}

Similar formulas hold for general $\mathbf{E}^k(\sigma)$. Notice that, since the geometric dual substitution $\mathbf{E}^k(\sigma)$ is conjugate to $\mathbf{E}_{n-k}^*(\sigma)$, it still depends on wedges of $n-k$ letters (see Equation \eqref{eq:Edm1} for $k=d-1$). 

\subsection{Matrices}
We define here abelianizations of the $k$-dimensional extension $\mathbf{E}_k(\sigma)$, of its dual map $\mathbf{E}_k^*(\sigma)$ and of the geometric realization $\mathbf{E}^k(\sigma)$ of the dual map $\mathbf{E}_{n-k}^*(\sigma)$ for a substitution $\sigma$. To this effect, we order the elements of $\mathcal{O}_k$ lexicographically and define the following matrices with entries indexed by pairs $(\und{b},\und{a})$ of elements of $\mathcal{O}_k$.

\subsubsection*{Matrix $B_k$ for $\mathbf{E}_k(\sigma)$}
For $\und{a}\in\mathcal{O}_k$, we rewrite \eqref{eq:Ek} in such a way that $\und{b}\in\mathcal{O}_k$ for all faces appearing in the formula:
$$\mathbf{E}_k(\sigma)(\mathbf{x},\und{a}) = \\ 
\sum_{\und{b}\in\mathcal{O}_k}\sum_{\tau\in\textrm{Per}_k}\sum_{\sigma(\und{a})=\mathbf{p}\und{b}_\tau\mathbf{s}} \textrm{sgn}(\tau)(M_\sigma \mathbf{x} + \mathbf{l}(\mathbf{p}), \und{b} ).
$$ 
Then 
$$(B_k)_{\und{b},\und{a}}:=\sum_{\tau\in\textrm{Per}_k}\sum_{\sigma(\und{a})=\mathbf{p}\und{b}_\tau\mathbf{s}} \textrm{sgn}(\tau)=:|\mathbf{E}_k(\sigma)(\mathbf{0},\und{a})|_{\und{b}} \in\mathbb{Z}
$$
Here $\lvert\cdot\rvert_{\und{b}}$ can be seen as the number of times a face of a given type $\und{b}\in\mathcal{O}_k$ occurs (this number may be negative). Note that $B_1=M_\sigma$.  More generally, we have
$$B_k=\bigwedge_{i=1}^kM_\sigma,
$$
the $k$-th exterior product of the matrix $M_\sigma$, obtained by considering all the minors of order $k$ of the matrix $M_\sigma$. More precisely, the coefficient $(\und{b}=b_1\wedge \cdots\wedge b_k,\und{a}=a_1\wedge\cdots\wedge a_k)$ of $\bigwedge_{i=1}^kM_\sigma$ is the minor obtained from $M_\sigma$ by keeping the rows $b_1,\ldots,b_{k}$ and the columns $a_1,\ldots,a_{k}$ of $M_\sigma$. The eigenvalues of $\bigwedge_{i=1}^kM_\sigma$ are the products of $k$ distinct eigenvalues of $M_\sigma$ (counted with their multiplicity).

\subsubsection*{Matrix $M_k^*$ for the dual $\mathbf{E}_k^*(\sigma)$}
Similarly, we define
$$(M_k^*)_{\und{b},\und{a}}:=\sum_{\tau\in\textrm{Per}_k}\sum_{\sigma(\und{b})=\mathbf{p}\und{a}_\tau\mathbf{s}} \textrm{sgn}(\tau)=:|\mathbf{E}_k^*(\sigma)(\mathbf{0},\und{a})^*|_{\und{b}} \in\mathbb{Z},
$$
and obviously we have the relation
$$M_k^*={}^tB_k=\bigwedge_{i=1}^k{}^tM_\sigma.
$$
Note that here $\lvert\cdot\rvert_{\und{b}}$ counts the number of times a dual face of type $\und{b}$ occurs.

\subsubsection*{Matrix $M_k$ for the geometric dual map $\mathbf{E}^k(\sigma)$}
We finally define the matrix $M_k$ associated with the operator $\mathbf{E}^k(\sigma)$ by 
\begin{equation} \label{eq:Mk}
(M_k)_{\und{b}^*,\und{a}^*} :=\sum_{\tau\in\textrm{Per}_{n-k}}\sum_{\sigma(\und{b})=\mathbf{p}\und{a}_\tau\mathbf{s}}\textrm{sgn}(\tau)(-1)^{\und{a}+\und{b}}
=: |\mathbf{E}^k(\sigma)(\mathbf{0},\und{a}^*)|_{\und{b}^*} \in\mathbb{Z}
\end{equation}
for $\und{a},\und{b}\in\mathcal{O}_{n-k}$. In other words, 
\begin{equation} \label{eq:MkRel}
(M_k)_{\und{b}^*,\und{a}^*}=(-1)^{\und{a}+\und{b}}\left(M_{n-k}^*\right)_{\und{b},\und{a}}=(-1)^{\und{a}+\und{b}}\left(\bigwedge_{i=1}^{n-k}{}^tM_\sigma\right)_{\und{b},\und{a}}.
\end{equation}
For this reason, the matrix $M_k$ will be written down with respect to the basis $\{\und{a}^*;\und{a}\in\mathcal{O}_{n-k}\}$, ordered in the same order as the elements of $\mathcal{O}_{n-k}$ (see Example~\ref{ex:tribo}). 

We are mostly interested in $\mathbf{E}^{d-1}(\sigma)$ and $\mathbf{E}^{d-2}(\sigma)$. The matrices $M_{d-1}$ and $M_{d-2}$ describe respectively the growth of $(d-1)$ and $(d-2)$-dimensional faces in the images of $\mathbf{E}^{d-1}(\sigma)$ and $\mathbf{E}^{d-2}(\sigma)$.

\medskip

\begin{example}\label{ex:tribo}
Let $\sigma:1\mapsto 12,\, 2\mapsto 13,\, 3\mapsto 1$ be the Tribonacci substitution. Then
\begin{alignat*}{2}
\mathbf{E}^2(\sigma) 
& : &\,  (\mathbf{0},2\wedge 3) &\mapsto (M^{-1}_\sigma\mathbf{e}_2,2\wedge 3) -(M^{-1}_\sigma\mathbf{e}_3,1\wedge 3) + (\mathbf{0},1\wedge 2),
\\
& &\, (\mathbf{0},1\wedge 3) &\mapsto -(\mathbf{0},2\wedge 3),   \\
& & \,(\mathbf{0},1\wedge 2) &\mapsto -(\mathbf{0},1\wedge 3),   \\
\mathbf{E}^1(\sigma) 
& : &\,  (\mathbf{0},3) &\mapsto -(M^{-1}_\sigma\mathbf{e}_3,3) + (\mathbf{0},2)\\
& &\, (\mathbf{0},2) &\mapsto -(M^{-1}_\sigma\mathbf{e}_2,3) + (\mathbf{0},1)  \\
& &\,(\mathbf{0},1) &\mapsto (\mathbf{0},3) 
\end{alignat*}
with matrices
\[
M_2 = \Bigl(\begin{smallmatrix} 1 & -1 & 0 \\ -1 & 0 & -1 \\ 1 & 0 & 0 \end{smallmatrix}\Bigr), \qquad   
M_1 = \Bigl(\begin{smallmatrix} -1& -1 & 1 \\ 1 & 0 & 0 \\ 0 & 1 & 0 \end{smallmatrix}\Bigr).
\]
The matrix $M_2$ is written with respect to the basis $\{1^*,2^*,3^*\}$, and similarly $M_1$ is written with respect to the basis $\{(1\wedge2)^*,(1\wedge3)^*,(2\wedge3)^*\}$. One can check \eqref{eq:MkRel} by comparing these matrices with 
\[
M^*_{1} = {}^tM_\sigma = \Bigl(\begin{smallmatrix} 1 & 1 & 0 \\ 1 & 0 & 1 \\ 1 & 0 & 0 \end{smallmatrix}\Bigr), \qquad   
M^*_{2} = {}^tM_\sigma\wedge {}^tM_\sigma = \Bigl(\begin{smallmatrix} -1& 1 & 1 \\ -1 & 0 & 0 \\ 0 & -1 & 0 \end{smallmatrix}\Bigr).
\]
\end{example}

\noindent {\bf Main example.} 
We consider the one-parameter family of unit reducible Pisot substitutions defined in Section~\ref{sec:main-ex}. Since $\#\mathcal{A}=5$ and $\deg(\beta)=3$ we will consider the geometric dual substitution $\mathbf{E}^2(\sigma_t)$ conjugate to $\mathbf{E}_3^*(\sigma_t)$. We have $\mathcal{O}_3=\{1\wedge2\wedge 3,\ldots,3\wedge4\wedge5\}$.

For $\sigma_t$ ($t\geq 0$), the dual maps are given explicitly by the following images of the basis $\{(\mathbf{0},\und{a}^*);\und{a}\in\mathcal{O}_{3}\}$, written in the lexicographic order of $\mathcal{O}_{3}$.

\begin{align} \label{eq:E2sigmat}
\mathbf{E}^2(\sigma_t): (\mathbf{0},4\wedge 5) &\mapsto -\sum_{j=1}^{t}(M_{\sigma_t}^{-1}((t-j)\mathbf{e}_1+\mathbf{e}_5),3\wedge 5) \;+\; (\mathbf{0},3\wedge 4)  \\  
(\mathbf{0},3\wedge 5) &\mapsto -\sum_{j=1}^{t}(M_{\sigma_t}^{-1}((t-j)\mathbf{e}_1+\mathbf{e}_5),2\wedge 5)\;+\;(\mathbf{0},2\wedge 4)  \nonumber \\
(\mathbf{0},3\wedge 4) &\mapsto (\mathbf{0},2\wedge 3) \nonumber  \\  
(\mathbf{0},2\wedge 5) &\mapsto \sum_{j=1}^{t+1}(M_{\sigma_t}^{-1}((t+1-j)\mathbf{e}_1+\mathbf{e}_2),4\wedge 5) \nonumber \\
&\;\; -\sum_{j=1}^{t}(M_{\sigma_t}^{-1}((t-j)\mathbf{e}_1+\mathbf{e}_5),1\wedge 5)
\;+\; (\mathbf{0},1\wedge 4) \nonumber \\  
(\mathbf{0},2\wedge 4) &\mapsto \sum_{j=1}^{t+1}(M_{\sigma_t}^{-1}((t+1-j)\mathbf{e}_1+\mathbf{e}_2),3\wedge 5) \;+\;(\mathbf{0},1\wedge 3) \nonumber \\  
(\mathbf{0},2\wedge 3) &\mapsto \sum_{j=1}^{t+1}(M_{\sigma_t}^{-1}((t+1-j)\mathbf{e}_1+\mathbf{e}_2)),2\wedge 5) \;+\; (\mathbf{0},1\wedge 2) \nonumber \\
(\mathbf{0},1\wedge 5) &\mapsto -(\mathbf{0},4\wedge 5) \nonumber \\ 
(\mathbf{0},1\wedge 4) &\mapsto -(\mathbf{0},3\wedge 5) \nonumber \\ 
(\mathbf{0},1\wedge 3) &\mapsto -(\mathbf{0},2\wedge 5) \nonumber \\ 
(\mathbf{0},1\wedge 2) &\mapsto -(\mathbf{0},1\wedge 5)\nonumber
\end{align}

The associated matrices are
$$M_2(\sigma_t)=\left(\begin{smallmatrix}
0&0&0&t+1&0&0&-1&0&0&0\\
-t&0&0&0&t+1&0&0&-1&0&0\\
1&0&0&0&0&0&0&0&0&0\\
0&-t&0&0&0&t+1&0&0&-1&0\\
0&1&0&0&0&0&0&0&0&0\\
0&0&1&0&0&0&0&0&0&0\\
0&0&0&-t&0&0&0&0&0&-1\\
0&0&0&1&0&0&0&0&0&0\\
0&0&0&0&1&0&0&0&0&0\\
0&0&0&0&0&1&0&0&0&0
\end{smallmatrix}\right)
$$
One can check that, for all $t\geq 0$, $M_3(\sigma_t)^*={}^tM_{\sigma_t}\wedge{}^tM_{\sigma_t}\wedge{}^tM_{\sigma_t}=\left|M_2(\sigma_t)\right|$ is a primitive matrix.

\begin{remark}\label{rmk:matrixconj}
Since by definition $\mathbf{E}^{d-1}(\sigma)$ is conjugate to $\mathbf{E}_{\bar n}^*(\sigma)$ via $\varphi_{\bar n}$, the matrix $M_{d-1}$ is conjugate to $M^*_{\bar n}$, i.e.~there exists $N\in \text{GL}(\binom{n}{d-1},\mathbb{Z})$ such that 
\[ M^*_{\bar n} = N^{-1} M_{d-1} N.\]
Furthermore we can choose $N$ to be a diagonal matrix with entries in $\{1,-1\}$. A similar approach treating $N$ was considered in \cite{FIR}.
\end{remark}

\subsection{Positivity and cancellation}\label{sec:canc}

Even if $M_\sigma$ is a non-negative matrix, we have seen that the matrices $M^*_k$ and $M_{n-k}$ can have negative entries. This implies that cancellation may occur for $\mathbf{E}_{k}^*(\sigma)$ and $\mathbf{E}^{n-k}(\sigma)$. 

There are two types of cancellation, as explained as follows.
By definition the abelianization $M^*_{k}$ counts the elements occurring in the image of the dual $\mathbf{E}^*_{k}(\sigma)$, ``forgetting'' their base points. In particular, if two elements $(\mathbf{x},\und{b})^*$ and $-(\mathbf{y},\und{b})^*$ with $\mathbf{x}\neq\mathbf{y}$ occur in the image $\mathbf{E}^*_{k}(\sigma)(\mathbf{0},\und{a})^*$, then they cancel in $M^*_{k}$, i.e.~$(M^*_k)_{\und{b},\und{a}} = 0$. This is an example of ``bad cancellation'': geometrically, the two elements should not cancel, as they both contribute to the total Lebesgue measure of $\mathbf{E}^{n-k}(\sigma)\varphi_{k}(\mathbf{0},\und{a})^*$. The only ``good cancellation'' that $M^*_k$ takes into account is the cancellation of two elements based at the same point and with the same type but opposite orientations. We rather wish to avoid bad cancellation in order to see  $M^*_k$ as a faithful algebraic description of the substitution $\mathbf{E}^*_{k}(\sigma)$.

We may be able to define $\mathcal{O}_k$ differently (i.e. without respecting the lexicographic order) in a way that each element of $\mathcal{O}_k$ is sent by $\mathbf{E}^*_k(\sigma)$ to a sum of  elements of $\mathcal{O}_k$. We then say that $\mathbf{E}^*_k(\sigma)$ is \emph{positive} (with respect to $\mathcal{O}_k$). If this is the case then $M^*_k$ is non-negative, no cancellation occurs at all, and $\mathbf{E}^*_k(\sigma)$ behaves as a substitution. This is the concept of {\it positivity} defined in \cite{AFHI} and which will motivate our hypothesis (P) in Section~\ref{sec:hyp}.

An easy calculation based on the conjugacy between $\mathbf{E}^*_k(\sigma)$ and $\mathbf{E}^{n-k}(\sigma)$ shows that if $\mathbf{E}^*_k(\sigma)$ is positive, then for every $\und{a}\in \mathcal{O}_{n-k}$, the faces of type $\und{b}$ occurring in $\mathbf{E}^{n-k}(\sigma)(\mathbf{0},\und{a})$ are all positively or all negatively oriented. This implies that no cancellation happens neither for $\mathbf{E}^{n-k}(\sigma)$.

\subsection{Boundary and coboundary operators}\label{sec:boundop}

All this can be done as in the classical simplicial homology and cohomology theory (see e.g.~\cite{Hat}). The following considerations can be found in~\cite{SAI}.  

There is a boundary operator which associates with a $k$-dimensional face its boundary consisting in a union of oriented $(k-1)$-dimensional faces. 
\begin{definition}\label{def:boundop}
The boundary operator $\partial:\cup_{k=1}^nC_k\to\cup_{k=0}^{n-1}C_k$ is defined on the basis elements $(\mathbf{x},\und{a})$ with $\mathbf{x}\in\mathbb{R}^n$ and $\und{a}=a_1\wedge\cdots\wedge a_k\in\mathcal{O}_k$, $1\leq k\leq n$, by
$$\begin{array}{rcl}
\partial(\mathbf{x},\und{a})&=&\sum_{i=1}^k(-1)^i\left((\mathbf{x},a_{1}\wedge\cdots\wedge \widehat{a_i}\wedge\cdots\wedge a_k)\right.\\
&&\left.-(\mathbf{x}+\mathbf{e}_{a_i},a_{1}\wedge\cdots\wedge \widehat{a_i}\wedge\cdots\wedge a_k)\right)
\end{array}
$$
Here, $a_{1}\wedge\cdots\wedge \widehat{a_i}\wedge\cdots\wedge a_k=a_{1}\wedge\cdots\wedge a_{i-1}\wedge a_{i+1}\wedge\cdots\wedge a_k$.  For $k=1$, we simply write $\widehat{a_i}:=\bullet$.
\end{definition}

By duality, a coboundary operator $\partial^*:\cup_{k=0}^{n-1}C_k^*\to\cup_{k=1}^{n}C_k^*$ acting on duals of faces can be defined as well. An explicit formula is given in \cite[Proposition 1.1]{SAI}. Moreover, the Poincar\'e maps conjugate the boundary and coboundary operators:
\begin{equation}\label{eq:poincbocobo}
\varphi_k\circ \partial^*=\partial \circ \varphi_{k-1}.
\end{equation}
An important property is that the boundary and coboundary operators commute with $\mathbf{E}_k(\sigma)$ and $\mathbf{E}^*_k(\sigma)$ respectively. More precisely, 
\begin{align}
\partial\circ\mathbf{E}_k(\sigma) &= \mathbf{E}_{k-1}(\sigma)\circ\partial, \\
\partial^*\circ\mathbf{E}^*_k(\sigma) &= \mathbf{E}^*_{k+1}(\sigma)\circ\partial^*. \label{boundcomm}
\end{align} 
Notice that positivity is not preserved under application of $\partial$ or $\partial^*$.

To derive geometric properties of the Rauzy fractals associated with a substitution, we will be interested in the geometric realizations $\mathbf{E}^{d-1}(\sigma)$ and $\mathbf{E}^{d-2}(\sigma)$ of the duals of a substitution. Note that it follows from Equation~\eqref{boundcomm}, \eqref{eq:poincbocobo} and Definition~\ref{defgeodual} that
\begin{equation}\label{eq:relgeomreal}
\partial\circ\mathbf{E}^{d-1}(\sigma) = \mathbf{E}^{d-2}(\sigma)\circ\partial.  
\end{equation}

We will see in Section~\ref{sec:infsub} that the formula \eqref{eq:relgeomreal} will play a prominent role in describing $\mathbf{E}^{d-1}(\sigma)$ as an inflate-and-subdivide rule.

\section{Stepped surfaces}
\label{sec:step}

In this section, we give the general construction for the stepped surface of a unimodular Pisot substitution. We introduce in Section~\ref{sec:hyp} and \ref{sec:geomfinprop} some fundamental properties in order for the stepped surface of the substitution to be well-defined and have the expected properties.
We use some terminology of \cite{AFHI}. Recall that $n=\#\mathcal{A}$ and $d=\deg(\beta)$ and, from Section~\ref{sec:faces}, that a geometric element of $C_{d-1}$ is an element $\sum_{l=1}^Nn_{l}(\mathbf{x}_l,\und{a}_l)$ with $n_{l}\in\{-1,1\}$ for all $l\in\{1,\ldots,N\}$.

\begin{definition}
A geometric element $P=\sum_{l=1}^Nn_{l}(\mathbf{x}_l,\und{a}_l)$ is said to \emph{project well onto} $\K_c$ if the projections $\pi_c\overline{(\mathbf{x}_l,\und{a}_l)}$, $\pi_c\overline{(\mathbf{x}_{l'},\und{a}_{l'})}$ of distinct faces $(\mathbf{x}_l,\und{a}_l)$, $(\mathbf{x}_{l'},\und{a}_{l'})$ ($l\ne l'$)  of $P$  have mutually disjoint interiors. 

A \emph{tiling} of $\mathbb{K}_c$ is a covering of $\mathbb{K}_c$ by compact sets such that Lebesgue almost every point of $\mathbb{K}_c$ is contained in exactly one element of the covering.

A \emph{stepped surface} $\Gamma$ is a union of $(d-1)$-dimensional faces such that $\pi_c(\overline{\Gamma})$ is a (polygonal) tiling of $\K_c$.
\end{definition}

We look for a stepped surface invariant under (a power of) $\mathbf{E}^{d-1}(\sigma)$.
Let $\mathcal{U}\in C_{d-1}$ be a geometric element of $(d-1)$-dimensional faces based at $\mathbf{0}$. Suppose that $\mathcal{U}$ projects well and that there exists an integer $m\geq 1$ such that $\mathbf{E}^{d-1}(\sigma)^m(\mathcal{U})$ is geometric and satisfies $\{\mathcal{U} \}\subseteq \{\mathbf{E}^{d-1}(\sigma)^m(\mathcal{U})\}$. We associate to such a geometric element  $\mathcal{U}$ the following potential candidate for a stepped surface:
\begin{equation}\label{gamma}
\Gamma_\mathcal{U} := \bigcup_{k\geq 0} \{\mathbf{E}^{d-1}(\sigma)^{mk}(\mathcal{U})\}. 
\end{equation}

\subsection{Faces near to $\mathbb{K}_c$}
Given two vectors $\mathbf{z}_1, \mathbf{z}_2\in\mathbb{R}^n$ we denote by $[\mathbf{z}_1,\mathbf{z}_2)$ the set $\{(1-t)\,\mathbf{z}_1 + t\mathbf{z}_2: t \in [0,1)\}$.

Let
\[ \mathcal{S}^*=\{(\mathbf{x},\und{a})^*: \mathbf{x} \in \mathbb{Z}^n, \und{a}\in \mathcal{O}_{\bar n},  \pi_e(\mathbf{x}) \in [-\pi_e(\mathbf{l}(\und{a})),\mathbf{0}) \}. \]
Recall that $\bar n = n-d+1$.

\begin{proposition}\label{invS} The inclusion
$\{\mathbf{E}^*_{\bar n}(\sigma)(\mathcal{S}^*)\} \subseteq \mathcal{S}^*$ holds.
\end{proposition}
\begin{proof}
We show that an element $(\mathbf{y},\und{b})^*\in \{\mathbf{E}^*_{\bar n}(\sigma)(\mathcal{S}^*)\}$ is in $\mathcal{S}^*$. We have that $(\mathbf{y},\und{b})^*$ is of the form $(M_\sigma^{-1}(\mathbf{x}-\mathbf{l}(\und{p})),\und{b})^*$ for some $(\mathbf{x},\und{a})^*\in\mathcal{S}^*$ and $\sigma(\und{b}) = \mathbf{p}\und{a}\mathbf{s}$. We must show that
\begin{align}\label{eq:comp}
\pi_e(M_\sigma^{-1}(\mathbf{x}-\mathbf{l}(\mathbf{p}))) \in [-\pi_e(\mathbf{l}(\und{b})),\mathbf{0}) &\Longleftrightarrow \pi_e(\mathbf{x}-\mathbf{l}(\mathbf{p})) \in [-M_\sigma\pi_e(\mathbf{l}(\und{b})),\mathbf{0}) \\
&\Longleftrightarrow \pi_e(\mathbf{x}-\mathbf{l}(\mathbf{p})) \in [-\pi_e(\mathbf{l}(\sigma(\und{b}))),\mathbf{0}) \notag \\
&\Longleftrightarrow \pi_e(\mathbf{x}) \in [-\pi_e(\mathbf{l}(\sigma(\und{b}))-\mathbf{l}(\mathbf{p})),\pi_e\mathbf{l}(\mathbf{p})) \notag \\
&\Longleftrightarrow \pi_e(\mathbf{x}) \in [-\pi_e(\mathbf{l}(\und{a}\mathbf{s})),\pi_e\mathbf{l}(\mathbf{p})) \notag
\end{align}
and this is true since $\pi_e(\mathbf{x})\in [-\pi_e\mathbf{l}(\und{a}),\mathbf{0})\subseteq [-\pi_e(\mathbf{l}(\und{a}\mathbf{s})),\pi_e\mathbf{l}(\mathbf{p}))$.
\end{proof}

\begin{remark}
Note that the inclusion $\{\mathbf{E}^*_{\bar n}(\sigma)(\mathcal{S}^*)\} \supseteq \mathcal{S}^*$ is much more complicated to show since one needs to control the cancellation of faces originating from the wedge formalism.
\end{remark}

\begin{definition} 
The set of \emph{faces near to} $\mathbb{K}_c$ is $\mathcal{S}:=\{\varphi_{\bar n}(\mathcal{S}^*)\}$.
\end{definition}
It follows from Definition \ref{defgeodual} and Proposition~\ref{invS} that $\mathcal{S}$ is invariant under $\mathbf{E}^{d-1}(\sigma)$, i.e. $\{\mathbf{E}^{d-1}(\sigma)(\mathcal{S})\} \subseteq \mathcal{S}$.

There is no clear interpretation of $\mathcal{S}$ as a stepped surface. We will show in Theorem~\ref{thm:poly} that, under certain hypotheses, $\Gamma_\mathcal{U}\subset \mathcal{S}$ is a stepped surface. This set depends heavily on the starting $\mathcal{U}$.
Intuitively we have that the iterations of $\mathbf{E}^{d-1}(\sigma)$ on $\mathcal{U}$ filter some elements of $\mathcal{S}$ which, seen as colored points near to $\mathbb{K}_c$, generate a good translation set for a tiling, namely a \emph{Delone set}, i.e.~a uniformly discrete and relatively dense set.

\begin{remark}
In the irreducible settings, it is proven in \cite[Lemma~2, Lemma~3]{AI} that the stepped surface is invariant under $\mathbf{E}_1^*(\sigma)$ and that two distinct faces have disjoint images (as sets of faces). In our case, the invariance of $\mathcal{S}^*$ is clear by Proposition~\ref{invS}, however the following examples show that 
two distinct faces  outside of $\mathcal{S}^*$ as well as in $\mathcal{S}^*$ may fail to have disjoint images by $\mathbf{E}^*_{\bar n}(\sigma)$.

Consider $\sigma=\sigma_1$ of the family defined in Section~\ref{sec:main-ex}. Then
\[ (-M_\sigma^{-1}\mathbf{l}(11),1\wedge 2\wedge 4)^* \in \{\mathbf{E}^*_{3}(\mathbf{0},1\wedge 3 \wedge 5)^*\} \cap \{\mathbf{E}^*_{3}(\mathbf{0},2\wedge 3 \wedge 1)^*\}.   \]
since the images
\begin{align*}
 \mathbf{E}^*_{3}(\mathbf{0},1\wedge 3 \wedge 5)^* &= (-M_\sigma^{-1}\mathbf{l}(1),1\wedge 2\wedge 4)^* + (-M_\sigma^{-1}\mathbf{l}(11),1\wedge 2\wedge 4)^* \\ &\quad+ (-M_\sigma^{-1}\mathbf{l}(1),5\wedge 2\wedge 4)^*,  \\
  \mathbf{E}^*_{3}(\mathbf{0},2\wedge 3 \wedge 1)^* &=  (-M_\sigma^{-1}\mathbf{l}(11),1\wedge 2\wedge 4)^* + (-M_\sigma^{-1}\mathbf{l}(11),1\wedge 2\wedge 5)^*.
\end{align*}
Notice that $(\mathbf{0},1\wedge 3 \wedge 5)^*, (\mathbf{0},2\wedge 3 \wedge 1)^*\notin\mathcal{S}^*$.
On the other hand, if we consider $\mathbf{E}^2(\sigma)$ then 
\begin{align*}
 \mathbf{E}^2(\mathbf{0},2\wedge 4) &= (M_\sigma^{-1}\mathbf{l}(12),3\wedge 5) + (M_\sigma^{-1}\mathbf{l}(2),3\wedge 5) + (\mathbf{0},1\wedge 3)  \\
  \mathbf{E}^2(\mathbf{0},4\wedge 5) &=  (M_\sigma^{-1}\mathbf{l}(5),3\wedge 5) + (\mathbf{0},3\wedge 4)
\end{align*}
and the two images are disjoint. 

As a second example, we consider the elements $(-M_{\sigma}\mathbf{e}_4,1\wedge 3\wedge 4)^*$, $(-M_{\sigma}\mathbf{e}_4+\mathbf{e}_1,1\wedge 3\wedge 4)^*\in\mathcal{S}^*$. Then
\[ 
(-\mathbf{e}_4,1\wedge 2\wedge 3)^*\in \{\mathbf{E}^*_{3}(\sigma)(-M_{\sigma}\mathbf{e}_4,1\wedge 3\wedge 4)^*\} \cap \{\mathbf{E}^*_{3}(\sigma)(-M_{\sigma}\mathbf{e}_4+\mathbf{e}_1,1\wedge 3\wedge 4)^*\}, \] in other words, the two images are not disjoint.
\end{remark}

\subsection{The hypotheses}\label{sec:hyp}

The following definition will be fundamental to achieve the next results.

\begin{definition}
We say that $\sigma$ is a \emph{nice reducible substitution} if:
\begin{enumerate}
\item[(S1)] \label{i:inflsubd1} The element $\mathbf{E}^{d-1}(\sigma)(U)$ is geometric and projects well, for each face $U\in\mathcal{S}$.

\item[(S2)] \label{i:inflsubd2} For any two distinct faces $U_1$, $U_2 \in \mathcal{S}$ such that $U_1+U_2$ projects well onto $\mathbb{K}_c$, then $\{\mathbf{E}^{d-1}(\sigma)(U_1)\}$ and $\{\mathbf{E}^{d-1}(\sigma)(U_2)\}$ are disjoint and $\mathbf{E}^{d-1}(\sigma)(U_1 + U_2)$ projects well onto $\mathbb{K}_c$.

\item[(P)] \label{i:prim} 
The image by $\mathbf{E}^*_{\bar n}(\sigma)$ of a positive face is the sum of positive faces, and the matrix $M_{\bar n}^*=\bigwedge_{i=1}^{\bar n}\,{}^tM_\sigma$ is primitive.

\item[(N)] \label{i:neutral} The neutral polynomial $g(x)$ has only simple roots of modulus one and $g(0)=1$. 
\end{enumerate}
\end{definition}

(S1) and (S2) will assure that $\mathbf{E}^{d-1}(\sigma)$ satisfies an inflation-and-subdivision rule (see Proposition~\ref{subrule}) and play a prominent role in Theorem~\ref{thm:poly}. 

(P) is the concept of {\it positivity} introduced in \cite{AFHI} together with the primitivity of the matrix. It is thoroughly explained in Section~\ref{sec:canc} and implies, in particular, that $\mathcal{O}_{\bar n}$ is a convenient orientation for $M^*_{\bar n}$ to be non-negative and that no cancellation occurs neither for $\mathbf{E}^*_{\bar n}(\sigma)$ nor for $\mathbf{E}^{d-1}(\sigma)$.

(N) is an algebraic assumption on the neutral polynomial so that the dominant eigenvalue of $|M_{d-1}|$ is the Pisot root $\beta$ (see Lemma~\ref{le:modone}). 

\medskip

\noindent {\bf Main example.}
The substitutions of the family considered in Section~\ref{sec:main-ex} are all nice reducible substitutions.

\begin{proposition}\label{prop:nice}
For each $t$, $\sigma_t$ is a nice reducible substitution.
\end{proposition}

\begin{proof}
For each $\sigma_t$ the polynomial $g(x)=x^2-x+1$, thus (N) is true. 
We have that $\mathbf{E}^*_{3}(\sigma)$ is positive and $M^*_3=\wedge^3 \,({}^t M_\sigma)$ is primitive.

For (S1) we can look at the definition of $\mathbf{E}^2(\sigma_t)$ and check that the images of each face are geometric and project well. 
For (S2) we must check that for every pairs of possible neighboring faces $U_1,U_2$ then $\mathbf{E}^2(\sigma_t)(U_1+U_2)$ is geometric and projecting well.
This can be readily done for two neighboring faces centered at $\mathbf{0}$ using again the definition of $\mathbf{E}^2(\sigma_t)$. On the other hand there could be faces $U_1, U_2$ such that their supports intersect in only one point or whose supports are disjoint but with $\overline{\mathbf{E}^2(\sigma_t)(U_1)} \cap \overline{\mathbf{E}^2(\sigma_t)(U_2)} \neq \emptyset$.
For this reason, we must check this disjointness condition for all pair of faces near to $\mathbb{K}_c$, i.e.~in $\mathcal{S}$, inside the ball $B(\mathbf{0},R)$ where 
\[ R = \frac{2 \max_{\mathbf{s} : \sigma(\underline{a}) =\mathbf{p}\underline{b}\mathbf{s}}\{\pi_c(\mathbf{l}(\mathbf{s})) \}}{1 - \lVert\beta\rVert_c}  \]
where $\lVert\beta\rVert_c = \max\{|\beta'|: \beta' \text{ Galois conjugate of }\beta\}$. Indeed, two faces whose base points are at a distance greater than $R$ will have necessarily images with disjoints supports. Notice that $B(\mathbf{0},R)\cap\{\pi_c(\mathbf{x}):  (\mathbf{x},\und{a})\in \mathcal{S}\}$ is finite since $\{\pi_c(\mathbf{x}):  (\mathbf{x},\und{a})\in \mathcal{S}\}$ is a Delone set. Thus we have to check only finitely many possibilities. One can check by computation that, for each $\sigma_t$, $\mathbf{E}^2(\sigma_t)(U_1+U_2)$ is geometric and projects well for any two geometric and projecting well $U_1,U_2$.
\end{proof}

\subsection{The inflate-and-subdivide rule}\label{sec:infsub}

Define the absolute value of a matrix $M=(m_{ij})$ to be the matrix $|M|=(|m_{ij}|)$. 

\begin{lemma}\label{le:modone}
If $\sigma$ is a nice reducible substitution then the dominant eigenvalue of $|M_{d-1}|$ is $\beta$. 
\end{lemma}
\begin{proof}
We deduce from~(\ref{eq:MkRel}) and from (P) that 
\[ |M_{d-1}|=\left|\bigwedge_{i=1}^{\bar n}{}^tM_\sigma\right|=\bigwedge_{i=1}^{\bar n}{}^tM_\sigma = M^*_{\bar n}. \] 
The eigenvalues of $\bigwedge_{i=1}^{\bar n}{}^tM_\sigma$ are the products of $\bar n$ distinct eigenvalues of $M_\sigma$, counted with multiplicity. Since $M^*_{\bar n}$ is primitive, it has a dominant eigenvalue, which, by (N), is $\beta\prod_i \zeta_i=\beta$, where the $\zeta_i$ are the $n-d$ roots of the unimodular neutral polynomial $g(x)$. Indeed, all other products of $\bar n$ distinct eigenvalues of $M_\sigma$ are less than $\beta$, since they would include one Galois conjugate of the Pisot number $\beta$.
\end{proof}

\begin{lemma}\label{lem:meseigenvec}
If $\sigma$ is a nice reducible substitution then the vector $(\lambda(\pi_c\ov{(\mathbf{0},\und{a})}))_{\und{a}\in\mathcal{O}_{d-1}}$ is an eigenvector of $|{}^tM_{d-1}|$ for the eigenvalue $\beta$.
\end{lemma}

Note that a similar result was obtained in the irreducible case (see~\cite[p.197]{AI} or~\cite[Lemma 2.3]{IR06}).  The proof for the reducible case is more technical, it can be found in the Appendix. 
\medskip

We introduce now the important concept of inflate-and-subdivide rule. This is usually defined in the context of tiling spaces, which is slightly different than ours in terms of notations. 
A {\it tile} in $\mathbb{R}^k$ is defined as a pair $T=(F,t)$ where $F=\text{supp}(T)$ (the support of $T$) is a compact set in $\mathbb{R}^k$ which is the closure of its interior, and $t$ is the type of $T$. We deal with particular tiles, namely only with faces $(\mathbf{x},\und{a})$ whose support is denoted by $\overline{(\mathbf{x},\und{a})}$ and whose type is $\und{a}$. The following is the definition of inflate-and-subdivide rule adapted to our notations.

\begin{definition}\label{def:iasimp}
Let $\alpha$ be an expanding linear transformation of $\R^k$ (all its eigenvalues are greater than one in modulus). A geometric dual map $\omega$ acting on a subset of $C_k$ induces an \emph{inflate-and-subdivide rule} on $\mathbb{R}^k$ with expansion $\alpha$ if $\pi_k(\overline{\omega(P)}) = \alpha(\pi_k(\overline{P}))$ for every geometric element $P$ which projects well onto $\mathbb{R}^k$, and where $\pi_k$ denotes a projection of $\mathbb{R}^n$ to $\mathbb{R}^k$. It induces an \emph{imperfect inflate-and-subdivide rule} with expansion $\alpha$ if $\pi_k(\overline{\omega(P)}) \neq \alpha(\pi_k(\overline{P}))$ but $\lambda(\pi_k(\overline{\omega(P)})) = \lambda(\alpha(\pi_k(\overline{P})))$. 
\end{definition}

It is clear that $\mathbf{E}_1(\sigma)$ induces an inflate-and-subdivide rule on $\mathbb{K}_e$ with expansion $M_\sigma|_{\mathbb{K}_e}$.
The following proposition states that the geometric dual map $\mathbf{E}^{d-1}(\sigma)$ induces an imperfect inflate-and-subdivide rule on $\mathbb{K}_c$ with expansion $M_\sigma^{-1}|_{\mathbb{K}_c}$. 

\begin{proposition}\label{subrule}
If $\sigma$ is a nice reducible substitution then $\mathbf{E}^{d-1}(\sigma):\mathcal{S}\to \mathcal{S}$ induces an imperfect inflate-and-subdivide rule on $\K_c$ with expansion $\alpha = M_\sigma^{-1}|_{\mathbb{K}_c}$:
\[  \lambda(\pi_c\,\overline{\mathbf{E}^{d-1}(\sigma)(P)}) = \lambda(M_\sigma^{-1}\pi_c(\overline{P})) = \beta\,\lambda(\pi_c(\overline{P})) \]
holds for every geometric element  $P\subset\mathcal{S}$. 
\end{proposition}

\begin{proof}
We start to show that $\lambda(\pi_c\,\overline{\mathbf{E}^{d-1}(\sigma)(\mathbf{0},\und{a})}) = \lambda(M_\sigma^{-1}\pi_c\overline{(\mathbf{0},\und{a})})$ for all the faces $(\mathbf{0},\und{a})$.
By (S1) and (P), the element $\mathbf{E}^{d-1}(\sigma)(\mathbf{0},\und{a})$ is geometric and contains exactly $|(M_{d-1})_{\und{b},\und{a}}|=|({}^t M_{d-1})_{\und{a},\und{b}}|$ faces of type $\und{b}$. Since this element even projects well, we have that 
\begin{align}
\lambda(\pi_c\,\overline{\mathbf{E}^{d-1}(\sigma)(\mathbf{0},\und{a})}) &= \sum_{\und{b}\in\mathcal{O}_{d-1}} |({}^t M_{d-1})_{\und{a},\und{b}}|\, \lambda(\pi_c\overline{(\mathbf{0},\und{b})}) \nonumber \\
&=  \beta\,\lambda(\pi_c\overline{(\mathbf{0},\und{a})}) \label{infl} \\
&= \lambda(M_\sigma^{-1}\,\pi_c \overline{(\mathbf{0},\und{a})}). \nonumber
\end{align}
In the first equality, we used that the projection of a face of type $\und{b}$ has measure $\lambda(\pi_c\overline{(\mathbf{0},\und{b})})$, independently of its base point.
The second equality follows from Lemma~\ref{lem:meseigenvec} and the third equality from~\eqref{eq:contrac}. 

Let $U_1$, $U_2$ be two faces of a geometric element $P\subset\mathcal{S}$. Then we have
\begin{align*} \lambda(\pi_c(\overline{\mathbf{E}^{d-1}(\sigma)(U_1 + U_2)})) &=  \lambda(\pi_c(\overline{\mathbf{E}^{d-1}(\sigma)(U_1)})) +  \lambda(\pi_c(\overline{\mathbf{E}^{d-1}(\sigma)(U_2)})) \\
&= \lambda(M_\sigma^{-1}\pi_c(\overline{U_1})) + \lambda(M_\sigma^{-1}\pi_c(\overline{U_2})) \\
&= \beta\,\lambda(\pi_c(\overline{U_1+U_2}))
\end{align*}
where the first equality is true by (S2), the second by \eqref{infl} and the third by \eqref{eq:contrac}. Since this is true for any couple of faces appearing in $P$ we have shown the result.
\end{proof}

\begin{remark} \label{subdiv}
The procedure of inflation and subdivision is described precisely by the formula
\begin{equation}\label{eq:twoduals}
\mathbf{E}^{d-2}(\sigma)\circ\partial = \partial\circ\mathbf{E}^{d-1}(\sigma). 
\end{equation} 
See Figure~\ref{fig:inflsub} for a visualization with $d=3$. The support of a geometric element $\pi_c\overline{P}$ gets inflated by $M_\sigma^{-1}$. Every inflated $(d-2)$-dimensional face $F$ of $\pi_c(\overline{\partial P})$ is replaced by $\pi_c(\overline{\mathbf{E}^{d-2}(F)})$, which has the same endpoints as $\pi_c(\overline{\partial F})$. Notice that cancellation can happen since the matrix $M_{d-2}$ of $\mathbf{E}^{d-2}(\sigma)$ can have negative entries. By \eqref{eq:twoduals} the new boundary of the inflated support equals $\ov{\partial\mathbf{E}^{d-1}(\sigma)(P)}$, which means that the support obtained with the procedure described above can be subdivided in projected supports of elementary faces, and these are exactly those that we get applying $\mathbf{E}^{d-1}(\sigma)$ on $P$.
\end{remark}

\begin{figure}[h]
\includegraphics[scale=.4]{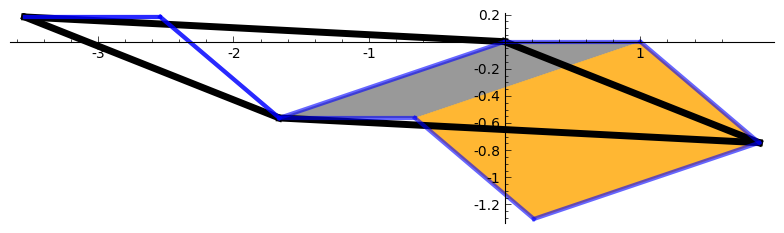}
\caption{The inflated polygon $M_\sigma^{-5}\pi_c\ov{(\mathbf{0},1\wedge 2)}$ (black thick line) and $\pi_c(\ov{\mathbf{E}^2(\sigma_0)^5(\mathbf{0},1\wedge 2)})$ for the Hokkaido substitution $\sigma_0$ (see Section~\ref{sec:hok}). The stepped lines (blue lines) are obtained by applying $\mathbf{E}^1(\sigma_0)^5$ to the boundary of $(\mathbf{0},1\wedge 2)$.}\label{fig:inflsub}
\end{figure}

\subsection{Geometric finiteness property}\label{sec:geomfinprop}
We do not know whether $\pi_c(\ov{\Gamma_\mathcal{U}})$ covers the entire representation space $\mathbb{K}_c$. This motivates the next important definition.

\begin{definition}\label{def:geomprop}
We say that $\sigma$ satisfies the \emph{geometric finiteness property} for a geometric well-projecting $\mathcal{U}\in C_{d-1}$ if there exists an integer $m\geq 1$ such that $\{\mathcal{U} \}\subseteq \{\mathbf{E}^{d-1}(\sigma)^m(\mathcal{U})\}$ and $\pi_c(\ov{\Gamma_\mathcal{U}})$ is a covering of $\mathbb{K}_c$.
\end{definition}

Geometric finiteness has been extensively studied in the irreducible case (see e.g.~\cite{ST09,BST10,MT14}). It is the geometrical interpretation of a certain finiteness property of Dumont-Thomas numeration systems, which was first introduced in the context of beta-numeration in \cite{FS92}.

\medskip

\noindent {\bf Main example.} Now we prove that our family of substitutions of Section~\ref{sec:main-ex} satisfies the geometric finiteness property for certain geometric elements. 


We use some notions of~\cite[Section~3.3]{BBJS:15}.
A geometric element is said to be \emph{edge-connected} if the supports of any two of its faces are connected by a path of supports of faces $\ov{F_1},\ldots,\ov{F_m}$ such that $\ov{F_k}$, $\ov{F_{k+1}}$ share an edge, for all $k\in\{1,\ldots,m-1\}$. 
We say that an edge-connected geometric element $A$ \emph{surrounds} an edge-connected geometric element $P$ if $\{P\}\subset \{A\}$ and $\ov{\partial P}\cap\ov{\partial A} = \emptyset$. If being surrounded is preserved under iterations of the dual substitution $\mathbf{E}^{d-1}(\sigma)$, then we say that the dual substitution satisfies the \emph{annulus property}.

\begin{proposition}\label{propfam}
The following holds for our family of substitutions:
\begin{enumerate}

\item\label{fam1} For every geometric element $\mathcal{U}$ we have $\mathcal{U} \subseteq \mathbf{E}^2(\sigma_t)^5(\mathcal{U})$.
\item\label{fam2} Let $\mathcal{U}_0$ be the sum of three faces centered at $\mathbf{0}$ whose supports intersect pairwise into an edge (we will call such an element $3$-touching in Section~\ref{sec:apertil}). Then for each $i\ge 0$ we have that $\ov{\mathcal{U}_{i+1}} := \ov{\mathbf{E}^2(\sigma_t)^{15}(\mathcal{U}_{i})}$ surrounds $\ov{\mathcal{U}_i}$. Therefore $\mathbf{E}^2(\sigma_t)$ satisfies the annulus property and the geometric finiteness property for $\mathcal{U}_0$ holds.
\end{enumerate}
\end{proposition}
\begin{proof}

It suffices to prove the statements for $\sigma_0$ since the support of $\mathbf{E}^2(\sigma_0)$ is contained in that of $\mathbf{E}^2(\sigma_t)$, for $t\geq 1$, as it can be readily seen in the definition of $\mathbf{E}^2(\sigma_t)$.

\eqref{fam1} By computation one can see that each face occurs with no additional translation in its image by $\mathbf{E}^2(\sigma)^5$.

\eqref{fam2} Again by computation we can check that the basic step of the induction is verified since $\mathbf{E}^2(\sigma_0)^{15}(\mathcal{U}_0)$ surrounds $\mathcal{U}_0$, for every $3$-touching element $\mathcal{U}_0$ (an example of this computation can be seen in Figure~\ref{fig:surrounds}). Furthermore the image of two edge-connected faces is edge-connected.

It remains to show that if $\mathcal{U}_i$ surrounds $\mathcal{U}_{i-1}$ then $\mathcal{U}_{i+1}$ surrounds $\mathcal{U}_i$. 
Suppose that $\mathcal{U}_i$ surrounds $\mathcal{U}_{i-1}$ ``enough'', meaning that the shortest distance between $\partial\mathcal{U}_i$ and $\partial\mathcal{U}_{i-1}$ is large enough. Then, since $\mathbf{E}^2(\sigma_0)$ acts as an imperfect inflate-and-subdivide rule on $\mathbb{K}_c$ with inflation $M_\sigma^{-1}|_{\mathbb{K}_c}$, the shape of $\ov{\mathcal{U}_{i+1}}=\ov{\mathbf{E}^2(\sigma_0)^{15}(\mathcal{U}_i)}$ will be very close to that of $\ov{\mathcal{U}_{i}}=\ov{\mathbf{E}^2(\sigma_0)^{15}(\mathcal{U}_{i-1})}$ inflated by $M_\sigma^{-1}|_{\mathbb{K}_c}$. Hence the shortest distance between $\partial\mathcal{U}_{i+1}$ and $\partial\mathcal{U}_i$ will be bigger than that one between $\ov{\partial\mathcal{U}_i}$ and $\ov{\partial\mathcal{U}_{i-1}}$, which implies that $\ov{\mathcal{U}_{i+1}}$ surrounds $\ov{\mathcal{U}_i}$ (see Figure~\ref{fig:nogeo} on the left for a visualization). 
\end{proof}

\begin{figure}[h]
\begin{center}$
\begin{array}{cc}
\includegraphics[scale=.2]{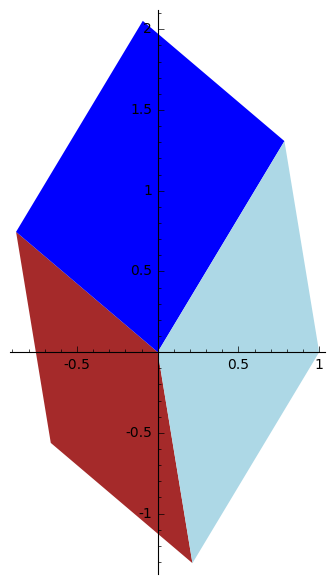} & \quad
\includegraphics[scale=.3]{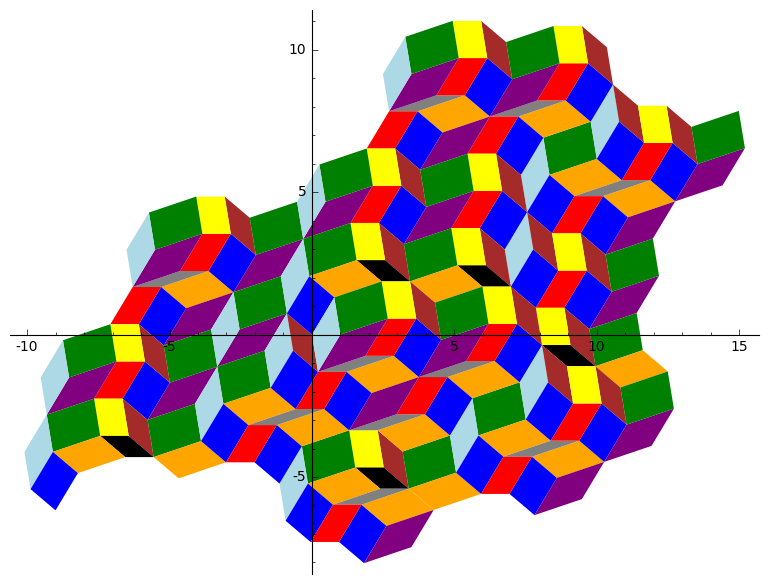}
\end{array}$
\end{center}
\caption{$\mathbf{E}^2(\sigma_0)^{15}(\mathcal{U}_0)$ surrounds $\mathcal{U}_0=\{(\mathbf{0},3\wedge 4),(\mathbf{0},3\wedge 5),(\mathbf{0},4\wedge 5)\}$} \label{fig:surrounds}
\end{figure}

Notice that there exist elements which are not $3$-touching but for which the geometric finiteness property nevertheless holds. The geometric element $\mathcal{U}$ of Figure~\ref{5patch} is an example. On the other hand, the geometric finiteness property does not hold for any single face $(\mathbf{0},\und{a})$ (see Figure~\ref{fig:nogeo} on the right for an example).

\begin{figure}[h]
\includegraphics[scale=.32]{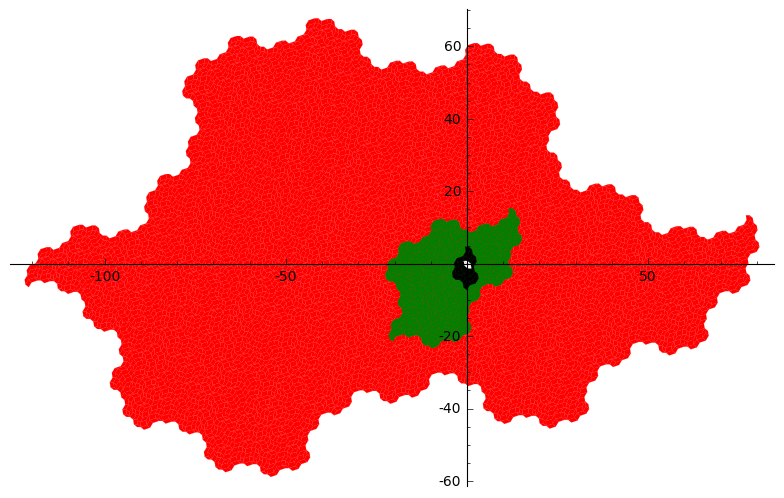}
\quad
\includegraphics[scale=.42]{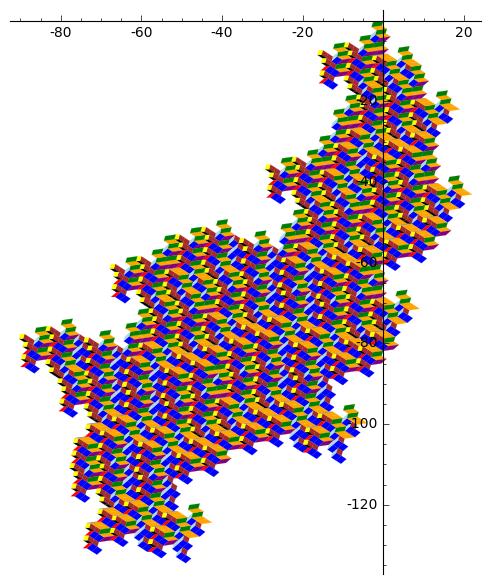}
\caption{$\mathbf{E}^2(\sigma_0)^{10k}((\mathbf{0},1\wedge 4)+(\mathbf{0},1\wedge 5)+(\mathbf{0},4\wedge 5))$ for $k=0,1,2,3$ (left) and $\mathbf{E}^2(\sigma_1)^{10}(\mathbf{0},2\wedge 4)$ (right).\label{fig:nogeo}}
\end{figure}

The aim of this section is to obtain stepped surfaces for reducible Pisot substitutions. Now we have all the necessary ingredients.

\begin{theorem}\label{thm:poly}
Let $\sigma$ be a nice reducible substitution satisfying the geometric finiteness property for some $\mathcal{U}$. 
Then $\Gamma_\mathcal{U}$ is a stepped surface invariant under the substitution rule associated with a power of $\mathbf{E}^{d-1}(\sigma)$.
Furthermore $\pi(\ov{\Gamma_\mathcal{U}})$ stays within bounded distance of $\mathbb{K}_c$. 
\end{theorem}
\begin{proof} 
By the geometric finiteness property there exists an integer $m$ such that $\{\mathcal{U} \}\subseteq \{\mathbf{E}^{d-1}(\sigma)^m(\mathcal{U})\}$ and $\pi_c(\ov{\Gamma_\mathcal{U}})$ covers $\mathbb{K}_c$. 

By Proposition~\ref{subrule}, we can see $\pi_c(\overline{\Gamma_{\mathcal{U}}})$ made of successive inflations and subdivisions of $\pi_c(\overline{\mathcal{U}})$ under $\mathbf{E}^{d-1}(\sigma)$. By Property (S1), for each face $U\in\mathcal{U}$, $\mathbf{E}^{d-1}(\sigma)(U)$ is geometric and projects well and by (S2) 
$\mathbf{E}^{d-1}(\sigma)(U_1+U_2)$ is geometric and projects well, for any two faces $U_1, U_2\in \mathcal{U}$. Thus each $\mathbf{E}^{d-1}(\sigma)^{km}(\mathcal{U})$ projects well onto $\mathbb{K}_c$ for  every $k$, therefore $\pi_c(\ov{\Gamma_\mathcal{U}})$ is a polygonal tiling.

By definition $\Gamma_\mathcal{U}$ is invariant under $\mathbf{E}^{d-1}(\sigma)^m$. Finally we have that $\pi(\overline{\Gamma_\mathcal{U}})$ stays within bounded distance of $\mathbb{K}_c$ since $\Gamma_\mathcal{U}\subset\mathcal{S}$.
\end{proof}
In the language of \cite{S05,Frank08}, $\pi_c(\overline{\Gamma_\mathcal{U}})$ forms a \emph{pseudo self-affine tiling}, since the dual substitution $\mathbf{E}^{d-1}(\sigma)$ acts as an imperfect inflate-and-subdivide rule on $\K_c$.

The next property is crucial in other substitution frameworks; see e.g.~\cite{F06,ABFJ,BF11}.
\begin{proposition}\label{stepstep}
Let $\sigma$ be a nice reducible substitution. Then the image by $\mathbf{E}^{d-1}(\sigma)$ of a stepped surface $\Gamma\subset\mathcal{S}$ is a stepped surface $\{\mathbf{E}^{d-1}(\sigma)\Gamma\}\subset\mathcal{S}$.
\end{proposition}
\begin{proof}
This is a direct consequence of Property (S1) and (S2). Indeed, each face of the stepped surface is replaced after applying $\mathbf{E}^{d-1}(\sigma)$ by a geometric element which projects well and any two distinct faces $U_1,U_2\in\Gamma$ have disjoint images and $\mathbf{E}^{d-1}(\sigma)(U_1+U_2)$ projects well.\\
We now show that $\pi_c(\overline{\mathbf{E}^{d-1}(\sigma)\Gamma})$ covers $\K_c$. Suppose that it is not the case. Then there exists a point $x\in\K_c$ and a face $F\in \mathbf{E}^{d-1}(\sigma)\Gamma$ with the property that $x\in\partial\pi_c(\overline{F})$ and $x\notin\partial\pi_c(\overline{F'})$ for every $F'\in \mathbf{E}^{d-1}(\sigma)\Gamma$ distinct from $F$. Note that this ``hole'' can not follow from a cancellation of faces in the image $\mathbf{E}^{d-1}(\sigma)\Gamma$ because of Property (P). Let $G\in\Gamma$ such that $F\in\mathbf{E}^{d-1}(\sigma)(G)$. In fact, $G$ is unique by (S2). Now, using (S1) and~\eqref{eq:twoduals}, we have
$$x\in\partial\pi_c(\overline{\mathbf{E}^{d-1}(\sigma)(G)})=\pi_c(\overline{\partial\mathbf{E}^{d-1}(\sigma)(G)})=\pi_c(\overline{\mathbf{E}^{d-2}(\sigma)(\partial G)}). 
$$
However, since $\Gamma$ is a stepped surface, we can infer the existence of a face $G'\in\Gamma$ distinct from $G$ such that $x\in\pi_c(\overline{\mathbf{E}^{d-2}(\sigma)(\partial G')})$ (as $\pi_c(\overline{\Gamma})$ is a tiling of $\K_c$). Using again~\eqref{eq:twoduals} and (S1), we obtain that 
$$x\in\pi_c(\overline{\partial\mathbf{E}^{d-1}(\sigma)(G')})=\partial\pi_c(\overline{\mathbf{E}^{d-1}(\sigma)(G')}).
$$  
It follows that $x\in \partial\pi_c(\overline{F'})$ for some $F'\in \mathbf{E}^{d-1}(\sigma)(G)$. By (S2), $F'\ne F$, which contradicts our assumption on $x$.
\end{proof}

Note that an analogous tiling property proved in~\cite[Lemma 4.1 and Proposition 4.2]{AFHI} relies on the fact that boundaries of polygons do not self-intersect. This property is insured by our assumptions (S1) and (S2). 

\section{Rauzy fractals and tiling results}\label{sec:rauzyfrac}

In this section we define a new class of Rauzy fractals generated by the dual substitution $\mathbf{E}^{d-1}(\sigma)$. 
We prove then that, under the geometrical finiteness property, they tile aperiodically and periodically their representation space $\K_c$.
Some of these methods are similar to those used in the irreducible setting, and we include for sake of completeness some proof sketches.

Nevertheless, there are some important conceptual changes: to tackle  reducible substitutions we work with the wedge formalism and with higher-dimensional geometric dual substitutions, thus well-known concepts of Rauzy fractals theory, like e.g.~the strong coincidence condition or the domain exchange, must be redefined.

\subsection{Rauzy fractals}

Since the sequence of sets $M_\sigma^k\,\pi_c(\overline{\mathbf{E}^{d-1}(\sigma)^k (\mathbf{x},\underline{a})})$ converges in the Hausdorff metric (see also \cite{SAI}), this leads to the following natural definition.

\begin{definition}\label{def:rauzyfrac}
The {\it Rauzy fractals} are defined as
\[ \mathcal{R}(\underline{a}) + \pi_c(\mathbf{x}) = \lim_{k\to\infty} M_\sigma^k\,\pi_c(\overline{\mathbf{E}^{d-1}(\sigma)^k (\mathbf{x},\underline{a})}),\quad \text{ for any } (\mathbf{x},\und{a})\in\mathcal{S}, \]
where the limit is taken with respect to the Hausdorff metric.
\end{definition}

\begin{proposition}\label{prop:seteq}
We have the set equations
\begin{equation}\label{eq:seteq} \mathcal{R}(\underline{a}) + \pi_c(\mathbf{x}) =  \bigcup_{(\mathbf{y},\underline{b})\in \mathbf{E}^{d-1}(\sigma)(\mathbf{x},\underline{a})} M_\sigma\big(\mathcal{R}(\underline{b}) + \pi_c(\mathbf{y})\big). 
\end{equation}
Furthermore, if $\sigma$ is a nice reducible substitution then the union on the right-hand side is measure disjoint. 
\end{proposition}
\begin{proof}
The set equations follow easily by definition of Rauzy fractal. For the measures the following holds
\[ \beta\;\lambda(\mathcal{R}(\underline{a})) \leq \sum_{\und{b}} |({}^t M_{d-1})_{\und{a}\,\und{b}}|\;\lambda(\mathcal{R}(\underline{b})).  \]
We get the equality since by Lemma~\ref{le:modone} the Perron-Frobenius eigenvalue of $|M_{d-1}|$ is $\beta$.
\end{proof}

Theorem~\ref{thm:poly} and the definition of $\Gamma_\mathcal{U}$ imply that $\Gamma_\mathcal{U}$ seen as set of colored points (the base points of the faces colored by their type) is a {\it substitution Delone set} (see \cite{LW}).
Rauzy fractals are constructed starting from a substitution Delone set, as Proposition~\ref{prop:seteq} shows.
Thus, we can use the result \cite[Theorem 5.5]{LW} (see also \cite[Theorem~6.1]{EIR}) to deduce good properties for the Rauzy fractals.

\begin{proposition}\label{prop:prop}
If $\sigma$ is a nice reducible substitution then the Rauzy fractals have the following properties.
\begin{enumerate}
\item They are compact sets with non-zero measure.
\item They are the closure of their interior.
\item Their fractal boundary has zero measure.
\end{enumerate}
\end{proposition}

\subsection{Strong coincidence condition}

We introduce an important combinatorial condition on the substitution.
\begin{definition}\label{def:scc}
We say that $\sigma$ has a \emph{coincidence} for $\und{a}_0$, $\und{b}_0 \in\mathcal{O}_{\bar n}$ if there exist $k\in\mathbb{N}$ and $\und{c}\in\mathcal{O}_{\bar n}$ such that, for every $i=0,\ldots,k-2$, $\sigma(\und{a}_i) = \mathbf{p}_{i}\und{a}_{i+1}\mathbf{s}_{i}$, $\sigma(\und{b}_i) = \mathbf{p}'_{i}\und{b}_{i+1}\mathbf{s}'_{i}$, and $\sigma(\und{a}_{k-1}) = \mathbf{p}_{k-1}\und{c}\mathbf{s}_{k-1}$, $\sigma(\und{b}_{k-1}) = \mathbf{p}'_{k-1}\und{c}\mathbf{s}'_{k-1}$,  with $M_\sigma^{k-1}\mathbf{l}(\mathbf{s}_0)+\cdots+M_\sigma\mathbf{l}(\mathbf{s}_{k-2}) + \mathbf{l}(\mathbf{s}_{k-1}) = M_\sigma^{k-1}\mathbf{l}(\mathbf{s}'_0)+\cdots+M_\sigma\mathbf{l}(\mathbf{s}'_{k-2})+\mathbf{l}(\mathbf{s}'_{k-1})$.

We say that the substitution $\sigma$ satisfies the (\emph{suffix}) {\it strong coincidence condition} if $\sigma$ has a coincidence for every pair $\und{a}_0, \und{b}_0\in\mathcal{O}_{\bar n}$ such that the geometric element $(\mathbf{0},\und{a}_0^*)+(\mathbf{0},\und{b}_0^*)$ projects well.

An equivalent formulation holds in terms of prefixes.
\end{definition}

\begin{remark}
Notice that in the irreducible case the strong coincidence condition coincides with the one defined in \cite{AI}. But since we work with $\mathbf{E}^{d-1}(\sigma)$ which depends on wedges on $\bar n$ letters, we will consider the natural generalization given by the above definition.
\end{remark}

\begin{proposition}\label{SCC}
Let $\sigma$ be a nice reducible substitution satisfying the strong coincidence condition. Then, for any well-projecting geometric element $\mathcal{U}$, the subtiles $\mathcal{R}(\und{a})$, with $(\mathbf{0},\und{a})\in\mathcal{U}$, are pairwise measure disjoint.
\end{proposition}
\begin{proof}
By the strong coincidence condition, for every $(\mathbf{0},\und{a}_0^*),(\mathbf{0},\und{b}_0^*)\in\mathcal{U}$, we have that there exist $k\in\mathbb{N}$ and $\und{c}^*\in\mathcal{O}_{d-1}$ such that $M_\sigma^k\mathcal{R}(\und{a}_0^*)+\pi_c(M_\sigma^{k-1}\mathbf{l}(\mathbf{s}_0)+\cdots+M_\sigma\mathbf{l}(\mathbf{s}_{k-2}) + \mathbf{l}(\mathbf{s}_{k-1}))$ and $M_\sigma^k\mathcal{R}(\und{b}_0^*)+\pi_c(M_\sigma^{k-1}\mathbf{l}(\mathbf{s}'_0)+\cdots+M_\sigma\mathbf{l}(\mathbf{s}'_{k-2})+\mathbf{l}(\mathbf{s}'_{k-1}))$ both appear in the $k$-fold iteration of the set equations of Proposition \ref{prop:seteq} for $\mathcal{R}(\und{c}^*)$, and furthermore are measure disjoint.
\end{proof}

\medskip

\noindent {\bf Main example}

\begin{proposition}\label{prop:scc}
Every $\sigma_t$ of the family of Section~\ref{sec:main-ex} satisfies the suffix strong coincidence condition.
\end{proposition}

\begin{proof}
Recall from Definition~\ref{def:scc} the notion of coincidence.
The table of Figure~\ref{tab:coinc} shows the list of $\und{a}$ and $\und{b}$ with the corresponding $k$ for which $\sigma_0$ has a coincidence.

\begin{figure}
\begin{center}
\begin{tabular}{ c | c | c }
$\und{a}$ & $\und{b}$ & $k$ \\ \hline  
$1\wedge 2\wedge 3$ & $1\wedge 2\wedge 4$ & $7$ \\
$1\wedge 2\wedge 3$ & $1\wedge 2\wedge 5$ & $7$ \\
$1\wedge 2\wedge 3$ & $2\wedge 3\wedge 4$ & $11$ \\
$1\wedge 2\wedge 3$ & $2\wedge 3\wedge 5$ & $10$ \\
$1\wedge 2\wedge 3$ & $2\wedge 4\wedge 5$ & $12$ \\
$1\wedge 2\wedge 4$ & $1\wedge 2\wedge 5$ & $6$ \\
$1\wedge 2\wedge 4$ & $1\wedge 3\wedge 4$ & $8$ \\
$1\wedge 2\wedge 4$ & $1\wedge 3\wedge 5$ & $8$ \\
$1\wedge 2\wedge 4$ & $2\wedge 3\wedge 5$ & $10$ \\
$1\wedge 2\wedge 4$ & $2\wedge 4\wedge 5$ & $10$ \\
$1\wedge 2\wedge 4$ & $3\wedge 4\wedge 5$ & $10$ \\
$1\wedge 2\wedge 5$ & $1\wedge 3\wedge 4$ & $8$ \\
$1\wedge 2\wedge 5$ & $1\wedge 3\wedge 5$ & $8$ \\
$1\wedge 2\wedge 5$ & $1\wedge 4\wedge 5$ & $8$ \\
$1\wedge 3\wedge 4$ & $1\wedge 3\wedge 5$ & $6$ \\
$1\wedge 3\wedge 4$ & $2\wedge 3\wedge 4$ & $9$ \\
$1\wedge 3\wedge 4$ & $2\wedge 3\wedge 5$ & $9$ \\
$1\wedge 3\wedge 4$ & $2\wedge 4\wedge 5$ & $10$ \\
$1\wedge 3\wedge 4$ & $3\wedge 4\wedge 5$ & $10$ \\
$1\wedge 3\wedge 5$ & $1\wedge 4\wedge 5$ & $7$ \\
$1\wedge 3\wedge 5$ & $2\wedge 3\wedge 4$ & $11$ \\
$1\wedge 3\wedge 5$ & $2\wedge 3\wedge 5$ & $9$ \\
$1\wedge 3\wedge 5$ & $2\wedge 4\wedge 5$ & $9$ \\
$1\wedge 4\wedge 5$ & $2\wedge 3\wedge 5$ & $9$ \\
$1\wedge 4\wedge 5$ & $2\wedge 4\wedge 5$ & $9$ \\
$1\wedge 4\wedge 5$ & $3\wedge 4\wedge 5$ & $9$ \\
$2\wedge 3\wedge 4$ & $2\wedge 3\wedge 5$ & $11$ \\
$2\wedge 3\wedge 4$ & $3\wedge 4\wedge 5$ & $10$ \\
$2\wedge 3\wedge 5$ & $2\wedge 4\wedge 5$ & $7$ \\
$2\wedge 4\wedge 5$ & $3\wedge 4\wedge 5$ & $8$ 
\end{tabular}
\end{center}
\caption{Table of coincidences for $\sigma_0$.\label{tab:coinc}}
\end{figure}

One can check that these are all possible $\und{a}$ and $\und{b}$ such that the geometric element $\mathcal{U}=(\mathbf{0},\und{a}^*)+(\mathbf{0},\und{b}^*)$ projects well. Thus we conclude that $\sigma_0$ satisfies the strong coincidence condition.
Observe that if $\sigma_0$ has a coincidence, then all $\sigma_t$ for $t\ge 1$ will have that coincidence, as one can see from the definition \eqref{eq:E2sigmat} of $\mathbf{E}^2(\sigma_t)$. 
Indeed, by Definition~\ref{def:sufgraph} there is an edge $\und{a}^* \xrightarrow{\mathbf{s}}\und{b^*}$ if and only if $\sigma(\und{a}) = \mathbf{p}\und{b}\mathbf{s}$, thus if $\sigma_0$ has a coincidence for $\und{a}_0,\und{b}_0$ there are two paths in this graph starting from $\und{a}_0$ and $\und{b}_0$ respectively and ending at the same $\und{c}$, with same suffix abelianization. But since the graph of $\sigma_t$ has the same set of vertices as the graph of $\sigma_0$ and a set of edges that includes the set of edges of the graph of $\sigma_0$, the same paths exist in the graph of $\sigma_t$, for every $t$.
Therefore, for every $t\ge 0$, $\sigma_t$ satisfies the strong coincidence condition.
\end{proof}

\subsection{Aperiodic and periodic tilings}\label{sec:apertil}

In the following theorem we show that a nice reducible substitution with the geometric finiteness and strong coincidence conditions induces a fractal aperiodic tiling simply replacing the polygonal faces of a stepped surface with the Rauzy fractals.

\begin{theorem}\label{thm:selfreptil}
Let $\sigma$ be a nice reducible substitution which satisfies the strong coincidence condition and the geometric finiteness property for the well-projecting geometric element $\mathcal{U}$ with faces based at $\mathbf{0}$. Then the collection $\{\mathcal{R}(\underline{a})+\pi_c(\mathbf{x}) : (\mathbf{x},\underline{a})\in\Gamma_\mathcal{U}\}$, where $\Gamma_\mc{U}$ is defined by \eqref{gamma}, is a self-replicating tiling of $\K_c$.
\end{theorem}
\begin{proof}
Start with the union of the $\mathcal{R}(\und{a})$ such that $(\mathbf{0},\und{a})\in \mathcal{U}$. By the strong coincidence condition they are measure disjoint.
By inflating each of them repeatedly by $M_\sigma^{-1}$ and applying the set equations~\eqref{eq:seteq} we get a measure-disjoint union of subtiles. The geometric finiteness property for $\mathcal{U}$ ensures that this process of inflation and subdivision on $\bigcup_{(\mathbf{0},\und{a})\in\mathcal{U}} \mathcal{R}(\und{a})$ will cover all $\mathbb{K}_c$. It remains to prove the measure disjointness of any two $\mathcal{R}(\und{a}_1) + \pi_c(\mathbf{x}_1) \subseteq M_\sigma^{-k}\mathcal{R}(\und{a})$, $\mathcal{R}(\und{a}_2) + \pi_c(\mathbf{x}_2)  \subseteq M_\sigma^{-k}\mathcal{R}(\und{b})$ with $(\mathbf{0},\und{a})$, $(\mathbf{0},\und{b})$ different elements in $\mathcal{U}$. But this is true since $\mathcal{R}(\und{a})$ and $\mathcal{R}(\und{b})$ are measure disjoint by the strong coincidence condition, and a posteriori the same holds for their inflations by $M_\sigma^{-k}$. 
\end{proof}

\begin{figure}[h]
\includegraphics[scale=0.45]{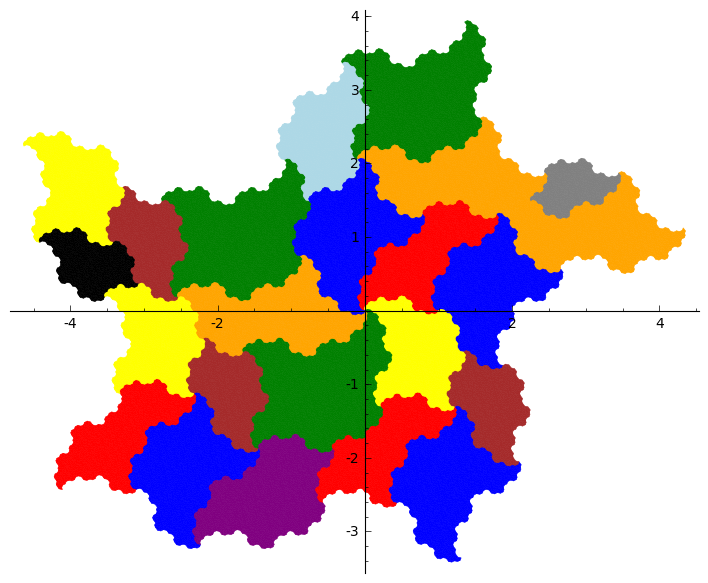}
\caption{A patch of the self-replicating tiling $\{\mathcal{R}(\und{a})+\pi_c(\mathbf{x}): (\mathbf{x},\und{a})\in \Gamma_\mathcal{U}\}$ for $\sigma_0$, where $\mathcal{U}$ is as in Figure~\ref{5patch}.}
\end{figure}

One of the main novelties is that we obtain natural periodic tilings by our new Rauzy fractals starting from periodic polygonal tilings.

A {\it periodic element} is a geometric element $\mathcal{P}\subset\mathcal{S}$ which, translated by a set of points $\Lambda_\mathcal{P}\subset \mathbb{Z}^n$ such that $\pi_c(\Lambda_\mathcal{P})$ is a lattice, forms a stepped surface. In this case we will call the latter a {\it periodic stepped surface}.
Examples of periodic stepped surfaces are:
\begin{enumerate}
\item A single $(d-1)$-dimensional face $(\mathbf{0},a_1\wedge\cdots\wedge a_{d-1})$ together with  $\Lambda_\mathcal{P} = \mathbf{e}_{a_1}\mathbb{Z}+\cdots+\mathbf{e}_{a_{d-1}}\mathbb{Z}$.
\item A \emph{touching pair}, that is, two faces projecting well  whose supports share a $(d-2)$-dimensional face: $(\mathbf{0},b\wedge a_2\wedge \cdots \wedge a_{d-1}) + (\mathbf{0},c\wedge a_2\wedge\cdots\wedge a_{d-1})$, $b\neq c$. The associated set is $\Lambda_\mathcal{P}=(\mathbf{e}_b-\mathbf{e}_c)\Z+\sum_{i=2}^{d-1}\mathbf{e}_{a_i}\Z$.
\item\label{item:dtouch} A $d$-\emph{touching} element, that is, $d=\deg(\beta)$ faces which are touching in pairs $\sum_{k=1}^d(\mathbf{0},a_1\wedge\cdots\wedge \widehat{a}_k\wedge\cdots\wedge a_d)$, where $\widehat{a}$ denotes that $a$ does not appear. The associated set is $\Lambda_\mathcal{P}=\sum_{i=2}^d (\mathbf{e}_{a_1}-\mathbf{e}_{a_i})\Z$.
\end{enumerate}

Let  $\mathcal{P}$ be $d$-touching such that its translations by $\Lambda_\mathcal{P}$ induce a polygonal periodic tiling. Let $\mathcal{R}_\mathcal{P} := \bigcup_{(\mathbf{0},\und{a})\in\mathcal{P}}\mathcal{R}(\und{a})$ and $\mathcal{R}_k(\und{a}):=M_\sigma^k\,\pi_c(\ov{\mathbf{E}^{d-1}(\sigma)^k(\mathbf{0},\und{a})})$.

\begin{proposition}\label{pcov}
Let $\sigma$ be a nice reducible substitution. For a periodic element $\mc{P}\subset\mc{S}$ we have that $\mathcal{R}_\mathcal{P}+\pi_c(\Lambda_\mathcal{P})$ is a periodic covering of $\mathbb{K}_c$.
\end{proposition}
\begin{proof}
Since by Proposition~\ref{stepstep} $\pi_c(\ov{\mathbf{E}^{d-1}(\sigma)(\mathcal{P})}) + \pi_c(M^{-1}_\sigma\Lambda_\mathcal{P})$ is again a polygonal tiling, we have that
\[ \lambda\big(M_\sigma^k\,\pi_c(\ov{\mathbf{E}^{d-1}(\sigma)^k(\mathbf{x},\und{a})}) \cap M_\sigma^k\,\pi_c(\ov{\mathbf{E}^{d-1}(\sigma)^k(\mathbf{y},\und{b})}\big)= 0 \] for any two faces $(\mathbf{x},\und{a})$, $(\mathbf{y},\und{b})\in\mathcal{P}+\Lambda_\mathcal{P}$ and any positive integer $k$. Furthermore $\lambda(\pi_c(M_\sigma^k\,\ov{\mathbf{E}^{d-1}(\sigma)^k(\mathbf{x},\und{a})})) = \lambda(\pi_c\ov{(\mathbf{x},\und{a})})$, since $(\lambda(\pi_c\ov{(\mathbf{0},\und{a})}))_{\und{a}\in\mathcal{O}_{d-1}}$ is a Perron-Frobenius eigenvector of $|{}^t M_{d-1}|$ associated with $\beta$. Thus $\bigcup_{(\mathbf{0},\und{a})\in\mathcal{P}}\mathcal{R}_k(\und{a})+\pi_c(\Lambda_\mathcal{P})$ is a periodic tiling, and since $\mathcal{R}_\mathcal{P}$ is the Hausdorff limit of the approximations $\bigcup_{(\mathbf{0},\und{a})\in\mathcal{P}}\mathcal{R}_k(\und{a})$ it follows that $\mathcal{R}_\mathcal{P}+\pi_c(\Lambda_\mathcal{P})$ is a covering. 
\end{proof}

\begin{theorem}\label{thm:pertil}
Let $\sigma$ be a nice reducible substitution such that the strong coincidence condition and the geometric finiteness property for the periodic element $\mathcal{P}\in\mathcal{S}$ hold. Then $\mathcal{R}_\mathcal{P} + \pi_c(\Lambda_\mathcal{P})$ is a periodic tiling.
\end{theorem}
\begin{proof}
If the geometric finiteness property for $\mathcal{P}$ holds then $\{\mathcal{R}(\underline{a})+\pi_c(\mathbf{x}) : (\mathbf{x},\underline{a})\in\Gamma_\mathcal{P}\}$ and $\pi_c(\Gamma_\mathcal{P})$ are tilings with the same translation set, where $\Gamma_\mathcal{P}$ is defined by~\eqref{gamma}. This implies that $\lambda(\mathcal{R}(\und{a}))=\lambda(\pi_c\ov{(\mathbf{0},\und{a})})$, for all $\und{a}$, in particular for those appearing in $\mathcal{P}$, and we know $\mathcal{R}_\mathcal{P} + \pi_c(\Lambda_\mathcal{P})$ is a covering by Proposition \ref{pcov}. Therefore $\mathcal{R}_\mathcal{P}+\pi_c(\Lambda_\mathcal{P})$ is a tiling.
\end{proof}

\begin{figure}
\includegraphics[scale=0.5]{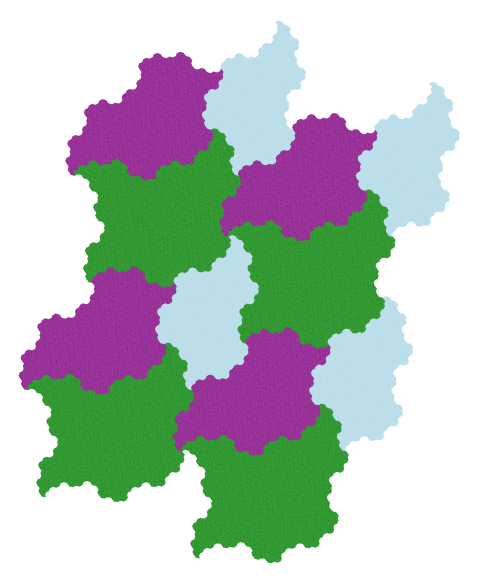}
\caption{A patch of the periodic tiling $\mathcal{R}_\mathcal{P}+\pi_c(\Lambda_\mathcal{P})$ induced by the Hokkaido substitution $\sigma_0$ of Section~\ref{sec:hok}, for the $3$-touching element $\mathcal{P}=(\mathbf{0},2\wedge 3)+(\mathbf{0},2\wedge 4)+(\mathbf{0},3\wedge 4)$ and with $\Lambda_\mathcal{P}=(\mathbf{e}_4-\mathbf{e}_2)\Z + (\mathbf{e}_4-\mathbf{e}_3)\Z$.\label{figper3}}
\end{figure}

The following proposition asserts that we get a periodic tiling whenever the boundaries of the approximations converge to the boundary of the Rauzy fractal (cf.~\cite[Theorem~3.3]{IR06} for the irreducible settings).
\begin{proposition}\label{limboundary}
Let $\sigma$ be a nice reducible substitution such that the strong coincidence condition and the geometric finiteness property for the periodic element $\mathcal{P}\in\mathcal{S}$ hold.
The collection $\mathcal{R}_\mathcal{P} + \pi_c(\Lambda_\mathcal{P})$ is a periodic tiling of $\K_c$ if and only if $\lim_{k\to\infty} \partial \mathcal{R}_k(\und{a}) = \partial \mathcal{R}(\und{a})$, $\forall\,\und{a}\in \mathcal{P}$.
\end{proposition}
\begin{proof} 
Assume $\lim_{k\to\infty} \partial \mathcal{R}_k(\und{a}) = \partial \mathcal{R}(\und{a})$. Then, since $\mathcal{R}_k(\und{a})\to \mathcal{R}(\und{a})$, we have that for any $\varepsilon > 0$ there exists $k=k(\varepsilon)$ such that 
\[ d_H(\mathcal{R}(\und{a}),\mathcal{R}_k(\und{a})) < \varepsilon\quad\text{and}\quad d_H(\partial\mathcal{R}(\und{a}),\partial\mathcal{R}_k(\und{a})) < \varepsilon. \]
By the former, for any $z \in \mathcal{R}(\und{a})\setminus \mathcal{R}_k(\und{a})$ there exists $z'\in B(z,\varepsilon)\cap \mathcal{R}_k(\und{a})$, which implies that the line segment from $z$ to $z'$ must intersect $\partial\mathcal{R}_k(\und{a})$. Hence there exists $z''\in\partial\mathcal{R}_k(\und{a})$ such that $|z-z''|<\varepsilon$, and
\[ \mathcal{R}(\und{a}) \subseteq \mathcal{R}_k(\und{a}) \cup [\partial\mathcal{R}_k(\und{a})]_\varepsilon, \] where $[X]_\varepsilon = \{x : |x-y|<\varepsilon \text{ for some } y \in X\}$. 
Since the approximations $\mathcal{R}_k(\und{a})$ have for every $k$ the same measure as the projected faces $\pi_c\ov{(\mathbf{0},\und{a})}$ and $\mathcal{R}_\mathcal{P}+\pi_c(\Lambda_\mathcal{P})$ is a covering we have $\lambda(\mathcal{R}_\mathcal{P})/\lambda(\bigcup_{(\mathbf{0},\und{a})\in\mathcal{P}}\mathcal{R}_k(\und{a})) \ge 1$. Thus we get equality if $\lim_{\varepsilon\to 0} \lambda([\partial\mathcal{R}_k(\und{a})]_\varepsilon) = 0$. But the inequality $d_H(\partial\mathcal{R}(\und{a}),\partial\mathcal{R}_k(\und{a})) < \varepsilon$ implies that $[\partial\mathcal{R}_k(\und{a})]_\varepsilon \subseteq [\partial\mathcal{R}(\und{a})]_{2\varepsilon}$ and $\lim_{\varepsilon\to 0} \lambda([\partial\mathcal{R}(\und{a})]_{2\varepsilon})= 0$ since $\partial\mathcal{R}(\und{a})$ has measure zero.

Let 
\[ C = \text{diam}(\mathcal{R}_{\mathcal{P}}) + \text{diam}\bigg(\bigcup_{(\mathbf{0},\und{a})\in\mathcal{P}} \pi_c\ov{(\mathbf{0},\und{a})}\bigg), \]
and suppose that $\mathcal{R}_\mathcal{P}+\pi_c(\Lambda_\mathcal{P})$ is a tiling. Then $M_\sigma^{-k}\mathcal{R}_\mathcal{P}$ and $M_\sigma^{-k}\bigcup_{(\mathbf{0},\und{a})\in\mathcal{P}}\mathcal{R}_k(\und{a})$ tile both $\mathbb{K}_c$ modulo $M_\sigma^{-k}\pi_c(\Lambda_\mathcal{P})$. 
This implies that if $B(z,R)\subset M_\sigma^{-k}\mathcal{R}_\mathcal{P}$, for some $z\in \mathbb{K}_c$ and $R>C$, then $B(z,R-C)\subset M_\sigma^{-k}\bigcup_{(\mathbf{0},\und{a})\in\mathcal{P}}\mathcal{R}_k(\und{a})$. Similarly, if $B(z,R)\subset M_\sigma^{-k}\bigcup_{(\mathbf{0},\und{a})\in\mathcal{P}}\mathcal{R}_k(\und{a})$ for some $z\in \mathbb{K}_c$ and $R>C$, then $B(z,R-C)\subset M_\sigma^{-k}\mathcal{R}_\mathcal{P}$. 
Hence 
\[ d_H(\partial(M_\sigma^{-k}\mathcal{R}_\mathcal{P}),\partial\bigg(\bigcup_{(\mathbf{0},\und{a})\in\mathcal{P}} M_\sigma^{-k}\mathcal{R}_k(\und{a})\bigg))<2C,  \] and the result follows.
\end{proof}

\noindent {\bf Main example.} 

\begin{corollary}\label{cor:tilings}
Let $\sigma_t$ be a substitution of the family \eqref{fam} and let $\mathcal{U}$ be a geometric element which projects well containing a $3$-touching element. Then the collection $\{\mathcal{R}(\und{a})+\pi_c(\mathbf{x}) : (\mathbf{x},\und{a})\in\Gamma_\mathcal{U}\}$ is a self-replicating tiling of $\mathbb{K}_c\cong \mathbb{C}$. Furthermore, if $\pi_c(\overline{\mathcal{U}})$ tiles periodically $\mathbb{K}_c$, then $\mathcal{R}_\mathcal{U}$ tiles periodically $\mathbb{K}_c$. Hence, $\mathcal{R}_\mathcal{U}$ is a fundamental domain for the torus $\mathbb{T}^2$ and the domain exchange $E_\mathcal{U}$ projects to a translation on it.
\end{corollary}
\begin{proof}
This is a direct consequence of Theorem~\ref{thm:selfreptil} and \ref{thm:pertil}, since we showed in Proposition~\ref{prop:nice} that each substitution $\sigma_t$ is a nice reducible substitution, in Proposition~\ref{propfam} that $\sigma_t$ satisfies the geometric finiteness property for every $3$-touching $\mathcal{U}$, and in Proposition~\ref{prop:scc} that the strong coincidence condition for $\sigma_t$ holds. 
\end{proof}

\subsection{Domain exchange}

Let  $\mathcal{P}$ be $d$-touching with associated set $\Lambda_\mathcal{P}$ and let $\mathcal{A}_\mathcal{P}$ be its alphabet, consisting of the single letters appearing in the types of faces of $\mathcal{P}$.
If the strong coincidence condition holds, then the components $\mathcal{R}(\und{a})$ of $\mathcal{R}_\mathcal{P}$ are pairwise measure disjoint by Proposition~\ref{SCC}. Therefore we can define $\lambda$-a.e. the \emph{domain exchange} on $\mathcal{R}_\mathcal{P}$ as
\begin{equation}
E_\mathcal{P} : \mathbf{x}  \mapsto \mathbf{x} + \pi_c(\mathbf{e}_\ell), \quad \mathbf{x}\in\mathcal{R}(\und{a}),
\end{equation}
for $\ell\in\mathcal{A}_\mathcal{P}\setminus\{a_1,\ldots,a_{d-1}\}$, $\und{a}=a_1\wedge\cdots\wedge a_{d-1}$.

If additionally the geometric finiteness property for $\mathcal{P}$ holds, then, by Theorem~\ref{thm:pertil}, $\mathcal{R}_\mathcal{P} + \pi_c(\Lambda_\mathcal{P})$ is a periodic tiling and $\mathcal{R}_\mathcal{P}$ is a fundamental domain for the $(d-1)$-dimensional torus $\mathbb{K}_c/\pi_c(\Lambda_\mathcal{P})$. Thus, the natural projection of $E_\mathcal{P}$ on this $(d-1)$-dimensional torus is a translation.

We are interested now in codings of the domain exchange $E_\mathcal{P}$ with respect to the natural partition $\{\mathcal{R}(\und{a}):\und{a}\in\mathcal{P}\}$ of $\mathcal{R}_\mathcal{P}$.
We will investigate the symbolic dynamical systems $\Omega$ which codes the orbits of $(\mathcal{R}_\mathcal{P},E_\mathcal{P})$ and establish connections with the original substitution dynamical system $(X_\sigma,S)$. This will be done in the next section for our one-parameter family of substitutions of Section~\ref{sec:main-ex}.


\section{Modified stepped lines}\label{sec:comb}

Being reducible for a substitution means that we have some linear dependencies between the $\pi(\mathbf{e}_i)$, for $i=1,\ldots,n$ (see Lemma~\ref{lem:redund}). For each $\sigma_t$ of the family of Section~\ref{sec:main-ex}, we have
\begin{equation}\label{eq:ratdep}
 \pi(\mathbf{e}_1) = \pi(\mathbf{e}_3) + \pi(\mathbf{e}_4), \qquad  \pi(\mathbf{e}_5) = \pi(\mathbf{e}_2) + \pi(\mathbf{e}_3).
\end{equation}
We have a {\it stepped line} in $\mathbb{R}^5$ which is the geometrical interpretation of a fixed point $u\in\mathcal{A}^\mathbb{N}$ of $\sigma_t$: 
\[  \overline{u} = \bigcup_{i\geq 1} \{(\mathbf{l}(u_0\cdots u_{i-1}),u_i)\},  \]
where $(\mathbf{x},i)$ denotes the segment from $\mathbf{x}$ to $\mathbf{x}+\mathbf{e}_i$.
Projecting the stepped line into $\mathbb{K}_\beta\cong \mathbb{R}^3$ the rational dependencies show up, and we get what we call a ``reducible'' stepped line, made of five different segments. We can change this stepped line using the rational dependencies, i.e.~we substitute every $\pi(\mathbf{e}_1)$ and $\pi(\mathbf{e}_5)$ with their linearly independent atoms $\pi(\mathbf{e}_2)$, $\pi(\mathbf{e}_3)$ and $\pi(\mathbf{e}_4)$ as in the relations (\ref{eq:ratdep}). Combinatorially this is equivalent to applying the morphism
\begin{equation}\label{chi}
\chi: \; 1\mapsto\, 34, \quad 2\mapsto\, 2, \quad 3\mapsto\, 3, \quad 4\mapsto\, 4, \quad 5\mapsto\, 32, 
\end{equation}
to the fixed point of $\sigma_t$. 
The projection of the stepped lines $\overline{u}$ and $\overline{\chi(u)}$ onto $\mathbb{K}_e\cong\mathbb{R}$ form two tilings (see Figure~\ref{fig:tline}). 

\begin{figure}[h]
\begin{tikzpicture}[scale=1]
\draw[very thick,red]  (0,0)--(1,0) (1+0.32+0.43+0.56+0.75,0)--(1+0.32+0.43+0.56+0.75+1,0) (1+0.32+0.43+0.56+0.75+1,0)--(1+0.32+0.43+0.56+0.75+2,0)
(1+0.32+0.43+0.56+0.75+2+0.32,0)--(1+0.32+0.43+0.56+0.75+3+0.32,0);
\draw[very thick,yellow] (1,0)--(1+0.32,0) (1+0.32+0.43+0.56+0.75+2,0)--(1+0.32+0.43+0.56+0.75+2+0.32,0) (1+0.32+0.43+0.56+0.75+3+0.32,0)--(1+0.32+0.43+0.56+0.75+3+0.64,0);
\draw[very thick,green] (1+0.32,0)--(1+0.32+0.43,0) (1+0.32+0.43+0.56+0.75+3+0.64,0)--(1+0.32+0.43+0.56+0.75+3+0.64+0.43,0);
\draw[very thick,blue] (1+0.32+0.43,0)--(1+0.32+0.43+0.56,0);
\draw[very thick,violet] (1+0.32+0.43+0.56,0)--(1+0.32+0.43+0.56+0.75,0);
\draw (0,-.1)--(0,.1) (1/2,.1)node[above]{$1$} (1,-.1)--(1,.1) (1.16,.1)node[above]{$2$} (1+0.32,-.1)--(1+0.32,.1) (1.52,.1)node[above]{$3$} (1+0.32+0.43,-.1)--(1+0.32+0.43,.1) (2,.1)node[above]{$4$} (1+0.32+0.43+0.56,-.1)--(1+0.32+0.43+0.56,.1)  (2.65,.1)node[above]{$5$} (1+0.32+0.43+0.56+0.75,-.1)--(1+0.32+0.43+0.56+0.75,.1) (3.65,.1)node[above]{$1$} (1+0.32+0.43+0.56+0.75+1,-.1)--(1+0.32+0.43+0.56+0.75+1,.1) (4.6,.1)node[above]{$1$} (1+0.32+0.43+0.56+0.75+2,-.1)--(1+0.32+0.43+0.56+0.75+2,.1) (5.2,.1)node[above]{$2$} (1+0.32+0.43+0.56+0.75+2+0.32,-.1)--(1+0.32+0.43+0.56+0.75+2+0.32,.1) (5.9,.1)node[above]{$1$} (1+0.32+0.43+0.56+0.75+3+0.32,-.1)--(1+0.32+0.43+0.56+0.75+3+0.32,.1) (6.5,.1)node[above]{$2$} (1+0.32+0.43+0.56+0.75+3+0.64,-.1)--(1+0.32+0.43+0.56+0.75+3+0.64,.1) (6.9,.1)node[above]{$3$} (1+0.32+0.43+0.56+0.75+3+0.64+0.43,-.1)--(1+0.32+0.43+0.56+0.75+3+0.64+0.43,.1);
\end{tikzpicture}

\begin{tikzpicture}[scale=1]
\draw[very thick,cyan!50!white] (1,0)--(1+0.32,0) (1+0.32+0.43+0.56+0.75+2,0)--(1+0.32+0.43+0.56+0.75+2+0.32,0) (1+0.32+0.43+0.56+0.75+3+0.32,0)--(1+0.32+0.43+0.56+0.75+3+0.64,0) (1+0.32+0.43+0.56+0.43,0)--(1+0.32+0.43+0.56+0.43+0.32,0);
\draw[very thick,green!60!black] (0,0)--(.43,0) (1+0.32,0)--(1+0.32+0.43,0) (1+0.32+0.43+0.56+0.75,0)--(1+0.32+0.43+0.56+0.75+0.43,0) (1+0.32+0.43+0.56+0.75+3+0.64,0)--(1+0.32+0.43+0.56+0.75+3+0.64+0.43,0) (1+0.32+0.43+0.56+0.75+1,0)--(1+0.32+0.43+0.56+0.75+1.43,0) (1+0.32+0.43+0.56+0.75+2+0.32,0)--(1+0.32+0.43+0.56+0.75+2+0.32+0.43,0) (1+0.32+0.43+0.56,0)--(1+0.32+0.43+0.56+0.43,0);
\draw[very thick,violet] (.43,0)--(1,0) (1+0.32+0.43,0)--(1+0.32+0.43+0.56,0) (1+0.32+0.43+0.56+0.75+0.43,0)--(1+0.32+0.43+0.56+0.75+1,0) (1+0.32+0.43+0.56+0.75+1.43,0)--(1+0.32+0.43+0.56+0.75+2,0) (1+0.32+0.43+0.56+0.75+2+0.32+0.43,0)--(1+0.32+0.43+0.56+0.75+3+0.32,0);
\draw (0,-.1)--(0,.1) (.43,-.1)--(0.43,.1) (.215,.1)node[above]{$3$}  (.7,.1)node[above]{$4$}  (1,-.1)--(1,.1) (1.16,.1)node[above]{$2$} (1+0.32,-.1)--(1+0.32,.1) (1.52,.1)node[above]{$3$} (1+0.32+0.43,-.1)--(1+0.32+0.43,.1) (2,.1)node[above]{$4$} (1+0.32+0.43+0.56,-.1)--(1+0.32+0.43+0.56,.1)  (1+0.32+0.43+0.56+0.22,.1)node[above]{$3$} (1+0.32+0.43+0.56+0.43+0.16,.1)node[above]{$2$} (1+0.32+0.43+0.56+0.75,-.1)--(1+0.32+0.43+0.56+0.75,.1) (1+0.32+0.43+0.56+0.75+0.22,.1)node[above]{$3$} (1+0.32+0.43+0.56+0.75+1+0.22,.1)node[above]{$3$} (1+0.32+0.43+0.56+0.75+0.43+0.28,.1)node[above]{$4$} (1+0.32+0.43+0.56+0.75+1+0.43+0.28,.1)node[above]{$4$} (1+0.32+0.43+0.56+0.75+2+0.22+0.32,.1)node[above]{$3$} (1+0.32+0.43+0.56+0.75+1+0.43+0.28+1+0.32,.1)node[above]{$4$}(1+0.32+0.43+0.56+0.75+1,-.1)--(1+0.32+0.43+0.56+0.75+1,.1) (1+0.32+0.43+0.56+0.75+2,-.1)--(1+0.32+0.43+0.56+0.75+2,.1) (5.2,.1)node[above]{$2$} (1+0.32+0.43+0.56+0.75+2+0.32,-.1)--(1+0.32+0.43+0.56+0.75+2+0.32,.1)  (1+0.32+0.43+0.56+0.75+3+0.32,-.1)--(1+0.32+0.43+0.56+0.75+3+0.32,.1) (6.5,.1)node[above]{$2$} (1+0.32+0.43+0.56+0.75+3+0.64,-.1)--(1+0.32+0.43+0.56+0.75+3+0.64,.1) (6.9,.1)node[above]{$3$} (1+0.32+0.43+0.56+0.75+3+0.64+0.43,-.1)--(1+0.32+0.43+0.56+0.75+3+0.64+0.43,.1) (1+0.32+0.43+0.56+0.75+0.43,-.1)--(1+0.32+0.43+0.56+0.75+0.43,.1) (1+0.32+0.43+0.56+0.75+3+0.64+0.43,-.1)--(1+0.32+0.43+0.56+0.75+3+0.64+0.43,.1) (1+0.32+0.43+0.56+0.75+1.43,-.1)--(1+0.32+0.43+0.56+0.75+1.43,.1) (1+0.32+0.43+0.56+0.43,-.1)--(1+0.32+0.43+0.56+0.43,.1) (1+0.32+0.43+0.56+0.75+2+0.32+0.43,-.1)--(1+0.32+0.43+0.56+0.75+2+0.32+0.43,.1);
\end{tikzpicture}
\caption{Effect of the morphism $\chi$ on the tiling of the line determined by the fixed point of $\sigma_0$.\label{fig:tline}}
\end{figure}
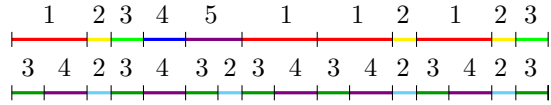

We see now that projecting the vertices of the modified stepped line $\overline{\chi(u)}$ onto $\mathbb{K}_c$ we get the connection with the Rauzy fractal generated by the dual substitution $\mathbf{E}^2(\sigma)$ applied on the geometric element $\mathcal{P}=(\mathbf{0},2\wedge 3)+(\mathbf{0},2\wedge 4)+(\mathbf{0},3\wedge 4)$.

Let $w=\chi(u)$ and consider the sets
\begin{equation}
\widetilde{\mathcal{R}}(a) : = \overline{\{\pi_c(\mathbf{l}(w_0\ldots w_{N-1})): N\in\mathbb{N}, w_{N} = a \}},
\end{equation}
for $a\in\{2,3,4\}$. Let $\widetilde{\mathcal{R}} =\bigcup_{a\in\{2,3,4\}} \widetilde{\mathcal{R}}(a)$.

\subsection{Relations between different definitions of Rauzy fractal}
Our aim is to find relations between our new Rauzy fractals $\mathcal{R}(a\wedge b)$, the classical Rauzy subtiles $\mathcal{R}(a)$ and the $\widetilde{\mathcal{R}}(a)$ obtained modifying the stepped line.

We need a preparatory lemma.

\begin{lemma}\label{le:dec}
We have 
\begin{align*}
\mathcal{R}(2\wedge 3)&=(-\mathcal{R}(1)-\pi_c(\mathbf{e}_1))\cup (-\mathcal{R}(4)-\pi_c(\mathbf{e}_4)), \\ \mathcal{R}(2\wedge 4)&=(-\mathcal{R}(1)-\pi_c(\mathbf{e}_3))\cup (-\mathcal{R}(3)-\pi_c(\mathbf{e}_3))\cup (-\mathcal{R}(5)-\pi_c(\mathbf{e}_3)), \\ \mathcal{R}(3\wedge 4)&=(-\mathcal{R}(2)-\pi_c(\mathbf{e}_2))\cup (-\mathcal{R}(5)-\pi_c(\mathbf{e}_5)). 
\end{align*}
\end{lemma}

\begin{proof}
Since the computations are rather lengthy and technical, we refer the interested reader to Appendix~\ref{app:B}.
\end{proof}

\begin{remark}
We can carry on the computations for the other $\mathcal{R}(a\wedge b)$ in a similar way as above. 
We obtain
\begin{align*}
\mathcal{R}(2\wedge 5) &= (-\mathcal{R}(1)-\pi_c(\mathbf{e}_1)) \cup (-\mathcal{R}(4)-\pi_c(\mathbf{e}_4)+\pi_c(\mathbf{e}_2)), \\
\mathcal{R}(3\wedge 5) &= (-\mathcal{R}(5)-\pi_c(\mathbf{e}_5)) \cup (-\mathcal{R}(1)-\pi_c(\mathbf{e}_4)),
\end{align*}
and for the others just use the set equation $\mathcal{R}(a\wedge b)=M_\sigma \mathcal{R}(a'\wedge b')$ for $a\wedge b \leftarrow a'\wedge b'$.

Similar formulas to express the Rauzy fractals in terms of the subtiles $\mathcal{R}(a)$ hold for the entire family of substitutions $\sigma_t$.
\end{remark}

\begin{proposition}\label{pro:iso}
We have $\widetilde{\mathcal{R}} = -E_\mathcal{P}(\mathcal{R}_\mathcal{P})$, for $\mathcal{P}=(\mathbf{0},2\wedge 3)+(\mathbf{0},2\wedge 4)+(\mathbf{0},3\wedge 4)$.
\end{proposition}
\begin{proof}
Using the definition \eqref{eq:bro} of the subtiles $\mathcal{R}(a)$ as projections of colored vertices of the stepped line and considering the subtiles $\widetilde{\mathcal{R}}(a)$, we see from the morphism $\chi$ that 
\begin{gather*} 
\mathcal{R}(1)\cup\mathcal{R}(3)\cup\mathcal{R}(5) = \widetilde{\mathcal{R}}(3),\quad (\mathcal{R}(1)+\pi_c(\mathbf{e}_3))\cup \mathcal{R}(4) = \widetilde{\mathcal{R}}(4), \\ 
\mathcal{R}(2)\cup(\mathcal{R}(5)+\pi_c(\mathbf{e}_3)) = \widetilde{\mathcal{R}}(2), 
\end{gather*}
hold for $\sigma_0$.
Therefore, using Lemma \ref{le:dec} and the relations \eqref{eq:ratdep}, we get that for each $a\in\{2,3,4\}$
\[  
\widetilde{\mathcal{R}}(a) + \pi_c(\mathbf{e}_a) = - \mathcal{R}(b\wedge c),\quad \text{ for } \{b,c\}\in\{2,3,4\}\setminus\{a\}.  
\]
A similar argument works for any $\sigma_t$ after having adapted the decomposition of Lemma~\ref{le:dec}.
\end{proof}

\subsection{First return}
Since the strong coincidence condition holds, we get that the subtiles $\widetilde{\mathcal{R}}(a)$, for $a\in\{2,3,4\}$, are pairwise measure disjoint. Therefore we can define the domain exchange 
\[ \widetilde{E}: \mathbf{x} \mapsto \mathbf{x} - \pi_c(\mathbf{e}_a), \quad 
\text{if }\mathbf{x}\in\widetilde{\mathcal{R}}(a). \]
By Proposition~\ref{pro:iso} we have $(\widetilde{\mathcal{R}},\widetilde{E}) = (-\mathcal{R}_\mathcal{P}, E^{-1}_\mathcal{P})$.

\medskip

Now we see that $(\widetilde{\mathcal{R}},\widetilde{E})$ can be related to $(\mathcal{R},E)$, where $\mathcal{R}$ is the classical Rauzy fractal defined in \eqref{eq:bro}.
Combinatorially the morphism $\chi$ describes the first return of $\widetilde{E}$ on $\mathcal{R}$. 

\begin{proposition}\label{pro:return}
$E$ is the first return\index{first return} of $\widetilde{E}$ on $\mathcal{R}$.  
\end{proposition}
\begin{proof}
From the decomposition of Lemma \ref{le:dec} and Proposition \ref{pro:iso} we see that $\widetilde{E}(\widetilde{\mathcal{R}}(4))=\widetilde{\mathcal{R}}(4)+\pi_c(\mathbf{e}_4) = (\mathcal{R}(1)+\pi_c(\mathbf{e}_3+\mathbf{e}_4))\cup (\mathcal{R}(4)+\pi_c(\mathbf{e}_4))$ and $\widetilde{E}(\widetilde{\mathcal{R}}(2))=\widetilde{\mathcal{R}}(2)+\pi_c(\mathbf{e}_2) = (\mathcal{R}(2)+\pi_c(\mathbf{e}_2))\cup (\mathcal{R}(5)+\pi_c(\mathbf{e}_3+\mathbf{e}_2))$. Furthermore $\mathcal{R}(1)$, $\mathcal{R}(3)$ and $\mathcal{R}(5)$ are in $\widetilde{\mathcal{R}}(3)$ and applying $\widetilde{E}$ we translate them by $\pi_c(\mathbf{e}_3)$. Using the relations $\pi_c(\mathbf{e}_3)+\pi_c(\mathbf{e}_4)=\pi_c(\mathbf{e}_1)$ and $\pi_c(\mathbf{e}_2)+\pi_c(\mathbf{e}_3)=\pi_c(\mathbf{e}_5)$ we get $\widetilde{E}|_\mathcal{R} = E$.
\end{proof}

\begin{remark}
Note that there are other three morphisms that can be obtained by flipping $\chi(1)$ or $\chi(5)$. The effect of these morphisms on the fixed point of $\sigma_0$ change the shapes of the subtiles $\widetilde{\mathcal{R}}(a)$ and the definition of the domain exchange $\widetilde{E}$ but the results of Proposition~\ref{pro:iso} and \ref{pro:return} remain true.
\end{remark}

\section{Symbolic dynamics}\label{sec:symbdyn}

As in Section~\ref{sec:comb}, we let $\sigma$ be a substitution of the family of Section~\ref{sec:main-ex}.

Given a fixed point $u$ of $\sigma$, let $\Omega=\overline{\{ S^k \chi(u) : k\in\mathbb{N} \}}$, where $\chi$ is the morphism defined by \eqref{chi}.

\begin{lemma}
$(\Omega,S)$ is minimal and uniquely ergodic.
\end{lemma}
\begin{proof}
We know that $(X_\sigma,S)$, where $X_\sigma=\overline{ \{ S^k u: k \in \mathbb{N}\} }$, is minimal. 
This means that every factor of $u$ occurs in $u$ with bounded gaps, and the same happens for $\chi(u)$. Thus $(\Omega,S)$ is minimal.

By primitivity of $\sigma$ the system $(X_\sigma,S,\mu)$ is uniquely ergodic, so the cone
\[ \bigcap_{n\geq 0} M^n_\sigma \mathbb{R}^5_+ = \mathbb{R}_+\mathbf{u}_\beta  \]
is one-dimensional and is parametrized by the right eigenvector $\mathbf{u}_\beta$ (rescaled such that $\|\mathbf{u}_\beta\|_1 = 1$) which coincides with the vector of letter frequencies $(\mu([1]),\ldots,\mu([5]))$. For the system $(\Omega,S)$, with $\Omega = \overline{ \{ S^k \chi(u): k \in \mathbb{N}\} }$, the cone
\[ M_\chi \bigcap_{n\geq 0} M^n_\sigma \mathbb{R}^5_+  \] is also one-dimensional, with $M_\chi$ incidence matrix of $\chi$. Hence $\Omega$ has uniform factor frequencies by~\cite[Theorem~5.7]{Berthe-Delecroix}, which is equivalent to unique ergodicity by \cite[Corollary~4.2]{Queffelec:10}.
\end{proof}
We want to show that the dynamical system $(\Omega,S,m)$, where $m$ is the unique $S$-invariant Borel probability measure on $\Omega$, is measurably conjugate to $(\widetilde{\mathcal{R}},\widetilde{E},\lambda)$.

Let $w=\chi(u)$ and consider the sets
\begin{equation}\label{eq:brored}
\widetilde{\mathcal{R}}(a_{[0,\ell)}) : = \overline{\{\pi_c(\mathbf{l}(w_{[0,N)}): N\in\mathbb{N}, w_{[N,N+\ell)} = a_{[0,\ell)} \}},
\end{equation}
where the notation $a_{[0,\ell)}$ stands for the word $a_0\cdots a_{\ell-1} \in\{2,3,4\}^*$, and similarly for $w_{[0,N)}$. Let $\widetilde{\mathcal{R}} =\bigcup_{a\in\{2,3,4\}} \widetilde{\mathcal{R}}(a)$.

We define the \emph{representation map}
\begin{equation}\label{eq:phi}
\phi : \Omega \to \widetilde{\mathcal{R}},\quad (a_i)_{i\in\mathbb{N}} \mapsto \bigcap_{\ell\in\mathbb{N}} \widetilde{\mathcal{R}}(a_{[0,\ell)}).
\end{equation}

\begin{lemma}
$\phi$ is well-defined, continuous and surjective.
\end{lemma}
\begin{proof}
Let $(a_i)_{i\in\mathbb{N}}\in\Omega$. Then $\widetilde{\mathcal{R}} = \widetilde{\mathcal{R}}(a_{[0,0)}) \supset \widetilde{\mathcal{R}}(a_{[0,1)}) \supset\cdots$, and $\widetilde{\mathcal{R}}(a_{[0,\ell)})\neq\emptyset$ for all $\ell\in\mathbb{N}$. The word $w=\chi(u)$ is uniformly recurrent, since $u$ is generated by a primitive substitution and $\chi$ does not affect the uniformly recurrence. Thus we have a sequence $(\ell_k)_{k\in\mathbb{N}}$ such that $a_{[\ell_k,\ell_k+k)} = w_{[0,k)}$, for all $k\in\mathbb{N}$. Since $\widetilde{\mathcal{R}}(a_{[0,\ell_k+k)}) \subset \widetilde{\mathcal{R}}(a_{[\ell_k,\ell_k+k)}) - \pi_c(\mathbf{l}(a_{[0,\ell_k)}))$, we need to show that the diameter of $\widetilde{\mathcal{R}}(a_{[\ell_k,\ell_k+k)}) = \widetilde{\mathcal{R}}(w_{[0,k)})$ converges to zero. Let $\mathcal{S}_k = \{ \pi_c(\mathbf{l}(w_{[0,j)}): 0\le j \le k \}$. Then $\widetilde{\mathcal{R}}(w_{[0,k)}) + \mathcal{S}_k \subset \widetilde{\mathcal{R}}$ for all $k\in \mathbb{N}$, and $\lim_{k\to\infty} \mathcal{S}_k = \widetilde{\mathcal{R}}$ with respect to the Hausdorff metric. But this implies that $\lim_{k\to\infty}\widetilde{\mathcal{R}}(w_{[0,k)}) = \{\mathbf{0}\}$, which proves that $\phi$ is well-defined.

The map $\phi$ is continuous since the sequence $(\widetilde{\mathcal{R}}(a_{[0,\ell)}))_{\ell\in\mathbb{N}}$ is nested and converges to a single point. The surjectivity follows from a Cantor diagonal argument.
\end{proof}

\begin{lemma}\label{pro:conj}
$(\Omega,S,m)$ is measurably conjugate to $(\widetilde{\mathcal{R}},\widetilde{E},\lambda)$ via $\phi$. 
\end{lemma}
\begin{proof}
The collections $\mathcal{K}_i=\{\widetilde{\mathcal{R}}(a_{[0,i)}) : a_{[0,i)}\in L_i(w)\}$, where $L_i(w)$ is the set factors of length $i$ of $w$, are measure-theoretic partitions of $\widetilde{\mathcal{R}}$ and $\widetilde{\mathcal{R}}(a_{[0,i)}) = \bigcap_{j=0}^{i-1} \widetilde{E}^{-j}\widetilde{\mathcal{R}}(a_j)$. Hence $\phi(w')\neq\phi(w'')$ for all $w', w''\in \phi^{-1}(\widetilde{\mathcal{R}}\setminus \bigcup_{i\in\mathbb{N}, K\in\mathcal{K}_i} \partial K)$ with $w' \neq w''$, and $m(\phi^{-1}(\bigcup_{i\in\mathbb{N}, K\in\mathcal{K}_i} \partial K)) = 0$ since we have $\lambda(\partial K) = 0$ for all $K\in \mathcal{K}_i$, $i\in\mathbb{N}$, by Proposition \ref{prop:prop}, and $\lambda\circ\phi$ is an $S$-invariant Borel measure, which equals $m$ by unique ergodicity of $(\Omega,S)$. Thus the map $\phi$ is injective almost everywhere. Finally $\phi((a_k)_{k\in\mathbb{N}})$ is a single point $\mathbf{z}=\bigcap_{\ell\in\mathbb{N}}\widetilde{\mathcal{R}}(a_{[0,\ell)})$. 
Since $\widetilde{\mathcal{R}}(a_{[0,\ell+1)})+\pi_c(\mathbf{e}_{a_0})\subset \widetilde{\mathcal{R}}(a_{[1,\ell+1)})$, for all $\ell\in\mathbb{N}$, we obtain that 
$\widetilde{E}(\mathbf{z})=\mathbf{z}+\pi_c(\mathbf{e}_{a_0})=\bigcap_{\ell\in\mathbb{N}}\widetilde{\mathcal{R}}(a_{[1,\ell)})$, but this is the same as shifting $(a_k)_{k\in\mathbb{N}}$ and applying $\phi$. Thus we checked that $\widetilde{E}\circ\phi = \phi\circ S$.
\end{proof}

\begin{theorem}\label{thm:conj}
For each $\sigma_t$ we have the following commutative diagram
\[ \xymatrix{X_{\sigma_t} \ar[r]^\chi\ar[d]_S & \Omega \ar[r]^\phi\ar[d]_S & \widetilde{\mathcal{R}} \ar[r]^{\pi_0} \ar[d]_{\widetilde{E}} & \mathbb{T}^2 \ar[d]_{\widetilde{E}} \\ X_{\sigma_t} \ar[r]^\chi & \Omega\ar[r]^\phi & \widetilde{\mathcal{R}}\ar[r]^{\pi_0}  & \mathbb{T}^2 } \] 
where the maps $\phi$ defined by \eqref{eq:phi} and the natural projection $\pi_0 : \mathbb{K}_c\to \mathbb{K}_c/\pi_c(\Lambda_\mathcal{P})\cong\mathbb{T}^2$ are measure-theoretical isomorphisms, while $\chi$ is the morphism \eqref{chi}. Thus $(X_{\sigma_t},S)$ is a first return of a translation on $\mathbb{T}^2$.
\end{theorem}
\begin{proof}
Notice that the relations (\ref{eq:ratdep}) hold for the whole family of $\sigma_t$, thus we can use the morphism $\chi$ for each $\sigma_t$. 
By Proposition~\ref{pro:iso}, 
\[ \widetilde{\mathcal{R}} = \bigcup_{a\in\{2,3,4\}}\widetilde{\mathcal{R}}(a) = -E_{\mathcal{P}}(\mathcal{R}_\mathcal{P}),\quad \text{ for } \mathcal{P} = (\mathbf{0},2\wedge 3)+(\mathbf{0},2\wedge 4)+(\mathbf{0},3\wedge 4), \] 
and $(\widetilde{\mathcal{R}},\widetilde{E})=(\mathcal{R},E_\mathcal{P}^{-1})$. Thus, by Corollary~\ref{cor:tilings}, $\widetilde{\mathcal{R}}$ is a fundamental domain of $\mathbb{T}^2$ and $\widetilde{E}$ projects by $\pi_0$ to a translation on $\mathbb{T}^2$. By Lemma \ref{pro:conj} $\phi$ is a measurable conjugation, thus the words in $\Omega$ are natural codings of this toral translation.

Finally, since $(X_{\sigma_t},S)$ is measurably conjugate to $(\mathcal{R},E)$ and $E$ is the first return of $\widetilde{E}$ on $\mathcal{R}$ by Proposition~\ref{pro:return}, we have that $(X_{\sigma_t},S)$ is a first return of the translation $\pi_0(\widetilde{E})$ on $\mathbb{T}^2$.
\end{proof}

\section{Perspectives}\label{sec:persp}

Future works will tackle the following problems.
\subsection{Non-projecting-well substitutions and neutral space}\label{sec:nonreg}

Consider the family of substitutions
\[ \sigma_t : 1\mapsto 1^{t-1}2,\; 2\mapsto 1^{t-1}3,\; 3\mapsto 4,\; 4\mapsto 1,\quad (t\ge 2)\] 
with characteristic polynomial $f(x)g(x)=(x^3-tx^2+x-1)(x+1)$. The rational dependency relation is $\pi(\mathbf{e}_1 + \mathbf{e}_3) = \pi(\mathbf{e}_2 + \mathbf{e}_4)$. Since $n=4$ and $d=3$, we deal with $\mathbf{E}_2^*(\sigma)$ and its geometric realization $\mathbf{E}^2(\sigma)$.

We observe from Figure~\ref{fig:nonreg} that the substitution $\sigma_2$ does not project well. Overlaps of this type can be observed for the whole family.
See also \cite{Furukado:06} for some similar polygonal overlaps obtained in the framework of non-Pisot unimodular matrices.

\begin{figure}[h]
\includegraphics[scale=.2]{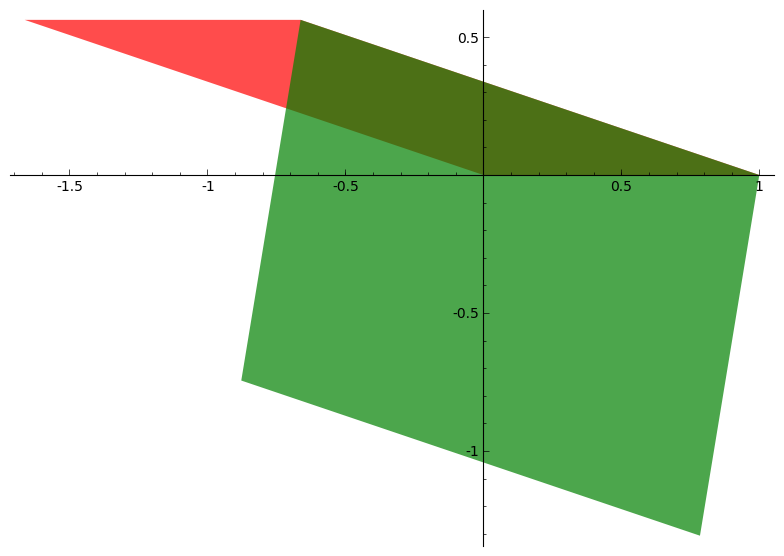} \qquad 
\includegraphics[scale=.45]{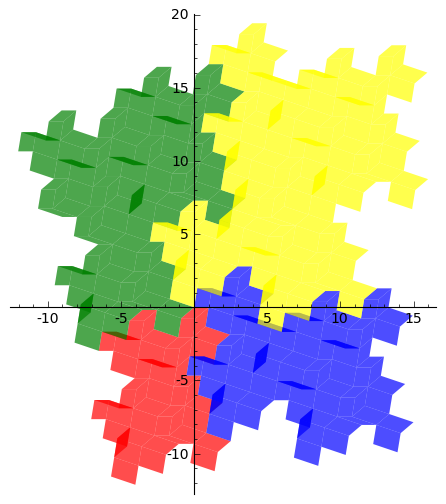} 
\caption{$\pi_c(\ov{\mathbf{E}^2(\sigma)(\mathbf{0},2\wedge 4)}) = \pi_c\ov{(\mathbf{0},1\wedge 3)} \cup \pi_c\ov{(\mathbf{e}_1-\mathbf{e}_4,3\wedge 4)}$ for $\sigma_2$ (left) and $\mathbf{E}^2(\sigma_1)$-iterate of a geometric element and its projection onto $\mathbb{K}_c$ (right).\label{fig:nonreg}}
\end{figure}

This problem could be solved by considering a different projection onto $\mathbb{K}_c$, or by the "retiling" method introduced in \cite{FIR}. This involves an accurate study of positivity and cancellation, together with manipulations, such as flips, on stepped surfaces, in the spirit of \cite{ABFJ,BF11}. 

The main question is whether we can generalize the constructions of this paper to any reducible Pisot substitution, improving the characterization of the stepped surfaces understanding what happens with respect to the neutral space. Some interesting studies on the role of the neutral space in the geometry and dynamics of reducible substitutions have been initiated in \cite{ABB}.

\subsection{Connections with irreducible substitutions}
Figure~\ref{fig:makeirr} suggests that we can change projection and view stepped surfaces considering a smaller number of faces. 
The study of the new combinatorial approach based on morphisms and modified stepped lines of Section~\ref{sec:comb} deserves more investigations. 
It is curious to observe that, using the morphisms $\chi$ defined in \eqref{chi} and $\bar{\chi}$ obtained from $\chi$ by flipping the image of $1$, we get a connection between the Hokkaido substitution $\sigma: 1 \mapsto 21, 2 \mapsto 3, 3 \mapsto 4, 4 \mapsto 5, 5 \mapsto 1$ and the irreducible substitution $\sigma_{\text{irr}}: 2 \mapsto 3, 3 \mapsto 4, 4 \mapsto 23$, having same Pisot polynomial:
\[ \xymatrix{\{1,\ldots,5\}^* \ar[r]^\sigma\ar[d]_{\bar{\chi}} & \{1,\ldots,5\}^* \ar[d]^{\chi}  \\  \{2,3,4\}^* \ar[r]^{\sigma_{\text{irr}}} & \{2,3,4\}^* } \] 
Can we in general connect the study of the dynamics of a reducible substitution to that of an irreducible one?

We saw in Theorem~\ref{thm:conj} that for a family of reducible Pisot substitutions $(X_\sigma,S)$ is the first return of a toral translation. Induced dynamics can have different behavior, thus it is natural to ask in which cases these first returns have pure discrete spectrum (see \cite{Rau84} for connections with \emph{bounded remainder sets}).

\begin{figure}
\begin{center}$
\begin{array}{cc}
\includegraphics[scale=.4]{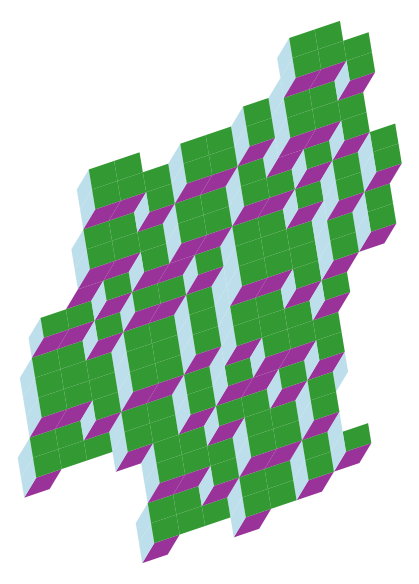} & \qquad
\includegraphics[scale=.4]{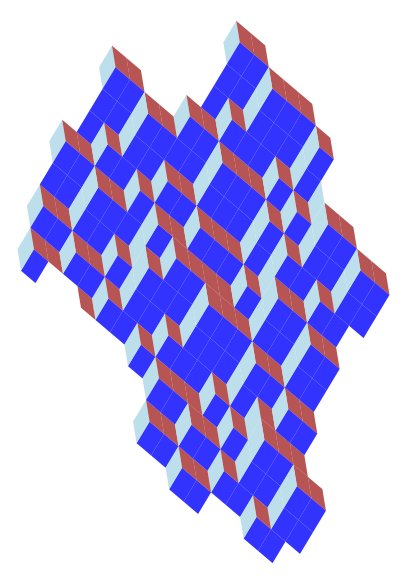}
\end{array}$
\end{center}
\caption{It is interesting to notice that changing suitably the projection we get different polygonal tilings by some faces of three different types.\label{fig:makeirr}}
\end{figure}

\subsection{Contact graphs}
Further research is motivated by the fact that the existence of polygonal approximations for the Rauzy fractals allows to define contact graphs (see \cite{T06,ST09}) even in the reducible case for the study of the fractal boundaries.

\appendix

\section{Proof of Lemma~\ref{lem:meseigenvec}}

\begin{lemma}\label{lem:bigwedge} Let $n=p+m$, $M\in\mathbb{R}^{n\times n}$ and $\{{\bf e}_i\}_{i=1,\ldots,n}$ be the canonical basis of $\mathbb{R}^n$. Moreover, let $\und{a}=a_1\wedge\cdots\wedge a_m\in \mathcal{O}_m$ and ${\bf x}_1,\ldots,{\bf x}_p\in\mathbb{R}^n$. Then 
$$\begin{array}{c}
\det(M{\bf x}_1,\ldots,M{\bf x}_p,{\bf e}_{a_1},\ldots{\bf e}_{a_m})
\\\displaystyle
=\sum_{\und{b}=b_1\wedge\cdots\wedge b_{m}\in\mathcal{O}_m}
\left(\bigwedge_{i=1}^{p}M_\sigma\right)_{\und{a}^*,\und{b}^*}(-1)^{\und{a}+\und{b}}
\det({\bf x}_1,\ldots,{\bf x}_{p},{\bf e}_{b_1},\ldots,{\bf e}_{b_{m}}).
\end{array}
$$
\end{lemma}

\begin{proof}
We write ${\bf x}_k=\sum_{j=1}^n{\bf x}_k^{(j)}{\bf e}_j$ for $k\in\{1,\ldots,p\}$, $M=(m_{ij})_{1\leq i,j\leq n}$ and compute the determinant:
$$\begin{array}{c}
\det(M{\bf x}_1,\ldots,M{\bf x}_p,{\bf e}_{a_1},\ldots{\bf e}_{a_m})\\
\displaystyle
=
\det\left(\sum_{i_1=1}^n\sum_{j_1=1}^nm_{i_1j_1}{\bf x}_1^{(j_1)}{\bf e}_{i_1},\ldots,
\sum_{i_p=1}^n\sum_{j_p=1}^nm_{i_pj_p}{\bf x}_p^{(j_p)}{\bf e}_{i_p},
{\bf e}_{a_1},\ldots{\bf e}_{a_m}\right)
\\\displaystyle
=\sum_{1\leq j_1,\ldots,j_p\leq n}{\bf x}_1^{(j_1)}\cdots{\bf x}_p^{(j_p)}
\det\left(\sum_{i_1=1}^nm_{i_1j_1}{\bf e}_{i_1},\ldots,\sum_{i_p=1}^nm_{i_pj_p}{\bf e}_{i_p},{\bf e}_{a_1},\ldots{\bf e}_{a_m}\right)\\
\displaystyle=\sum_{\begin{array}{c}1\leq j_1,\ldots,j_p\leq n,\\\#\{j_1,\ldots, j_p\}=p\end{array}}{\bf x}_1^{(j_1)}\cdots{\bf x}_p^{(j_p)}
\det\left(\sum_{i_1=1}^nm_{i_1j_1}{\bf e}_{i_1},\ldots,\sum_{i_p=1}^nm_{i_pj_p}{\bf e}_{i_p},{\bf e}_{a_1},\ldots{\bf e}_{a_m}\right).
\end{array}
$$
Remember that $\left(\bigwedge_{i=1}^{p}M\right)_{\und{a}^*,\und{b}^*}$ is the $p\times p$ minor of $M$ obtained by deleting the rows $a_1,\ldots,a_m$ and the columns $b_1,\ldots,b_m$ of $M$, where $\und{a}=a_1\wedge\cdots\wedge a_m$ and $\und{b}=b_1\wedge\cdots\wedge b_m$. Therefore, by definition, 
$$
\begin{array}{c}
\displaystyle
\det\left(\sum_{i_1=1}^nm_{i_1j_1}{\bf e}_{i_1},\ldots,\sum_{i_p=1}^nm_{i_pj_p}{\bf e}_{i_p},{\bf e}_{a_1},\ldots{\bf e}_{a_m}\right)\\
\displaystyle=(-1)^{\und{a}+((p+1)+\cdots+n)} \left(\bigwedge_{i=1}^{p}M\right)_{\und{a}^*,\und{j}'^*} {\rm sgn}(\tau).
\end{array}
$$
Here, $\und{j}'=j_1'\wedge\cdots\wedge j_m'\in\mathcal{O}_m$ such that $\{j_1,\ldots,j_p\}\cup\{j_1',\ldots,j_m'\}=\{1,\ldots,n\}$, $\tau$ is the permutation of $\{1,\ldots,n\}$ satisfying $\tau(j_1')=j_1',\ldots,\tau(j_m')=j_m'$ and $\tau(j_1)\leq \cdots \leq\tau(j_p)$ and ${\rm sgn}(\tau)$ its signature. We denote the set of permutations on $\{1,\ldots,n\}$ by ${\rm Per}_n$. In this way, we obtain
$$
\begin{array}{l}\displaystyle
\det(M{\bf x}_1,\ldots,M{\bf x}_p,{\bf e}_{a_1},\ldots{\bf e}_{a_m})
\\\displaystyle
= \sum_{\begin{array}{c}1\leq j_1,\ldots,j_p\leq n,\\\#\{j_1,\ldots, j_p\}=p\end{array}}\left(\bigwedge_{i=1}^{p}M\right)_{\und{a}^*,\und{j}'^*}(-1)^{\und{a}+((p+1)+\cdots+n)}
 {\rm sgn}(\tau){\bf x}_1^{(j_1)}\cdots{\bf x}_p^{(j_p)}\\
=\displaystyle\sum_{\und{j}=j_1\wedge\ldots\wedge j_p\in\mathcal{O}_p}\left(\bigwedge_{i=1}^{p}M\right)_{\und{a}^*,\und{j}'^*}(-1)^{\und{a}+((p+1)+\cdots+n)}
\sum_{\begin{array}{c}\tau\in {\rm Per}_n,\\\tau(j_k')=j_k'\\(1\leq k\leq m)\end{array}}{\rm sgn}(\tau){\bf x}_1^{(\tau(j_1))}\cdots{\bf x}_p^{(\tau(j_p))}.
 \end{array}
$$
Note that 
$$\begin{array}{rcl}\displaystyle\sum_{\begin{array}{c}\tau\in {\rm Per}_n,\\\tau(j_k')=j_k'\\(1\leq k\leq m)\end{array}}{\rm sgn}(\tau){\bf x}_1^{(\tau(j_1))}\cdots{\bf x}_p^{(\tau(j_p))}
=(-1)^{\und{j}'+((p+1)+\cdots+n)}\det({\bf x}_1,\ldots,{\bf x}_p,{\bf e}_{j_1'},\ldots,{\bf e}_{j_m'}),
\end{array}
$$
which leads to the desired equality, after renaming $\und{j}'\leftrightarrow \und{b}$:
\[ 
\begin{array}{l}
\det(M{\bf x}_1,\ldots,M{\bf x}_p,{\bf e}_{a_1},\ldots{\bf e}_{a_m})
\\\displaystyle
\displaystyle=\sum_{\und{b}=b_1\wedge\ldots\wedge b_m\in\mathcal{O}_m}\left(\bigwedge_{i=1}^{p}M\right)_{\und{a}^*,\und{b}^*}(-1)^{\und{a}+\und{b}}
\det({\bf x}_1,\ldots,{\bf x}_p,{\bf e}_{b_1},\ldots,{\bf e}_{b_m}).
\end{array} \qedhere
\]
\end{proof}

\begin{proof}[Proof of Lemma~~\ref{lem:meseigenvec}]
Let ${\bf u}_{d+1},\ldots,\mathbf{u}_{n}$ be a basis of $\mathbb{K}_n$ made of eigenvectors of $M_\sigma$ for the associated eigenvalues $\zeta_{d+1},\ldots,\zeta_n$. Also, let ${\bf v}_{d+1},\ldots,\mathbf{v}_{n}$ be eigenvectors of ${}^tM_\sigma$ for the associated eigenvalues $\zeta_{d+1},\ldots,\zeta_n$, normalized to have $\mathbf{u}_k\cdot\mathbf{v}_k=1$ for $k\in\{d+1,\ldots,n\}$. Since $\mathbf{v}_1,\mathbf{v}_{d+1},\ldots,\mathbf{v}_n$ are all orthogonal to the contracting space $\mathbb{K}_c$, the following equality holds for every $\und{a}=a_1\wedge\cdots\wedge a_{d-1}\in\mathcal{O}_{d-1}$:
\begin{equation}\label{voldet}\displaystyle\lambda_{d-1}(\pi_c\ov{(\mathbf{0},\und{a})})=
\frac{\left|\det\left(\mathbf{v}_1,\mathbf{v}_{d+1},\ldots,\mathbf{v}_{n},\pi_c({\bf e}_{a_1}),\ldots,\pi_c({\bf e}_{a_{d-1}})\right)\right|}{\lambda_{\bar n}\left(\mathbf{v}_1,\mathbf{v}_{d+1},\ldots,\mathbf{v}_{n}\right)}.
\end{equation}
Here, $\lambda_p$ denotes the Lebesgue measure on the $p$-dimensional space and $$\lambda_{\bar n}\left(\mathbf{v}_1,\mathbf{v}_{d+1},\ldots,\mathbf{v}_{n}\right)=:1/K_0$$ is the measure of the parallelotope generated by the vectors $\mathbf{v}_1,\mathbf{v}_{d+1},\ldots,\mathbf{v}_{n}$. 

Note that for all ${\bf x}\in \mathbb{R}^n$, we have the decomposition
\begin{equation}\label{decompRn}{\bf x}=\langle{\bf x},{\bf v}_1\rangle{\bf u}_1+\sum_{i=d+1}^n\langle{\bf x},{\bf v}_i\rangle{\bf u}_i+\pi_c({\bf x}).
\end{equation}
We use the above decomposition for the vectors $\mathbf{v}_1,\mathbf{v}_{d+1},\ldots,\mathbf{v}_{\bar n}$ and expand the  determinant  by multilinearity to obtain:
$$\begin{array}{rcl}
\displaystyle\lambda_{d-1}(\pi_c\ov{(\mathbf{0},\und{a})})&=&\displaystyle \underbrace{K_0\det\left(\langle{\bf v}_k,{\bf v}_l\rangle_{k,l=1,d+1,\ldots,n}\right)}_{=:K}\\
&&\cdot \displaystyle\;\left|\det({\bf u}_1,{\bf u}_{d+1},\ldots,{\bf u}_{n},\pi_c({\bf e}_{a_1}),\ldots,\pi_c({\bf e}_{a_{d-1}}))\right|.
\end{array}
$$
The terms containing $\pi_c({\bf v}_k)$  vanished for each value of $k\in\{1,d+1,\ldots,n\}$), since the $d$ vectors $\pi_c({\bf v}_k),\pi_c({\bf e}_{a_1}),\ldots,\pi_c({\bf e}_{a_{d-1}})$ are linearly dependent in the $d-1$-dimensional space $\mathbb{K}_c$. 
Using now~(\ref{decompRn}) for the vectors ${\bf e}_{a_1},\ldots,{\bf e}_{a_{d-1}}$, we can simplify the last expression to
\begin{equation}\label{eq:mesdet}
\displaystyle\lambda_{d-1}(\pi_c\ov{(\mathbf{0},\und{a})})=\displaystyle K
\left|\det({\bf u}_1,{\bf u}_{d+1},\ldots,{\bf u}_{n},{\bf e}_{a_1},\ldots,{\bf e}_{a_{d-1}})\right|.
\end{equation}

Now, by (N), $g(0)=1=\Pi_{k=d+1}^n\zeta_k$ and therefore
$$\begin{array}{rcl}
\det({\bf u}_1,{\bf u}_{d+1},\ldots,{\bf u}_{n},{\bf e}_{a_1},\ldots,{\bf e}_{a_{d-1}})&=&\displaystyle \frac{1}{\beta}\det(\beta{\bf u}_1,\zeta_{d+1}{\bf u}_{d+1},\ldots,\zeta_n{\bf u}_{n},{\bf e}_{a_1},\ldots,{\bf e}_{a_{d-1}})\\\\
&=&\displaystyle \frac{1}{\beta}\det(M_\sigma {\bf u}_1,M_\sigma{\bf u}_{d+1},\ldots,M_\sigma{\bf u}_{n},{\bf e}_{a_1},\ldots,{\bf e}_{a_{d-1}}).
\end{array}
$$

By Lemma~\ref{lem:bigwedge}, we can relate the last determinant to $\bigwedge_{i=1}^{\bar n}M_\sigma$: for all 
$\und{a}=a_1\wedge\cdots\wedge a_{d-1}\in\mathcal{O}_{d-1}$, we have
$$
\begin{array}{l}
\det(M_\sigma {\bf u}_1,M_\sigma{\bf u}_{d+1},\ldots,M_\sigma{\bf u}_{n},{\bf e}_{a_1},\ldots,{\bf e}_{a_{d-1}})
\\\displaystyle
=\sum_{\und{b}=b_1\wedge\cdots\wedge b_{d-1}\in\mathcal{O}_{d-1}}
\left(\bigwedge_{i=1}^{\bar n}M_\sigma\right)_{\und{a}^*,\und{b}^*}(-1)^{\und{a}+\und{b}}
\det({\bf u}_1,{\bf u}_{d+1},\ldots,{\bf u}_{n},{\bf e}_{b_1},\ldots,{\bf e}_{b_{d-1}})\\
\displaystyle=\sum_{\und{b}\in\mathcal{O}_{d-1}}
\left(\bigwedge_{i=1}^{\bar n}{}^tM_\sigma\right)_{\und{b}^*,\und{a}^*}(-1)^{\und{a}^*+\und{b}^*}
\det({\bf u}_1,{\bf u}_{d+1},\ldots,{\bf u}_{n},{\bf e}_{b_1},\ldots,{\bf e}_{b_{d-1}})\\
\displaystyle=\sum_{\und{b}\in\mathcal{O}_{d-1}}(M_{d-1})_{\und{b},\und{a}}\det({\bf u}_1,{\bf u}_{d+1},\ldots,{\bf u}_{n},{\bf e}_{b_1},\ldots,{\bf e}_{b_{d-1}}).
\end{array}
$$
We used here that $(-1)^{\und{a}+\und{b}}=(-1)^{\und{a}^*+\und{b}^*}$. By the above computation, this means that the vector 
$${\bf V}=\left(\det({\bf u}_1,{\bf u}_{d+1},\ldots,{\bf u}_{n},{\bf e}_{a_1},\ldots,{\bf e}_{a_{d-1}})\right)_{\und{a}\in\mathcal{O}_{d-1}}
$$ is an eigenvector of ${}^tM_{d-1}$ for the eigenvalue $\beta$. It follows that
$$\beta |\mathbf{V}|=|{}^tM_{d-1}\mathbf{V}|\leq |{}^tM_{d-1}||\mathbf{V}|.
$$
Now, it follows from~(\ref{eq:MkRel}) and (P) that $|{}^tM_{d-1}|=\bigwedge_{i=1}^{\bar n}M_\sigma$ is a primitive matrix. Therefore, by Lemma~\ref{le:modone}, $\beta$ is its Perron-Frobenius eigenvalue. Consequently, the inequality $\beta |\mathbf{V}|\leq |{}^tM_{d-1}||\mathbf{V}|$ is an equality and 
$$\left(\lambda_{d-1}(\pi_c\ov{(\mathbf{0},\und{a})})\right)_{\und{a}\in\mathcal{O}_{d-1}}= K|{\bf V}|
$$
is an eigenvector of $|{}^tM_{d-1}|$ for the eigenvalue $\beta$.
\end{proof}

\section{Proof of Lemma~\ref{le:dec}}\label{app:B}

The classical Rauzy fractal subtiles $\mc{R}(a)$ can be described via Dumont-Thomas numeration (see \cite{DT89}) as
\begin{equation}\label{eq:num1} \mc{R}(a) =  \Big\{\sum_{i\ge 0} \pi_c (M_\sigma^i\,\mathbf{l}(p_i)): (p_i)_{i\ge 0} \in \mathcal{G}_p(a)\Big\}, \end{equation}
where $\mathcal{G}_p(a)$ denotes the set of labels of infinite paths in the prefix graph of the substitution ending at $a\in\mathcal{A}$ (for more details see e.g.~\cite{CS,BS}). We can follow as well infinite paths in the suffix graph. If we do this we get instead
\begin{equation}\label{eq:num2}  -\mc{R}(a)-\pi_c(\mathbf{e}_a) =  \Big\{\sum_{i\ge 0} \pi_c (M_\sigma^i\,\mathbf{l}(s_i)): (s_i)_{i\ge 0} \in \mathcal{G}_s(a)\Big\}, \end{equation}
where $\mathcal{G}_s(a)$ denotes the set of labels of infinite paths in the suffix graph of the substitution ending at $a\in\mathcal{A}$ (see \cite[Section 5]{CS}).

We can use Dumont-Thomas numeration to describe also the subtiles $\mathcal{R}(a\wedge b)$ using the {\it $\mathbf{E}^2(\sigma)$-suffix graph} that is defined as follows.

\begin{definition}\label{def:sufgraph}
The {\it $\mathbf{E}^2(\sigma)$-suffix graph} has set of vertices $\{\und{a}^*:\und{a}\in\wedge^3\mathcal{A}\}$, and there is an edge $\und{a}^* \xrightarrow{\mathbf{s}}\und{b^*}$ if and only if $\sigma(\und{a}) = \mathbf{p}\und{b}\mathbf{s}$, or equivalently if and only if $(M_\sigma^{-1}\mathbf{l}(\mathbf{s}),\und{a}^*) \in \mathbf{E}^2(\sigma)(\mathbf{0},\und{b}^*)$.
\end{definition}

In Figure~\ref{fig:graph} the $\mathbf{E}^2(\sigma_t)$-suffix graph is depicted. Observe that this graph with reversed edges describes the images of every face by $\mathbf{E}^2(\sigma_t)$.

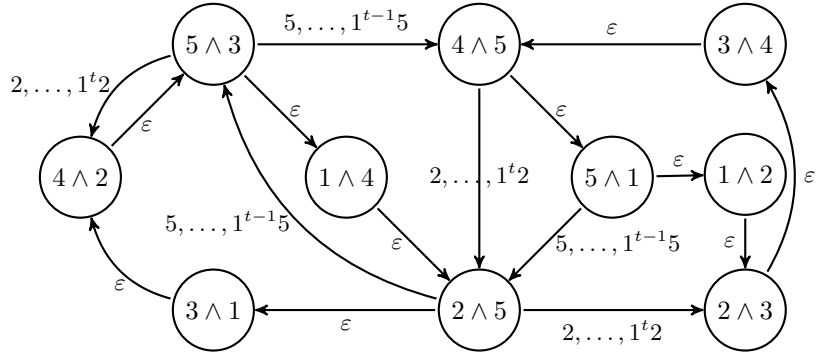
\begin{figure}
\begin{tikzpicture}[<-,>=stealth',shorten >=1pt,auto,node distance=2.5cm,
                    thick,main node/.style={circle,draw,font=\sffamily\bfseries}]

  \node[main node] (1) {$5\wedge 3$};
  \node[main node] (2) [below left of=1] {$4\wedge 2$};
  \node[main node] (3) [below right of=2] {$3\wedge 1$};
  \node[main node] (4) [below right of=1] {$1\wedge 4$};
  \node[main node] (5) [below right of= 4] {$2\wedge 5$};
  \node[main node] (6) [above right of= 4] {$4\wedge 5$};
  \node[main node] (7) [below right of=6] {$5\wedge 1$};
  \node[main node] (10) [below right of=7] {$2\wedge 3$};
  \node[main node] (9) [ above = 0.7cm of 10] {$1\wedge 2$};
  \node[main node] (8) [above right of=7] {$3\wedge 4$};

  \path[every node/.style={font=\sffamily\small}]
    (1) edge node [below] {$\varepsilon$} (2)    
        edge [bend right] node [left] {$5,\ldots,1^{t-1}5$} (5)    
    (2) edge [bend left] node [left] {$2,\ldots,1^t2$} (1)
        edge [bend right] node[below] {$\varepsilon$} (3)
    (3) edge node [below] {$\varepsilon$} (5)
    (4) edge node [right] {$\varepsilon$} (1)
    (5) edge node [left] {$\varepsilon$} (4)
        edge node [anchor=mid] {$2,\ldots, 1^t2$} (6)
        edge node [right] {$5,\ldots,1^{t-1}5$} (7)
    (6) edge node [above] {$5,\ldots,1^{t-1}5$} (1)
        edge node [above] {$\varepsilon$} (8)
    (7) edge node [right] {$\varepsilon$} (6)
    (8) edge [bend left] node [right] {$\varepsilon$} (10) 
    (9) edge node [above] {$\varepsilon$} (7)
    (10) edge node [below] {$2,\ldots,1^t2$} (5)
         edge node [left] {$\varepsilon$} (9);
\end{tikzpicture}
\caption{The $\mathbf{E}^2(\sigma_t)$-suffix graph. 
Note that an edge labeled by $5,\ldots,1^{t-1}5$ denotes that there exist $t$ edges, one labeled by $5$, one by $15$, and so on, until $1^{t-1}5$ (if $t=0$ then there is no edge of this type). The same is valid for the edges labeled by $2,\ldots,1^{t}2$  (for $t=0$ there is only one edge labeled by $2$).
\label{fig:graph}}
\end{figure}
 
\begin{proposition}
We have
\[ \mc{R}(a\wedge b) = \left\{ \sum_{i\ge 0}\pi_c(M_\sigma^i \mathbf{l}(\mathbf{s}_i)) :  (\mathbf{s}_i)_{i\ge 0} \in \mathcal{G}_{s}(a\wedge b)\right\} \]
where $\mathcal{G}_s(a\wedge b)$ denotes the set of labels of infinite walks in the $\mathbf{E}^2(\sigma)$-suffix graph ending at state $a\wedge b$.
\end{proposition}
\begin{proof}
This is a direct consequence of Proposition~\ref{prop:seteq} and of the definition of $\mathbf{E}^2(\sigma)$.
\end{proof}

By abuse of notation we will write $0$ instead of $\epsilon$ by reading labels of walks in the suffix or $\mathbf{E}^2(\sigma_0)$-suffix graphs.

We will relate the elements $\sum_{i\ge 0} \pi_c(M_\sigma^i \mathbf{l}(\mathbf{s}_i))$ with $(\mathbf{s}_i)_{i\ge 0}\in \mathcal{G}_{s}(a\wedge b)$ with those $\sum_{i\ge 0} \pi_c(M_\sigma^i \mathbf{l}(\mathbf{s}_i))$ for $(\mathbf{s}_i)_{i\ge 0} \in \mathcal{G}_s(a)$.

For the Hokkaido substitution $\sigma_0$ we have
\[ \mathcal{G}_s(a) = \{ (\mathbf{s}_i)_{i\geq 0} = \cdots 0^5 2^{k_2} 0^5 2^{k_1} 0^{a-1} : 0\leq k_i \leq \infty\}. \]

\begin{proof}[Proof of Lemma~~\ref{le:dec}]
Observe that 
\begin{equation}\label{eq:rel}
\pi_e(M_\sigma^3\mathbf{e}_2) = \pi_e(M_\sigma\mathbf{e}_2) + \pi_e(\mathbf{e}_2) \quad\text{i.e.}\quad \delta(2000.) = \delta(0022.).
\end{equation}
where $\delta(w) = \sum_{i=0}^{|w|} \pi_e(M_\sigma^i \mathbf{l}(w_i))$. Notice that we can extend $\delta$ to infinite strings $(s_i)_{i\geq 0}$. 
We will prove using \eqref{eq:rel} that $\delta(\mc{G}_s(2\wedge 3))=\delta(\mathcal{G}_s(1)\cup\mathcal{G}_s(4))$. For this reason we will write $w =_\delta w'$ if $\delta(w)=\delta(w')$. The cycle
\begin{center}
\begin{tikzpicture}[->,>=stealth',shorten >=1pt,auto,node distance=2cm]

  \node[] (1) {$2\wedge 3$};
  \node[] (2) [right of=1] {$2\wedge 5$};
  \node[] (3) [right of=2] {$2\wedge 4$};
  \node[] (4) [right of=3] {$3\wedge 5$};
  \node[] (5) [right of=4] {$2\wedge 5$};
  \node[] (6) [right of=5] {$2\wedge 3$};
  \path[every node/.style={font=\sffamily\small}]
    (5) edge node[above] {$2$} (6)
    (4) edge node[above] {$00$} (5)
    		edge [loop above] node[right] {$(20)^k$} (4)
    (3) edge node[above] {$0$} (4)
    (2) edge node[above] {$00$} (3)
    (1) edge node[above] {$002$} (2);
\end{tikzpicture}
\end{center}
in the graph of Figure~\ref{fig:graph} produces strings of type $0020^3(20)^k002=_\delta 0^5 2^{2k+2}$. Starting from state $2\wedge 5$ we get strings $0^5 2^{2k+1}$. Walking from the first node $2\wedge 5$ to the second $2\wedge 5$ returns $0^5 2^{2k}$ and extending this walk to the left starting from $2\wedge 3$ we obtain $0^5 2^{2k+3}$. 
Walking in Figure~\ref{fig:graph} from $2\wedge 3$ to $2\wedge 5$ we get the word $0022002=_\delta 2 0^5 2$. Thus, $\delta(\mc{G}_s(2\wedge 3)) = \delta(\cdots 0^52^{k_2}0^52^{k_1})$, with $0\leq k_i\leq \infty$, i.e. $\delta(\mathcal{G}_s(1))\subseteq \delta(\mc{G}_s(2\wedge 3))$. Strings ending with $0^3$ are obtained following the loop $2\wedge 3 \stackrel{00}{\rightarrow}4\wedge 5\stackrel{2}{\rightarrow}2\wedge 5\stackrel{2}{\rightarrow}2\wedge 3$. Since these are all possible non-trivial paths ending at $2\wedge 3$ we have proven that $\delta(\mc{G}_s(2\wedge 3))=\delta(\mathcal{G}_s(1)\cup\mathcal{G}_s(4))$ which implies $\mathcal{R}(2\wedge 3) = (-\mathcal{R}(1)-\pi_c(\mathbf{e}_1)) \cup (-\mathcal{R}(4)-\pi_c(\mathbf{e}_4))$ by (\ref{eq:num2}).

Since $2\wedge 3$ goes to $3\wedge 4$ by reading a $0$ we deduce immediately that all the strings ending at $3\wedge 4$ are equivalent under $\delta$ to those in $\mathcal{G}_s(2)\cup \mathcal{G}_s(5)$. Hence by (\ref{eq:num2}) we get $\mathcal{R}(3\wedge 4) = (-\mathcal{R}(2)-\pi_c(\mathbf{e}_2)) \cup (-\mathcal{R}(5)-\pi_c(\mathbf{e}_5))$.

Starting from $2\wedge 5$ and going to $2\wedge 4$ passing by $1\wedge 3$ we read $00$ and by the above reasonings we get then all possible strings $\cdots 2^{k_2}0^52^{k_1}00$ belonging to $\mathcal{G}_s(3)$. From $2\wedge 5 \stackrel{00}{\rightarrow} 2\wedge 4\stackrel{0}{\rightarrow}3\wedge 5\stackrel{2}{\rightarrow}2\wedge 4$ we get all expansions in $\mathcal{G}_s(5)+2$, where with the latter we mean the set of $(s_i)_{i\geq 0}\in \mathcal{G}_s(5)$ such that $s_0=2$. Walking $k$ times through the loop $2\wedge 4\rightarrow 3\wedge 5 \rightarrow 2\wedge 4$ and extending to the left with $2\wedge 5$ we get strings $0^2 (02)^k$. Subtracting $v_4 = \delta(200.)$ we get $0^2(02)^{k-2}0002 =_\delta 0^5 2^{2k-3}$. Walking through the loop $2\wedge 4 \rightarrow 3\wedge 5 \rightarrow  2\wedge 5 \rightarrow 2\wedge 4$ and then once into $2\wedge 4\rightarrow 3\wedge 5 \rightarrow 2\wedge 4$ we read the string $020^5 =_\delta 0^42^3$ and, after subtracting $v_4$, we get $0^5 2^2$. Repeating this loop we get arbitrary large strings ending with an even number of $2$s. Thus we have shown we get strings in $\mathcal{G}_s(1)+4$.

So by (\ref{eq:num2}) we just proved that $\mathcal{R}(2\wedge 4)$ is made of the domains $\delta(\mathcal{G}_s(3)) = -\mathcal{R}(3)-\pi_c(\mathbf{e}_3)$, $\delta(\mathcal{G}_s(5)+2) = -\mathcal{R}(5)-\pi_c(\mathbf{e}_5)+\pi_c(\mathbf{e}_2) = -\mathcal{R}(5)-\pi_c(\mathbf{e}_3)$ and $\delta(\mathcal{G}_s(1)+4) = -\mathcal{R}(1)-\pi_c(\mathbf{e}_1) + \pi_c(\mathbf{e}_4) = -\mathcal{R}(1)-\pi_c(\mathbf{e}_3)$, since $\pi_c(\mathbf{e}_1)=\pi_c(\mathbf{e}_3)+\pi_c(\mathbf{e}_4)$ and $\pi_c(\mathbf{e}_5)=\pi_c(\mathbf{e}_2)+\pi_c(\mathbf{e}_3)$.
\end{proof}

\subsection*{Acknowledgements}
The authors are grateful to J\"org Thuswaldner for the precious help and many inspiring discussions which contributed to the creation of this paper, to Val\'erie Berth\'e for reading carefully a preliminary version, and to Pierre Arnoux for giving remarkable comments. We also thank warmly the referees: their suggestions increased  considerably the quality and readability of the paper.

\bibliographystyle{amsalpha}
\bibliography{family}

\end{document}